\documentclass[11.5pt]{article}
\usepackage{amsmath,amssymb,amsthm,graphicx,subfigure,float,caption,epstopdf,tabularx,color, bm,amsfonts,epic}
\usepackage[top=1.25in, bottom=1.0in, left=1.0in, right=1.0in]{geometry}
\usepackage{appendix}
\usepackage{amssymb}
\usepackage{multirow}
\usepackage{mathrsfs}
\usepackage[table]{xcolor}
\usepackage[numbers,sort&compress]{natbib}
\usepackage{comment,enumerate,multicol,xspace}
\usepackage{fancyhdr}
\usepackage{soul}
\usepackage{enumerate}
\usepackage{makecell}
\definecolor{Cobalt}{rgb}{0.13,0.55,0.13}
\usepackage[linkcolor=red,anchorcolor=red,citecolor=red,colorlinks]{hyperref}

\usepackage{bbm}

\usepackage{diagbox}

\usepackage{algorithm}
\usepackage{algorithmic}
\usepackage{booktabs}
\DeclareMathOperator{\sech}{sech}
\DeclareMathOperator{\ess}{ess}
\allowdisplaybreaks
\newtheorem{thm}{Theorem}[section]
\newtheorem{lem}{Lemma}[section]

\newtheorem{den}{Definition}[section]
\newtheorem{ex}{Example}[section]

\numberwithin{equation}{section}

\usepackage{indentfirst}
\graphicspath{{Fig}}

\usepackage{lineno}
\begin{document}


 \title{Unconditional convergence of conservative spectral Galerkin methods for the coupled fractional nonlinear Klein-Gordon-Schr\"odinger equations}

\author{Dongdong Hu$^{1}$, \quad Yayun Fu$^{2}$, \quad Wenjun Cai$^{3}$,\quad Yushun Wang$^{3}$\footnote{{Correspondence author.}
Email addresses:{~hudongdong@jxnu.edu.cn (Dongdong Hu),} {~fyyly@xcu.edu.cn (Yayun Fu),} {~caiwenjun@njnu.edu.cn (Wenjun Cai),} {~wangyushun@njnu.edu.cn (Yushun Wang).}}\\{\small ${}^1$~School of Mathematics and Statistics, Jiangxi Normal University, Nanchang 330022, China}\\{\small ${}^2$~School of Science, Xuchang University, Xuchang 461000, China}\\{\small ${}^3$~School of Mathematical Sciences,  Nanjing Normal University,  Nanjing 210023, China}
}
\date{}
\maketitle
\begin{abstract}
In this work, two novel classes of structure-preserving spectral Galerkin methods are proposed which based on the Crank-Nicolson scheme and the exponential scalar auxiliary variable method respectively, for solving the coupled fractional nonlinear Klein-Gordon-Schr\"odinger equation. The paper focuses on the theoretical analyses and computational efficiency of the proposed schemes, the Crank-Nicoloson scheme is proved to be unconditionally convergent and has the maximum-norm boundness of numerical solutions. The exponential scalar auxiliary variable scheme is linearly implicit and decoupled, but lack of the maximum-norm boundness, also, the energy structure has been modified. Subsequently, the efficient implementations of the proposed schemes are introduced in detail. Both the theoretical analyses and the numerical comparisons show that the proposed spectral Galerkin methods have high efficiency in long-time computations.
\\[2ex]
\end{abstract}


 \vskip 5mm
\section{Introduction}
Fractional differential equations have been widely applied in many fields due to the nonlocal properties of fractional calculus, they can be employed to simulate some complex anomalous~(super)~diffusion phenomena which cannot be simulated by classical differential equations, and gained many researches and considerations by physicists and mathematicians, such that there has created extensive excellent theoretical achievements and realistic value. For instance, Laskin \cite{LaskinPHY2000} proposed a class of space-fractional nonlinear Schr\"odinger equation by adopting $\alpha$-stable L\'evy process instead of Brownian process, which plays an essential role in quantum mechanics, water wave dynamics and optics \cite{water,Opt} \emph{etc.} Almost simultaneously, Alfimov and V\'azquez \emph{et al.} \cite{Alfimov} established a kind of space-fractional nonlinear wave equation, which is utilized to describe the propagation of fluxons in Josephson junctions.  Recently, the study on the well-posedness of solutions of nonlinear partial differential equations are complex and essential. In particular, due to the nonlocal properties of fractional operators, there exist many challenging problems on the study of analytical solutions of fractional differential equations \cite{kdv2020JDE,decasy}. Therefore, it is essential to develop some efficient and stable numerical methods for fractional differential equations \cite{wangpd2,zhaox,zengf,Ainsworth2017SIAM,JEM2020ANM,fuhf2019JCP}. Interestingly, a special dynamical ststem, called as the Klein-Gordon-Schr\"odinger equation which consists of the Schr\"odinger equation and the wave equation, has attracted considerable attentation in recent years \cite{zhang2021JSC,wang1999JMP,lu2001JDE}. In this paper, we focus on effective spectral Galerkin methods for the following strong coupled fractional nonlinear Klein-Gordon-Schr\"odinger (FNKGS) equation:
\begin{align}
&\textrm{i} u_t-\frac{\lambda}{2}(-\Delta)^{\frac{\alpha}{2}}u+\kappa_1 u\phi+2\kappa_2 |u|^2u\phi=0\quad \textrm{in}\quad \Omega\times (0,T],\label{equ1}\\
&\phi_{tt}+\gamma(-\Delta)^{\frac{\alpha}{2}}\phi+\eta^2 \phi-\kappa_1 |u|^2-\kappa_2 |u|^4=0\quad \textrm{in}\quad \Omega\times (0,T],\label{equ2}
\end{align}
subject to the initial-boundary conditions
\begin{align}
&u({ x},0)=u_0({ x}),\quad \phi({ x},0)=\phi_0({ x}),\quad \phi_t(x,0)=\phi_1(x)\quad \textrm{in}\quad \Omega,\label{equ3}\\
& u(x,t)=0,\quad\phi(x,t)=0\quad \textrm{on}\quad \partial\Omega\times (0,T].\label{equ4}
\end{align}
The system \eqref{equ1}--\eqref{equ4} is the fractional model describing the dynamics of conserved complex
nucleon fields $u$ interacting with neutral real scalar meson fields $\phi$ in the bounded domain $\Omega=(a,b)$, where $1<\alpha<2$, $\textrm{i}^2 =-1$, $\lambda,\kappa_1,\kappa_2,\gamma\geq 0$ and $\eta\in\mathbb{R}$ are any given constants, $u_0(x), \phi_1(x)$ and $\phi_1(x)$ are given smooth functions. The Riesz fractional derivative is represented by
\begin{align*}
(-\Delta)^{\frac{\alpha}{2}}u=\frac{1}{2\cos(\frac{\alpha\pi}{2})}\Big( {}_{a}^{RL}\!D^{{\alpha}}_{x}+{}_{x}^{RL}\!D^{{\alpha}}_{b} \Big)u,
\end{align*}
where the left and right Riemann-Liouville fractional derivatives are defined as
\begin{align*}
{}_{a}^{RL}\!D^{{\alpha}}_{x}u=\frac{1}{\Gamma(2-{\alpha})}\frac{\partial^2}{\partial x^2}\int_{a}^{x}\frac{u(\xi,t)d\xi}{(x-\xi)^{{\alpha}-1}},\qquad
{}_{x}^{RL}\!D^{{\alpha}}_{b}u=\frac{1}{\Gamma(2-{\alpha})}\frac{\partial^2}{\partial x^2}\int_{x}^{b}\frac{u(\xi,t)d\xi}{(\xi-x)^{{\alpha}-1}},
\end{align*}
in which $\Gamma(z)=\int^{\infty}_0 x^{z-1}e^{-x}dx$.

It is worth noticing that the solutions of FNKGS equation conserve the mass of the nucleon field as follows:
\begin{align*}
\mathcal{M}(t)=\int_{\Omega} |u|^2dx=\mathcal{M}(0),\quad t\in (0,T],
\end{align*}
and the total energy as follows:
\begin{align*}
\mathcal{E}(t)=
\int_{\Omega}(\phi_t)^2+\gamma{\phi}(-\Delta)^{\frac{\alpha}{2}}\phi+\eta^2 \phi^2+\lambda\bar{u}(-\Delta)^{\frac{\alpha}{2}}u-(2\kappa_1|u|^2+2\kappa_2|u|^4)\phi dx=\mathcal{E}(0),\quad t\in (0,T],
\end{align*}
where $\bar{u}$ is the complex conjugate of $u$.

Our object is to investigate efficient numerical schemes and the associated convergence analyses of the FNKGS equation. The main difficulties of the studies on numerical schemes for the FNKGS equation can be summarized as follows: $(1)$ The proposed numerical algorithms should preserve the inherent physical invariants. $(2)$ The unconditional convergence of numerical methods should be obtained without any Lipschitz continuity restrictions on nonlinear terms of original system. $(3)$ The numerical scheme should be linearly implicit and decoupled, also, the convergence can be obtained. Based on the aforementioned issues, we try to construct a numerical scheme which enjoys high-order accuracy in space and high efficiency in time, for solving the FNKGS equation to overcome some of the above challenges.

As is known to us, structure-preserving algorithm is famous for preserving the inherent invariants of nonlinear conservative systems. In generally, traditional numerical algorithms are proposed to solve the conservative system, which often ignore the intrinsic conservative structure of the system. Although they has the same stable computational ability as the structure-preserving algorithms in a short time, the latter are superior in long-time simulations because they can inherit the intrinsic physical characteristics of a given dynamical system, please refer to \cite{fengk2010book,wang2008local} and references therein. Undoubtedly, time-stepping method is the key of the construction of structure-preserving algorithms. For classical examples, the averaged vector field method \cite{AVF2008JPAMT}, the symplectic Runge-Kutta method \cite{Hairer2006book}, the Hamiltonian boundary value method \cite{HBVMs2016FL} and so on have catched much attention in recent years. However, the above methods share a common drawback that the numerical implementations are usually fully implicit and the convergence analyses are hard to obtain. In the past decades, some novel tools are developed to provide the possibilities for constructing the high-order linearly implicit structure-preserving methods. For instances, readers can refer to the invariant energy quadratization (IEQ) method \cite{IEQ}, the scalar auxiliary variable (SAV) method \cite{SAV}, the exponential scalar auxiliary variable (ESAV) method \cite{ESAV} and so forth, which possess a modified energy function. In recent years, there exist some numerical researches on the FNKGS equation. For examples, Fu and Cai \emph{et al.} \cite{fu2020ANM} proposed a conservative scheme by combining the partitioned averaged vector field method and pseudo spectral method. The numerical scheme is proved to be linearly implicit and decoupled, but the convergence analysis is unacquirable. Li and Huang \emph{et al.} \cite{liNA} derived a class of structure-preserving scheme with the leapfrog method and standard finite element method. Although the convergence is obtained, the energy structure of the proposed method is modified and it suffers a low-order accuracy in space. Simultaneously, Shi and Ma \emph{et al.} \cite{shiy} constructed a kind of numerical scheme by using the leapfrog method and center difference method, which is linearly implicit and unconditionally convergent, but the energy structure is still modified and the long-time conservation of the proposed scheme is less efficient.

The purpose of this paper is to construct a conservative Crank-Nicolson spectral Galerkin method (CN-SGM) which enjoys the original mass and energy structures as well as spectral accuracy in space for the FNKGS equation. Meanwhile, we obtain the unconditional convergence analysis and unique solvability. It is worth to noticing that the existing theoretical result \cite{Boulenger2016JFA} shows that the solution of the FNKGS equation blows up in finite time if the initial energy $\mathcal{E}(0)<0$. For numerical comparisons, we propose another conservative numerical scheme by combining the exponential scalar auxiliary variable method and spectral Galerkin method (ESAV-SGM). Extensive numerical results illustrate two numerical phenomena: (1) The ESAV-SGM is linearly implicit and decoupled, it possesses the higher computational efficiency than CN-SGM. (2) Although the CN-SGM is fully implicit and coupled, it is more accurate and enjoys the superior capabillty of capturing blow-up than CN-SGM when $\mathcal{E}(0)<0$.

The rest of this paper are arranged as follows: In Section \ref{pre}, some essential definitions and lemmas of the Galerkin method are introduced. In Section \ref{cngls}, a class of structure-preserving numerical method is constructed by combining the spectral Galerkin method and the Crank-Nicolson method. Then the invariant conservations and unconditional convergence are proved. In section \ref{esavsgm}, a linearly implicit and decoupled ESAV-SGM is derived for numerical comparison. Subsequently, the detailed numerical implementations of the proposed schemes are offered in Section \ref{imple}. Next, numerical results are reported to illustrate the high efficiency of the proposed schemes in Section \ref{numerexper}. Finally, some conclusions are drawed in Section \ref{conclu}.


\section{Preliminaries}\label{pre}
In this section, we introduce some essential definitions and lemmas of Galerkin methods which can be found in Ref. \cite{roop}.
\begin{den}
Define the inner product, $L^p$ norm $(1\leq p< \infty)$ and $L^\infty$ norm as
\begin{align*}
(u,v):=\int_\Omega \bar{v}udx,\quad \Big\|u\Big\|:=\sqrt{\Big(u,u\Big)},\quad \Big\| u \Big\|_p:=\sqrt[p]{\int_\Omega | u |^pdx},\quad \Big\| u \Big\|_\infty:=\ess \underset{x\in\Omega}{\sup}\Big\{ |u(x)| \Big\}.
\end{align*}
\end{den}

%

\begin{den}[\textbf{Symmetric fractional derivative space} ]\label{symmetric}
For $\mu>0$, where $\mu\neq n-1/2, n\in\mathbb{N}$, define the semi-norm and the norm as
\begin{align*}
\Big|u\Big|_{J_S^\mu(\Omega)}:= \Big|\Big( {}_{a}^{RL}\!D^{\mu}_{x}u,{}_{x}^{RL}\!D^{\mu}_{b}u \Big)\Big|^{1/2},\quad \Big\|u\Big\|_{J_S^\mu(\Omega)}:=\Big(\Big\|u\Big\|^2+\Big|u\Big|^2_{J_S^\mu(\Omega)}\Big)^{1/2}.
\end{align*}
$J_S^\mu(\Omega)$ (or $J_{S,0}^\mu(\Omega)$) denotes the closure of $C^\infty(\Omega)$ (or $C_0^\infty(\Omega)$) with respect to $\Big\|\cdot\Big\|_{J_S^\mu(\Omega)}$.
\end{den}

\begin{den}[\textbf{Fractional Sobolev space} ]
For $\mu>0$, define the semi-norm and the norm as
\begin{align*}
\Big|u\Big|_{H^\mu(\Omega)}:=\Big\| |\xi|^\mu\mathcal{F}[u,\xi] \Big\|,\quad \Big\|u\Big\|_{H^\mu(\Omega)}:=\Big(\Big\|u\Big\|^2+\Big|u\Big|^2_{H^\mu(\Omega)}\Big)^{1/2}.
\end{align*}
$H^\mu(\Omega)$ (or $H_0^\mu(\Omega)$) denotes the closure of $C^\infty(\Omega)$ (or $C_0^\infty(\Omega)$) with respect to $\Big\|\cdot\Big\|_{H^\mu(\Omega)}$, and $\mathcal{F}$ represents the Fourier transform.
\end{den}

\begin{lem}\label{equivalent}
Suppose $\mu>0$ and $\mu\neq n-1/2, n\in\mathbb{N}$. Then the spaces $J_S^\mu(\Omega)$ and $H^\mu(\Omega)$
are equal with equivalent semi-norms and norms, and the spaces $J_{S,0}^\mu(\Omega)$ and $H_0^\mu(\Omega)$ are equal with equivalent semi-norms and norms.
\end{lem}

\begin{lem}
 Suppose $1 < \mu< 2$. For any $u\in H_0^\mu(\Omega)$ and $v\in H_0^{\mu/2}(\Omega)$, we have
 \begin{align*}
 \Big( {}_{a}^{RL}\!D^{\mu}_{x}u,v \Big)= \Big( {}_{a}^{RL}\!D^{\mu/2}_{x}u,{}_{x}^{RL}\!D^{\mu/2}_{b}v \Big),\quad \Big( {}_{x}^{RL}\!D^{\mu}_{b}u,v \Big)= \Big( {}_{x}^{RL}\!D^{\mu/2}_{b}u,{}_{a}^{RL}\!D^{\mu/2}_{x}v \Big).
 \end{align*}
\end{lem}


\section{Crank--Nicolson spectral Galerkin method}\label{cngls}
Define the following bilinear forms
\begin{align*}
\mathcal{B}(u,w)=&\frac{1}{2\cos(\frac{\alpha\pi}{2})}\Big( ( {}_{a}^{RL}\!D^{\alpha/2}_{x}u,{}_{x}^{RL}\!D^{\alpha/2}_{b}w )+( {}_{x}^{RL}\!D^{\alpha/2}_{b}u,{}_{a}^{RL}\!D^{\alpha/2}_{x}w ) \Big),
\quad \forall u,w\in H_0^{\alpha/2}(\Omega).
\end{align*}
For simplicity, we introduce the semi-norms and the norms as
\begin{align*}
\Big|u\Big|_{\alpha/2}:=\sqrt{\mathcal{B}(u,u)},\qquad \Big\|u\Big\|_{\alpha/2}:=\sqrt{ \Big\|u\Big\|^2+\Big|u\Big|^2_{\alpha/2}}.
\end{align*}
We observe from Lemma \ref{equivalent} that $|\cdot|_{\alpha/2}$ and $\|\cdot\|_{\alpha/2}$ are equivalent with the semi-norms and norms of $J_S^\mu(\Omega)$ and $H^\mu(\Omega)$. From Ref. \cite{roop}, the bilinear form $\mathcal{B}(u,w)$ has the continuous and coercive properties, i.e., there exist positive constants
$C_1$ and $C_2$, for any $u,w\in H_0^{{\alpha/2}}(\Omega)$, such that
\begin{align*}
|\mathcal{B}(u,w)|\leq C_1 \Big\|u\Big\|_{\alpha/2}\Big\|w\Big\|_{\alpha/2},\qquad |\mathcal{B}(u,u)|\geq C_2 \Big\|u\Big\|_{\alpha/2}^2.
\end{align*}

Based on Lemma \ref{symmetric} and the definition of Riesz fractional derivative,
 the weak form of \eqref{equ1} reads: finding $u,\phi\in H_0^{{\alpha/2}}(\Omega)$, such that
\begin{align}
&\textrm{i} (u_t,w)-\frac{\lambda}{2}\mathcal{B}(u,w)+\kappa_1 (u\phi,w)+2\kappa_2 (|u|^2u\phi,w)=0,\quad \forall w \in H_0^{{\alpha/2}}(\Omega),\label{weak1}\\
&(\phi_{tt},w)+\gamma\mathcal{B}(\phi,w)+\eta^2 (\phi,w)-\kappa_1 (|u|^2,w)-\kappa_2 (|u|^4,w)=0,\quad \forall w \in H_0^{{\alpha/2}}(\Omega),\label{weak2}
\end{align}
with the initial conditions given by
\begin{align*}
u(x,0)=u_0(x),\quad \phi(x,0)=\phi_0(x),\quad \phi_t(x,0)=\phi_1(x).
\end{align*}

\subsection{Fully discrete scheme}

Introduce $\psi= \phi_t$ to reformulate \eqref{weak1}--\eqref{weak2} into an equivalent system as follows, i.e., finding $u,\phi,\psi\in H_0^{{\alpha/2}}(\Omega)$, for any $w \in H_0^{{\alpha/2}}(\Omega)$, such that
\begin{align}
&\textrm{i} (u_t,w)-\frac{\lambda}{2}\mathcal{B}(u,w)+\kappa_1 (u\phi,w)+2\kappa_2 (|u|^2u\phi,w)=0,\label{weak1-1}\\
&(\phi_t,w)=(\psi,w),\label{weak1-2}\\
&(\psi_{t},w)+\gamma\mathcal{B}(\phi,w)+\eta^2 (\phi,w)-\kappa_1 (|u|^2,w)-\kappa_2 (|u|^4,w)=0.\label{weak1-3}
\end{align}

For simplicity, we introduce the following notations for $n=0,1,\cdots,N_t$,
\begin{align*}
u^n=u(x,t_n),\quad \delta_tu^{n+\frac12}=\frac{u^{n+1}-u^n}{\tau},\quad u^{n+\frac12}=\frac{u^{n+1}+u^n}{2},\quad \widetilde{u}^{n+\frac12}=\frac{3u^n-u^{n-1}}{2},
\end{align*}
where $t_n=n\tau$, and $\tau=T/N_t$ is the time step size. We denote $u^n$ and $U_N^n$ as the exact solution and numerical approximation at $t=t_n$, respectively.
For theoretical analyses, assume that
 \begin{align}\label{exactassume}
 \underset{0\leq t\leq T}{\max}\Big\{ \Big\|u\Big\|,\Big|u\Big|_{\alpha/2},\Big\|u\Big\|_\infty
 ,\Big\|\phi\Big\|,\Big|\phi\Big|_{\alpha/2},\Big\|\phi\Big\|_\infty,\Big\|\psi\Big\| \Big\}\leq \mathcal{C}.
 \end{align}
Without loss of generality, let $\mathcal{C}$ be a general positive constant which is independent of $\tau$ and $N$. For a fixed positive integer $N$, denote $P_N(\Omega)$ to be the polynomial space with the degree no more than $N$ in interval $\Omega$. The
approximation space $X_N^0(\Omega)$ is defined as
\begin{align*}
X_N^0(\Omega)=P_N(\Omega)\cap H_0^{{\alpha/2}}(\Omega).
\end{align*}
It is clear that $X_N^0(\Omega)$ is a subspace of $H_0^{{\alpha/2}}(\Omega)$.
By using the modified Crank--Nicolson scheme for temporal derivative in \eqref{weak1-1}-\eqref{weak1-3},
the Crank--Nicolson scheme with truncation errors reads: finding $u^n,\psi^n,\phi^n,\in H_0^{{\alpha/2}}(\Omega)$, for any $w\in H_0^{{\alpha/2}}(\Omega)$, such that
\begin{align}
&\textrm{i} (\delta_tu^{n+1/2},w)-\frac{\lambda}{2}\mathcal{B}(u^{n+1/2},w)+\kappa_1 (u^{n+1/2}\phi^{n+1/2},w)\label{semidiscrete1}\\
&\qquad\qquad\qquad\qquad\qquad+\kappa_2 \Big((|u^{n+1}|^2+|u^{n}|^2)u^{n+1/2}\phi^{n+1/2},w\Big)=(\widehat{\mathscr{R}}_1^{n},w),\nonumber\\
&(\delta_t \phi^{n+1/2},w)=(\psi^{n+1/2},w)+(\widehat{\mathscr{R}}_2^{n},w),\label{semidiscrete2}\\
&(\delta_t \psi^{n+1/2},w)+\gamma\mathcal{B}(\phi^{n+1/2},w)+\eta^2 (\phi^{n+1/2},w)-\frac{\kappa_1}{2} \Big((|u^{n+1}|^2+|u^{n}|^2),w\Big)\label{semidiscrete3}\\
&\qquad\qquad\qquad\qquad\qquad-\frac{\kappa_2}{2} \Big((|u^{n+1}|^4+|u^{n}|^4),w\Big)=(\widehat{\mathscr{R}}_3^{n},w).\nonumber
\end{align}
Using the Taylor's expansion, suppose $u(\cdot,t),\phi(\cdot,t)\in C^3([0,T])$. One easily arrives at the following truncation errors
\begin{align}
\big|(\widehat{\mathscr{R}}_1^n)\big| \leq \mathcal{C} \tau^2,\qquad\big|(\widehat{\mathscr{R}}_2^n)\big| \leq \mathcal{C} \tau^2,\qquad \big|(\widehat{\mathscr{R}}_3^n)\big| \leq \mathcal{C} \tau^2.
\end{align}
Similar to the proof of \cite{dd2}, assume that $\phi(\cdot,t)\in C^4([0,T])$, one has
\begin{align}\label{derTRU}
\big|\delta_t(\widehat{\mathscr{R}}_2^{n+1/2})\big| \leq \mathcal{C} \tau^2.
\end{align}
Omitting the small terms in \eqref{semidiscrete1}-\eqref{semidiscrete3} and replacing the exact solutions $u^n$, $\psi^n$ and $\phi^n$ with the numerical solutions $U_N^n$, $\Psi_N^n$ and $\Phi_N^n$, we propose the fully discrete CN--SGM scheme as follows: finding $U_N^n,\Psi_N^n,\Phi_N^n\in X_N^0(\Omega)$, for any $w_N\in X_N^0(\Omega)$, such that
\begin{align}
&\textrm{i} (\delta_tU_N^{n+1/2},w_N)-\frac{\lambda}{2}\mathcal{B}(U_N^{n+1/2},w_N)+\kappa_1 (U_N^{n+1/2}\Phi_N^{n+1/2},w_N)\label{CNGLS1-1}\\
&\qquad\qquad\qquad\qquad+\kappa_2 \Big((|U_N^{n+1}|^2+|U_N^{n}|^2)U_N^{n+1/2}\Phi_N^{n+1/2},w_N\Big)=0,\nonumber\\
&(\delta_t \Phi_N^{n+1/2},w_N)=(\Psi_N^{n+1/2},w_N),\label{CNGLS1-2}\\
&(\delta_t \Psi_N^{n+1/2},w_N)+\gamma\mathcal{B}(\Phi_N^{n+1/2},w_N)+\eta^2 (\Phi_N^{n+1/2},w_N)-\frac{\kappa_1}{2} \Big((|U_N^{n+1}|^2+|U_N^{n}|^2),w_N\Big)\label{CNGLS1-3}\\
&\qquad\qquad\qquad\qquad-\frac{\kappa_2}{2} \Big((|U_N^{n+1}|^4+|U_N^{n}|^4),w_N\Big)=0,\nonumber
\end{align}
with the initial conditions
\begin{align}\label{initialvalue}
U^0_N=\Pi_N^{{\alpha/2},0}u_0(x),\quad \Phi^0_N=\Pi_N^{{\alpha/2},0}\phi_0(x),\quad \Psi^0_N=\Pi_N^{{\alpha/2},0}\phi_1(x),
\end{align}
where the orthogonal projection operator $\Pi_N^{{\alpha/2},0}$: $H_0^{{\alpha/2}}(\Omega)\rightarrow X_N^0(\Omega)$ is defined as
\begin{align*}
\mathcal{B}( u-\Pi_N^{\alpha/2,0}u,w_N )=0,\quad \forall w_N\in X_N^0(\Omega).
\end{align*}

\subsection{Conservations and boundness}

\begin{thm}[\textbf{Mass and energy conservative laws}]\label{EP1-1}
The spectral Galerkin scheme \eqref{CNGLS1-1}--\eqref{CNGLS1-3} possesses the mass and energy conservative laws in the following discrete sense that
\begin{align}\label{massenergy}
\mathcal{M}^n=\cdots=\mathcal{M}^0,\qquad \mathcal{E}^n=\cdots=\mathcal{E}^0,\quad n=1,2,\cdots,N_t,
\end{align}
in which
\begin{align*}
\mathcal{M}^n=\Big\|U_N^n \Big\|^2
\end{align*}
and
\begin{align*}
\mathcal{E}^n=\Big\|\Psi_N^n\Big\|^2+\gamma\Big| \Phi^n_N \Big|^2_{\alpha/2}+\eta^2 \Big\| \Phi^n_N \Big\|^2
+\lambda\Big| U^n_N \Big|^2_{\alpha/2}
-2\int_\Omega \Big(\kappa_1 \Big| U_N^n \Big|^2+\kappa_2 \Big| U_N^n \Big|^4\Big)\Phi_N^ndx.
\end{align*}
\begin{proof}
Choosing $w_N=U_N^{n+1/2}$ in \eqref{CNGLS1-1} and taking the imaginary part, we obtain $\Big\|U_N^{n+1} \Big\|^2=\Big\|U_N^n \Big\|^2$ which implies the mass conservation.

Selecting $w_N=\delta_t U_N^{n+1/2}$ in \eqref{CNGLS1-1} and taking the real part, we have
\begin{align}
& \textrm{Re}(\textrm{i}\delta_tU_N^{n+1/2},\delta_t U_N^{n+1/2})-\frac{\lambda}{2}\textrm{Re}\Big\{\mathcal{B}(U_N^{n+1/2},\delta_t U_N^{n+1/2})\Big\}
+\kappa_1 \textrm{Re}(U_N^{n+1/2}\Phi_N^{n+1/2},\delta_t U_N^{n+1/2})\nonumber\\
&\qquad\qquad\qquad\qquad+\kappa_2 \textrm{Re}\Big((|U_N^{n+1}|^2+|U_N^{n}|^2)U_N^{n+1/2}\Phi_N^{n+1/2},\delta_t U_N^{n+1/2}\Big)=0,\label{proofenergy1}
\end{align}
where ``$\textrm{Re}$'' represents the real part of a complex number.

By some calculations, we have
\begin{align}\label{proofenergy1-1}
\textrm{Re}(\textrm{i}\delta_tU_N^{n+1/2},\delta_t U_N^{n+1/2})=0,\qquad\textrm{Re}\Big\{\mathcal{B}(U_N^{n+1/2},\delta_t U_N^{n+1/2})\Big\}=\frac{\Big| U^{n+1}_N \Big|^2_{\alpha/2}-\Big| U^n_N \Big|^2_{\alpha/2}}{2\tau},
\end{align}
\begin{align}\label{proofenergy1-2}
\textrm{Re}(U_N^{n+1/2}\Phi_N^{n+1/2},\delta_t U_N^{n+1/2})&=\frac{1}{2\tau}\int_\Omega \Big( \Big|U_N^{n+1}\Big|^2-\Big| U_N^{n} \Big|^2 \Big)\Phi_N^{n+1/2}dx\nonumber\\
&=\frac{1}{4\tau}\int_\Omega \Big( \Big|U_N^{n+1}\Big|^2-\Big| U_N^{n} \Big|^2 \Big)\Big( \Phi_N^{n+1}+\Phi_N^{n}\Big)dx
\end{align}
and
\begin{align}\label{proofenergy1-3}
\textrm{Re}\Big((|U_N^{n+1}|^2+|U_N^{n}|^2)U_N^{n+1/2}\Phi_N^{n+1/2},\delta_t U_N^{n+1/2}\Big)&=\frac{1}{2\tau}\int_\Omega \Big( \Big|U_N^{n+1}\Big|^4-\Big| U_N^{n} \Big|^4 \Big)\Phi_N^{n+1/2}dx\nonumber\\
&=\frac{1}{4\tau}\int_\Omega \Big( \Big|U_N^{n+1}\Big|^4-\Big| U_N^{n} \Big|^4 \Big)\Big( \Phi_N^{n+1}+\Phi_N^{n}\Big)dx.
\end{align}
Substituting \eqref{proofenergy1-1}--\eqref{proofenergy1-3} into \eqref{proofenergy1}, we conclude
\begin{align}\label{proofenergy1-0}
\lambda\Big( \Big| U^{n+1}_N \Big|^2_{\alpha/2}-\Big| U^n_N \Big|^2_{\alpha/2} \Big)&-2\kappa_1\int_\Omega
\Big( \Big|U_N^{n+1}\Big|^2 \Phi_N^{n+1}-\Big| U_N^{n} \Big|^2\Phi_N^{n}+\Big| U_N^{n+1} \Big|^2\Phi_N^{n}-\Big| U_N^{n} \Big|^2\Phi_N^{n+1} \Big)dx\nonumber\\
&-2\kappa_2\int_\Omega
\Big( \Big|U_N^{n+1}\Big|^4 \Phi_N^{n+1}-\Big| U_N^{n} \Big|^4\Phi_N^{n}+\Big| U_N^{n+1} \Big|^4\Phi_N^{n}-\Big| U_N^{n} \Big|^4\Phi_N^{n+1} \Big)dx=0.
\end{align}
Setting $w_N=\delta_t\Phi_N^{n+1/2}$ in \eqref{CNGLS1-3}, we get
\begin{align}\label{proofenergy2}
&(\delta_t \Psi_N^{n+1/2},\delta_t\Phi_N^{n+1/2})+\gamma\mathcal{B}(\Phi_N^{n+1/2},\delta_t\Phi_N^{n+1/2})+\eta^2 (\Phi_N^{n+1/2},\delta_t\Phi_N^{n+1/2})\nonumber\\
&\qquad\qquad\qquad-\frac{\kappa_1}{2} \Big((|U_N^{n+1}|^2+|U_N^{n}|^2),\delta_t\Phi_N^{n+1/2}\Big)-\frac{\kappa_2}{2} \Big((|U_N^{n+1}|^4+|U_N^{n}|^4),\delta_t\Phi_N^{n+1/2}\Big)=0,
\end{align}
in which
\begin{align}\label{proofenergy2-1}
(\delta_t \Psi_N^{n+1/2},\delta_t\Phi_N^{n+1/2})=(\Psi_N^{n+1/2},\delta_t \Psi_N^{n+1/2})=\frac{\Big\| \Psi_N^{n+1} \Big\|^2-\Big\| \Psi_N^{n} \Big\|^2}{2\tau},
\end{align}
\begin{align}\label{proofenergy2-2}
\mathcal{B}(\Phi_N^{n+1/2},\delta_t\Phi_N^{n+1/2})=\frac{\Big| \Phi^{n+1}_N \Big|^2_{\alpha/2}-\Big| \Phi^n_N, \Big|^2_{\alpha/2}}{2\tau},
\end{align}
\begin{align}\label{proofenergy2-3}
(\Phi_N^{n+1/2},\delta_t\Phi_N^{n+1/2})=\frac{\Big\| \Phi_N^{n+1} \Big\|^2-\Big\| \Phi_N^{n} \Big\|^2}{2\tau},
\end{align}
\begin{align}\label{proofenergy2-4}
\Big((|U_N^{n+1}|^2+|U_N^{n}|^2),\delta_t\Phi_N^{n+1/2}\Big)=\frac{1}{\tau}\int_\Omega \Big(\Big|U_N^{n+1}\Big|^2+\Big|U_N^{n}\Big|^2\Big)
\Big( \Phi_N^{n+1}-\Phi_N^{n}\Big)dx
\end{align}
and
\begin{align}\label{proofenergy2-5}
\Big((|U_N^{n+1}|^4+|U_N^{n}|^4),\delta_t\Phi_N^{n+1/2}\Big)=\frac{1}{\tau}\int_\Omega \Big(\Big|U_N^{n+1}\Big|^4+\Big|U_N^{n}\Big|^4\Big)
\Big( \Phi_N^{n+1}-\Phi_N^{n}\Big)dx.
\end{align}
Substituting \eqref{proofenergy2-1}--\eqref{proofenergy2-4} into \eqref{proofenergy2}, we obtain
\begin{align}\label{proofenergy2-0}
\Big( \Big\| \Psi_N^{n+1} \Big\|^2&-\Big\| \Psi_N^{n} \Big\|^2 \Big)+\gamma\Big( \Big\| \Phi_N^{n+1} \Big\|^2-\Big\| \Phi_N^{n} \Big\|^2 \Big)+\eta^2\Big( \Big| \Phi^{n+1}_N \Big|^2_{\alpha/2}-\Big| \Phi^n_N, \Big|^2_{\alpha/2} \Big)\nonumber\\
&-2\kappa_1\int_\Omega \Big(\Big|U_N^{n+1}\Big|^2\Phi_N^{n+1}-\Big|U_N^{n}\Big|^2\Phi_N^{n}+\Big|U_N^{n}\Big|^2\Phi_N^{n+1}-\Big|U_N^{n+1}\Big|^2\Phi_N^{n}\Big)dx\\
&-2\kappa_2\int_\Omega \Big(\Big|U_N^{n+1}\Big|^4\Phi_N^{n+1}-\Big|U_N^{n}\Big|^4\Phi_N^{n}+\Big|U_N^{n}\Big|^4\Phi_N^{n+1}-\Big|U_N^{n+1}\Big|^4\Phi_N^{n}\Big)dx=0.\nonumber
\end{align}
Summing up \eqref{proofenergy1-0} and \eqref{proofenergy2-0}, the energy conservation \eqref{massenergy} is immediate. This ends the proof.
\end{proof}
\end{thm}

\begin{lem}[\cite{zengf,huangSCM2014}]\label{orgalpha}
Let $\mu$ and $r$ be arbitrary real numbers satisfying $1<\alpha<2$, $r\geq 1$. Then there exists a positive
constant $\mathcal{C}$ independent of $N$, such that, for any function $u\in H_0^{\frac{\alpha}{2}}(\Omega)\cap H^r(\Omega)$, the following estimates hold:
\begin{equation*}
\left\{\begin{array}{ll}
&\Big\|u-\Pi_N^{\frac{\alpha}{2},0}u\Big\|\leq \mathcal{C} N^{-r}\Big\|u\Big\|_{H^r(\Omega)},\quad \alpha\neq \frac32\\
&\Big\|u-\Pi_N^{\frac{\alpha}{2},0}u\Big\|\leq \mathcal{C} N^{\epsilon-r}\Big\|u\Big\|_{H^r(\Omega)},\quad \alpha=\frac32,\quad 0<\epsilon<\frac12
\end{array}
\right.
\end{equation*}
and
\begin{align*}
\Big|u-\Pi_N^{\frac{\alpha}{2},0}u\Big\|_{\frac{\alpha}{2}}\leq \mathcal{C} N^{\frac{\alpha}{2}-r}\Big\|u\Big\|_{H^r(\Omega)}.
\end{align*}
\end{lem}

\begin{lem}[\textbf{Fractional Gagliardo-Nirenberg inequality} \cite{fgn,meilq}]\label{fng}
For $1<\alpha<2$, there exist positive constants $\mathcal{C}_4$ and $\mathcal{C}_8$, such that
\begin{align*}
\Big\|u\Big\|^4_4\leq \mathcal{C}_4\Big|u\Big|_{\alpha/2}^{2/\alpha}\Big\|u\Big\|^{4-2/\alpha},\qquad \Big\|u\Big\|^8_8\leq \mathcal{C}_8 \Big|u\Big|_{\alpha/2}^{2/\alpha}\Big\|u\Big\|^{8-2/\alpha}.
\end{align*}
\end{lem}

\begin{lem}[\textbf{Sobolev inequality} \cite{Kirkpatrick}]\label{boundinf}
For $\frac12 < \mu< 1$, then there exists a positive constant $\mathcal{C}$, such that
\begin{align*}
\Big\|u\Big\|_\infty\leq \mathcal{C} \Big\|u\Big\|_{\mu},\quad \forall u\in H^\mu_0(\Omega).
\end{align*}
\end{lem}

In theoretical analyses, we will frequently use the following Young's inequality
\begin{align*}
ab\leq \varepsilon a^2+\frac{1}{4\varepsilon}b^2\quad \textrm{for} \quad a,b\geq 0,\quad \varepsilon>0.
\end{align*}

\begin{thm}[\textbf{Boundness of numerical solution}]\label{BOUND}
The numerical solution of spectral Galerkin scheme \eqref{CNGLS1-1}--\eqref{CNGLS1-3} is long-time bounded in the following sense
\begin{align*}
\Big\| U_N^n \Big\|\leq \mathcal{C},\quad \Big| U^n_N \Big|_{\alpha/2}\leq \mathcal{C},\quad \Big\| U^{n}_N \Big\|_\infty\leq \mathcal{C},\qquad n=0,1,\cdots,N_t
\end{align*}
and
\begin{align*}
\Big\| \Psi^n_N \Big\|^2 \leq \mathcal{C},\quad\Big\| \Phi_N^n \Big\|\leq \mathcal{C},\quad \Big| \Phi^n_N \Big|_{\alpha/2}\leq \mathcal{C},\quad \Big\| \Phi^{n}_N \Big\|_\infty\leq \mathcal{C},\qquad n=0,1,\cdots,N_t.
\end{align*}
\begin{proof}
From the triangle inequality, the assumptions on exact solution \eqref{exactassume} and Lemma \ref{orgalpha}, for sufficient small $N^{-1}$, we derive
\begin{align*}
\Big\|U_N^{0}\Big\|=\Big\|\Pi_N^{{\alpha/2},0}u_0\Big\| \leq \Big\|u_0\Big\|+\Big\|u_0-\Pi_N^{{\alpha/2},0}u_0\Big\|\leq \mathcal{C}.
\end{align*}
Thus, we have $\mathcal{M}^0\leq \mathcal{C}$. By similar procedure, we deduce $\mathcal{E}^0\leq \mathcal{C}$.

According to the mass conservation \eqref{massenergy}, we have
\begin{align*}
\Big\| U_N^n \Big\|^2= \mathcal{M}^n= \mathcal{M}^0=\Big\|U_N^{0}\Big\|^2\leq \mathcal{C},\qquad n=0,1,\cdots,N_t.
\end{align*}
 It follows from energy conservation \eqref{massenergy} and the Young's inequality that
 \begin{align}\label{boundine1}
\Big\|\Psi_N^n\Big\|^2&+\gamma\Big| \Phi^n_N \Big|^2_{\alpha/2}+\eta^2 \Big\| \Phi^n_N \Big\|^2
+\lambda\Big| U^n_N \Big|^2_{\alpha/2}=
\mathcal{E}^0+2\int_\Omega \Big(\kappa_1 \Big| U_N^n \Big|^2+\kappa_2 \Big| U_N^n \Big|^4\Big)\Phi_N^ndx\nonumber\\
&\leq \mathcal{E}^0+\frac{\eta^2}{2}\Big\| \Phi_N^n \Big\|^2+\frac{2}{\eta^2}\int_\Omega \Big(\kappa_1 \Big| U_N^n \Big|^2+\kappa_2 \Big| U_N^n \Big|^4\Big)^2dx\\
&\leq \mathcal{C}+\frac{\eta^2}{2}\Big\| \Phi_N^n \Big\|^2+\frac{4\kappa_1^2}{\eta^2} \Big\| U_N^n \Big\|_4^4+\frac{4\kappa_2^2}{\eta^2} \Big\| U_N^n \Big\|^8_8.\nonumber
\end{align}
By using Lemma \ref{fng} and the Holder inequality, we derive
\begin{align}\label{boundine1-1}
\Big\| U_N^n \Big\|_4^4\leq\mathcal{C}_4\Big|U_N^n\Big|_{\alpha/2}^{2/\alpha}\Big\|U_N^n\Big\|^{4-2/\alpha}\leq
\mathcal{C}_4\epsilon_1\Big|U_N^n\Big|_{\alpha/2}^2+\mathcal{C}_4\mathcal{C}(\epsilon_1)\Big\|U_N^n\Big\|^{\frac{4\alpha-2}{\alpha-1}}
\leq \mathcal{C}_4\epsilon_1\Big|U_N^n\Big|_{\alpha/2}^2+\mathcal{C},
\end{align}
where $\epsilon_1$ is a positive constant, and $\mathcal{C}(\epsilon_1)>0$ is a constant dependent of the positive constant $\epsilon_1$. Similarly, we obtain
\begin{align}\label{boundine1-2}
\Big\| U_N^n \Big\|_8^8\leq \mathcal{C}_8\epsilon_2\Big|U_N^n\Big|_{\alpha/2}^2+\mathcal{C}.
\end{align}
Substituting \eqref{boundine1-1} and \eqref{boundine1-2} into \eqref{boundine1}, we conclude
 \begin{align}\label{boundine1-3}
\Big\|\Psi_N^n\Big\|^2&+\gamma\Big| \Phi^n_N \Big|^2_{\alpha/2}+\frac{\eta^2}{2} \Big\| \Phi^n_N \Big\|^2
+\lambda\Big| U^n_N \Big|^2_{\alpha/2}\leq \mathcal{C}+\Big( \frac{4\kappa_1^2\mathcal{C}_4\epsilon_1}{\eta^2}+\frac{4\kappa_2^2\mathcal{C}_8\epsilon_2}{\eta^2}\Big)\Big|U_N^n\Big|_{\alpha/2}^2.
\end{align}
Taking appropriate $\epsilon_1$ and $\epsilon_2$, such that
\begin{align}\label{boundine1-4}
\frac{4\kappa_1^2\mathcal{C}_4\epsilon_1}{\eta^2}+\frac{4\kappa_2^2\mathcal{C}_8\epsilon_2}{\eta^2}=\frac{\lambda}{2}.
\end{align}
Substituting \eqref{boundine1-4} into \eqref{boundine1-3}, we derive
\begin{align*}
\Big\|\Psi_N^n\Big\|^2&+\gamma\Big| \Phi^n_N \Big|^2_{\alpha/2}+\frac{\eta^2}{2} \Big\| \Phi^n_N \Big\|^2
+\frac{\lambda}{2}\Big| U^n_N \Big|^2_{\alpha/2}\leq \mathcal{C}.
\end{align*}
Therefore, we have
\begin{align*}
\Big\| \Psi^n_N \Big\|^2 \leq \mathcal{C},\quad\Big| \Phi^n_N \Big|^2_{\alpha/2}\leq \mathcal{C},\quad \Big\| \Phi^n_N \Big\|^2 \leq \mathcal{C},\quad \Big| U^n_N \Big|^2_{\alpha/2}\leq \mathcal{C},\qquad n=0,1,\cdots,N_t.
\end{align*}
By admitting Lemma \ref{boundinf}, we obtain
\begin{align*}
\Big\| U^{n}_N \Big\|_\infty\leq \mathcal{C}\sqrt{\Big| U^n_N \Big|^2_{\alpha/2}+\Big\| U^n_N \Big\|^2} \leq \mathcal{C},\quad \Big\| \Phi^{n}_N \Big\|_\infty\leq \mathcal{C}\sqrt{\Big| \Phi^n_N \Big|^2_{\alpha/2}+\Big\| \Phi^n_N \Big\|^2} \leq \mathcal{C}.
\end{align*}
This completes the proof.
\end{proof}
\end{thm}

\subsection{Unique solvability}
\begin{lem}\label{nonlinear}
Suppose that $X_i$, $Y_i$ (i=1,2,3,4) are complex functions. Then the following inequalities hold
\begin{align}
&(1)~~\Big| X_1X_2-Y_1Y_2 \Big|\leq \max\Big\{ |X_1|,|Y_2| \Big\}\Big( \Big| X_1-Y_1 \Big|+\Big| X_2-Y_2 \Big| \Big),\label{nonlinear1-1}\\
&(2)~~\Big| |X_1|^2-|Y_1|^2 \Big|\leq 2\max\Big\{ |X_1|,|Y_1| \Big\} \Big| X_1-Y_1 \Big|,\label{nonlinear1-2}\\
&(3)~~\Big| |X_1|^4-|Y_1|^4 \Big|\leq 2(|X_1|^2+|Y_1|^2)\max\Big\{ |X_1|,|Y_1| \Big\} \Big| X_1-Y_1 \Big|,\label{nonlinear1-3}
\\
&(4)~~\Big| (|X_3|^2+|X_4|^2)X_1X_2-(|Y_3|^2+|Y_4|^2)Y_1Y_2 \Big|\nonumber\\
 &\qquad\qquad\leq 2\max\Big\{ |X_3|^2,|Y_1||Y_2| \Big\}\max\Big\{ |X_3|,|Y_3| \Big\}\Big|X_3-Y_3\Big|\nonumber\\
 &\qquad\qquad\qquad+2\max\Big\{ |X_4|^2,|Y_1||Y_2| \Big\}\max\Big\{ |X_4|,|Y_4| \Big\}\Big|X_4-Y_4\Big|\\
 &\qquad\qquad\qquad+\max\Big\{ |X_3|^2,|Y_1||Y_2| \Big\}\max\Big\{ |X_1|,|Y_2| \Big\}\Big(\Big|X_1-Y_1\Big|+\Big|X_2-Y_2\Big|\Big)\label{nonlinear1-4}\nonumber\\
 &\qquad\qquad\qquad+\max\Big\{ |X_4|^2,|Y_1||Y_2| \Big\}\max\Big\{ |X_1|,|Y_2| \Big\}\Big(\Big|X_1-Y_1\Big|+\Big|X_2-Y_2\Big|\Big).\nonumber
\end{align}
\end{lem}
\begin{proof}
From the triangle inequality and some careful calculations, one derives
\begin{align*}
\Big| X_1X_2-Y_1Y_2 \Big|&=\Big|X_1X_2+(X_1-Y_1)Y_2-X_1Y_2\Big|=\Big| (X_1-Y_1)Y_2 + X_1(X_2-Y_2) \Big| \nonumber\\
& \leq \Big|X_1-Y_1\Big| \Big|Y_2\Big| + \Big|X_1\Big|\Big|X_2-Y_2\Big|\\
& \leq \max\Big\{ |X_1|,|Y_2| \Big\} \Big( \Big| X_1-Y_1 \Big|+\Big| X_2-Y_2 \Big| \Big).
\end{align*}
Following the inequality \eqref{nonlinear1-1}, one arrives at
\begin{align*}
\Big| |X_1|^2-|Y_1|^2 \Big|&=\Big|X_1\overline{X_1}-Y_1\overline{Y_1}\Big| \\
&\leq \max\Big\{ |X_1|,|\overline{Y_1}| \Big\} \Big( \Big| X_1-Y_1 \Big|+\Big| \overline{X_1}-\overline{Y_1} \Big| \Big).\nonumber\\
&= 2\max\Big\{ |X_1|,|Y_1| \Big\} \Big| X_1-Y_1 \Big|.
\end{align*}
Similarly, we conclude
\begin{align*}
\Big| |X_1|^4-|Y_1|^4 \Big|&=\Big| (|X_1|^2+|Y_1|^2)(|X_1|^2-|Y_1|^2) \Big|\nonumber\\
&\leq (|X_1|^2+|Y_1|^2)  \Big||X_1|^2-|Y_1|^2 \Big|\nonumber\\
&\leq 2(|X_1|^2+|Y_1|^2)\max\Big\{ |X_1|,|Y_1| \Big\} \Big| X_1-Y_1 \Big|.
\end{align*}
It follows from the relation \eqref{nonlinear1-1} and \eqref{nonlinear1-2} that
\begin{align*}
 &\Big| (|X_3|^2+|X_4|^2)X_1X_2-(|Y_3|^2+|Y_4|^2)Y_1Y_2 \Big|\nonumber\\
 &=\Big| |X_3|^2X_1X_2-|Y_3|^2Y_1Y_2 \Big|+ \Big| |X_4|^2X_1X_2-|Y_4|^2Y_1Y_2 \Big|\\
 &\leq\max\Big\{ |X_3|^2,|Y_1||Y_2| \Big\} \Big( \Big| |X_3|^2-|Y_3|^2 \Big|+\Big| X_1X_2-Y_1Y_2 \Big| \Big)\\
 &\qquad +\max\Big\{ |X_4|^2,|Y_1||Y_2| \Big\} \Big( \Big| |X_4|^2-|Y_4|^2 \Big|+\Big| X_1X_2-Y_1Y_2 \Big| \Big)\\
 &\leq 2\max\Big\{ |X_3|^2,|Y_1||Y_2| \Big\}\max\Big\{ |X_3|,|Y_3| \Big\}\Big|X_3-Y_3\Big|\\
 &\qquad+2\max\Big\{ |X_4|^2,|Y_1||Y_2| \Big\}\max\Big\{ |X_4|,|Y_4| \Big\}\Big|X_4-Y_4\Big|\\
 &\qquad+\max\Big\{ |X_3|^2,|Y_1||Y_2| \Big\}\max\Big\{ |X_1|,|Y_2| \Big\}\Big(\Big|X_1-Y_1\Big|+\Big|X_2-Y_2\Big|\Big)\\
 &\qquad+\max\Big\{ |X_4|^2,|Y_1||Y_2| \Big\}\max\Big\{ |X_1|,|Y_2| \Big\}\Big(\Big|X_1-Y_1\Big|+\Big|X_2-Y_2\Big|\Big).
\end{align*}
This completes the proof.
\end{proof}

 \begin{lem} [\textbf{Browder fixed point theorem} \cite{Browder}]\label{Browder}
 Let $(\mathcal{H},(\cdot,\cdot))$ be a finite dimensional inner product space, $||\cdot||$ the associated norm, and
 $\mathscr{F}: \mathcal{H}\rightarrow \mathcal{H}$ be continuous, such that
 \begin{align*}
 \exists\  \delta>0,\  \forall\  z\in \mathcal{H},\ \|z\|=\delta,\ s.t.\ \text{Re}( \mathscr{F}(z),z) \geq 0,
 \end{align*}
 there exists $z^{*}\in \mathcal{H}, \|z^{*}\|\leq \delta$ such that $\mathscr{F}(z^{*})=0$.
\end{lem}

\begin{thm}\label{unique}
For the given initial values and sufficient small $\tau$, the numerical solutions of spectral Galerkin scheme \eqref{CNGLS1-1}--\eqref{CNGLS1-3} are uniquely solvable.
\end{thm}
\begin{proof}
 For the sake of readability, we leave the proof of this theorem to Appendix \ref{app}.
\end{proof}

\subsection{Convergence analysis}
\begin{lem}[\cite{dd1}]\label{ineq1}
For time sequences $w=\{ w^0,w^1,\cdots,w^n,w^{n+1} \}$ and $g=\{ g^0,g^1,\cdots,g^n,g^{n+1} \}$, there is
\begin{equation*}
\mid 2\tau\sum\limits^n_{k=0}g^k\delta_tw^{k+1/2} \mid\leq \tau\sum\limits^n_{k=1}\mid w^k \mid^2+\tau\sum\limits^{n-1}_{k=0}\mid \delta_t g^{k+1/2} \mid^2+\frac12\mid w^{n+1} \mid^2+2\mid g^n \mid^2+\mid w^0 \mid^2+\mid g^0 \mid^2.
\end{equation*}
\end{lem}

\begin{lem}[\textbf{Gronwall inequality I} \cite{ineq2}]\label{ineq2}
Suppose that the discrete grid function $\{ w^n\mid~ n=0,1,2,\cdots,N_t;\\ N_t\tau=T \}$ satisfies the following inequality
\begin{equation*}
w^n-w^{n-1}\leq A\tau w^n+B\tau w^{n-1}+C_n\tau,
\end{equation*}
where $A$, $B$ and $C_n$ are non-negative constants, then
\begin{equation*}
\underset{1\leq n\leq N_t}{\max}\mid w^n\mid\leq\Big( w^0+\tau\sum\limits^{N_t}_{k=1}C_k \Big)e^{2(A+B)T},
\end{equation*}
where $\tau$ is sufficiently small, such that $(A+B)\tau<\frac12$, $(N_t>1)$.
\end{lem}

\begin{lem}[\textbf{Gronwall inequality II} \cite{ineq2}]\label{ineq3}
Suppose that the discrete grid function $\{ w^n\mid~ n=0,1,2,\cdots,N_t;\\ N_t\tau=T \}$ satisfies the following inequality
\begin{equation*}
w^n\leq A+\tau\sum\limits^{n}_{k=1}B_kw^k,
\end{equation*}
where $A$ and $B_k~(k=0,1,2,\cdots,N_t)$ are non-negative constants, then
\begin{equation*}
\underset{1\leq n\leq N_t}{\max}\mid w^n\mid\leq A\exp({2\tau\sum\limits^{N_t}_{k=1}B_k}),
\end{equation*}
where $\tau$ is sufficiently small, such that $\tau\underset{1\leq k\leq N_t}{\max}B_k \leq1/2$.
\end{lem}

\begin{thm}\label{thmend}
Suppose $u\in C^3( H_0^{{\frac{\alpha}{2}}(\Omega)}(\Omega)\cap H^r(\Omega),[0,T] )$ and $\phi\in C^4( H_0^{{\frac{\alpha}{2}}(\Omega)}(\Omega)\cap H^r(\Omega),[0,T] )$ $(r> 1)$ are the solutions of \eqref{equ1}--\eqref{equ2}, $U_N^n$ and $\Phi_N^n$ are the solutions of \eqref{CNGLS1-1}--\eqref{CNGLS1-3}. For sufficiently small $\tau$ and $N^{-1}$, the CN--SGM scheme is unconditionally convergent in the sense that
\begin{equation*}
\left\{\begin{array}{ll}
\Big\| u^n-U^n_N \Big\| \leq \mathcal{C}(\tau^2+N^{-r}),\quad \Big\| \phi^n-\Phi^n_N \Big\|\leq \mathcal{C}(\tau^2+N^{-r}),
\quad \alpha\neq \frac32,\\
\Big\| u^n-U^n_N \Big\| \leq \mathcal{C}(\tau^2+N^{\epsilon-r}),\quad
\Big\| \phi^n-\Phi^n_N \Big\|\leq \mathcal{C}(\tau^2+N^{\epsilon-r}),
\quad \alpha= \frac32,\quad 0<\epsilon<\frac12
\end{array}
\right.
\end{equation*}
and
\begin{align*}
\Big\| \phi^n-\Phi^n_N \Big\|_\infty \leq \mathcal{C}(\tau^2+N^{\frac{\alpha}{2}-r}).
\end{align*}
\end{thm}
\begin{proof}
Denote the error functions
\begin{align*}
&\varepsilon_u^n=u^n-U_N^n=( u^n-\Pi_N^{\alpha/2,0}u^n )+(\Pi_N^{\alpha/2,0}u^n-U_N^n)=\eta_u^n+\xi_u^n,\\
&\varepsilon_\phi^n=\phi^n-\Phi_N^n=( \phi^n-\Pi_N^{\alpha/2,0}\phi^n )+(\Pi_N^{\alpha/2,0}\phi^n-\Phi_N^n)=\eta_\phi^n+\xi_\phi^n,\\
&\varepsilon_\psi^n=\psi^n-\Psi_N^n=( \psi^n-\Pi_N^{\alpha/2,0}\psi^n )+(\Pi_N^{\alpha/2,0}\psi^n-\Phi_N^n)=\eta_\psi^n+\xi_\psi^n,\\
&\varepsilon_u^0=u_0-\Pi_N^{\alpha/2,0}u_0,\qquad \varepsilon_\phi^0=\phi_0-\Pi_N^{\alpha/2,0}\phi_0,
\qquad \varepsilon_\psi^0=\phi_1-\Pi_N^{\alpha/2,0}\phi_1.
\end{align*}
Noticing that
\begin{align*}
\xi_u^{0}=\Pi_N^{\alpha/2,0}u_0-U_N^0=\Pi_N^{\alpha/2,0}u_0-\Pi_N^{\alpha/2,0}u_0=0,\\
\xi_\phi^{0}=\Pi_N^{\alpha/2,0}\phi_0-\Phi_N^0=\Pi_N^{\alpha/2,0}\phi_0-\Pi_N^{\alpha/2,0}\phi_0=0,\\
\xi_\psi^{0}=\Pi_N^{\alpha/2,0}\phi_1-\Psi_N^0=\Pi_N^{\alpha/2,0}\phi_1-\Pi_N^{\alpha/2,0}\phi_1=0.
\end{align*}
Subtracting \eqref{CNGLS1-1}--\eqref{CNGLS1-3} from \eqref{semidiscrete1}--\eqref{semidiscrete3}, in view of the definition of orthogonal projection operator $\Pi_N^{{\alpha/2},0}$, we have
\begin{align}
&\textrm{i} (\delta_t\xi_u^{n+1/2},w_N)-\frac{\lambda}{2}\mathcal{B}(\xi_u^{n+1/2},w_N)
+\kappa_1 (u^{n+1/2}\phi^{n+1/2}-U_N^{n+1/2}\Phi_N^{n+1/2},w_N)\label{errorF1-1}\nonumber\\
&\quad+\kappa_2
\Big((|u^{n+1}|^2+|u^{n}|^2)u^{n+1/2}\phi^{n+1/2}-(|U_N^{n+1}|^2+|U_N^{n}|^2)U_N^{n+1/2}\Phi_N^{n+1/2},w_N\Big)\\
&\quad=(\mathscr{R}_1^{n},w_N),\quad \forall w_N \in X_N^0(\Omega),\nonumber\\
&(\delta_t \xi_{\phi}^{n+1/2},w_N)=(\xi_{\psi}^{n+1/2},w_N)+(\mathscr{R}_2^{n},w_N),\quad \forall w_N \in X_N^0(\Omega),\label{errorF1-2}\\
&(\delta_t \xi_{\psi}^{n+1/2},w_N)+\gamma\mathcal{B}(\xi_{\phi}^{n+1/2},w_N)+\eta^2 (\xi_{\phi}^{n+1/2},w_N)\nonumber\\
&\qquad\qquad-\frac{\kappa_1}{2} \Big((|u^{n+1}|^2+|u^{n}|^2)-(|U_N^{n+1}|^2+|U_N^{n}|^2),w_N\Big)\label{errorF1-3}\\
&\qquad\qquad-\frac{\kappa_2}{2} \Big((|u^{n+1}|^4+|u^{n}|^4)-(|U_N^{n+1}|^4+|U_N^{n}|^4),w_N\Big)=(\mathscr{R}_3^{n},w_N),\quad \forall w_N \in X_N^0(\Omega),\nonumber
\end{align}
where
\begin{align*}
&\mathscr{R}_1^{n}=\widehat{\mathscr{R}}_1^{n}-\textrm{i} \delta_t\eta_u^{n+1/2},\\
&\mathscr{R}_2^{n}=\widehat{\mathscr{R}}_2^{n}+\eta_{\psi}^{n+1/2}-\delta_t \eta_{\phi}^{n+1/2},\\
&\mathscr{R}_3^{n}=\widehat{\mathscr{R}}_3^{n}-\delta_t \eta_{\psi}^{n+1/2}-\eta^2\eta_{\phi}^{n+1/2}.
\end{align*}
For the case $\alpha\neq \frac32$, according to Taylor's expansion and Lemma \ref{orgalpha}, we have
\begin{align*}\Big\| \eta_\psi^{n+\frac12}\Big\|&\leq\Big\| \eta_\psi^{n+\frac12}- \eta_\psi(\cdot,t_{n+1/2})\Big\|+\Big\| \eta_\psi(\cdot,t_{n+1/2})\Big\|\leq \mathcal{C}(\tau^2+N^{-r}),\\
\Big\|\delta_t \eta_\phi^{n+\frac12}\Big\|&\leq\Big\|\delta_t \eta_\phi^{n+\frac12}-\partial_t \eta_\phi(\cdot,t_{n+1/2})\Big\|+\Big\|\partial_t \eta_\phi(\cdot,t_{n+1/2})\Big\|\leq \mathcal{C}(\tau^2+N^{-r}).
\end{align*}
Similarly, we have
\begin{align*}
\Big\| \eta_\phi^{n+\frac12}\Big\|\leq \mathcal{C}(\tau^2+N^{-r}),\quad \Big\|\delta_t \eta_\psi^{n+\frac12}\Big\|\leq \mathcal{C}(\tau^2+N^{-r}),\quad \Big\|\delta_t \eta_u^{n+\frac12}\Big\|\leq \mathcal{C}(\tau^2+N^{-r}).
\end{align*}
Thus, we obtain
\begin{align*}
\Big\| \mathscr{R}_1^{n} \Big\| \leq \mathcal{C}(\tau^2+N^{-r}),\quad \Big\| \mathscr{R}_2^{n} \Big\| \leq \mathcal{C}(\tau^2+N^{-r}),\quad \Big\| \mathscr{R}_3^{n} \Big\| \leq \mathcal{C}(\tau^2+N^{-r}).
\end{align*}
By admitting the above relations and \eqref{derTRU}, we deduce
\begin{align*}
\Big\| \delta_t\mathscr{R}_2^{n+1/2} \Big\|\leq \mathcal{C}(\tau^2+N^{-r}).
\end{align*}

Selecting $w_N=\xi_u^{n+1/2}$ in \eqref{errorF1-1} and taking the imaginary part, we deduce
\begin{align}\label{error1}
&\textrm{Im}\Big\{ \textrm{i} (\delta_t\xi_u^{n+1/2},\xi_u^{n+1/2})\Big\}=-\kappa_1 \textrm{Im}(u^{n+1/2}\phi^{n+1/2}-U_N^{n+1/2}\Phi_N^{n+1/2},\xi_u^{n+1/2})\nonumber\\
&\quad\quad-\kappa_2
\textrm{Im}\Big((|u^{n+1}|^2+|u^{n}|^2)u^{n+1/2}\phi^{n+1/2}-(|U_N^{n+1}|^2+|U_N^{n}|^2)U_N^{n+1/2}\Phi_N^{n+1/2},\xi_u^{n+1/2}\Big)\\
&\quad\quad+\textrm{Im}(\mathscr{R}_1^{n},\xi_u^{n+1/2}).\nonumber
\end{align}
By some careful calculations, we obtain
\begin{align}\label{error1-1}
\textrm{Im}\Big\{ \textrm{i} (\delta_t\xi_u^{n+1/2},\xi_u^{n+1/2})\Big\}=\frac{1}{2\tau}\Big( \Big\| \xi_u^{n+1} \Big\|^2
-\Big\| \xi_u^{n} \Big\|^2 \Big).
\end{align}
Employing Lemma \ref{nonlinear}, Lemma \ref{orgalpha} and the Cauchy-Schwarz inequality, we conclude
\begin{align}\label{error1-2}
-\kappa_1 &\textrm{Im}(u^{n+1/2}\phi^{n+1/2}-U_N^{n+1/2}\Phi_N^{n+1/2},\xi_u^{n+1/2})\leq \kappa_1\Big\| u^{n+1/2}\phi^{n+1/2}-U_N^{n+1/2}\Phi_N^{n+1/2} \Big\|  \Big\|\xi_u^{n+1/2}\Big\|\nonumber\\
&\leq\mathcal{C}\Big( \Big\| u^{n+1/2}\phi^{n+1/2}-U_N^{n+1/2}\Phi_N^{n+1/2} \Big\|^2+\Big\|\xi_u^{n+1/2}\Big\|^2 \Big)\nonumber\\
&\leq \mathcal{C}\Big( \Big\| u^{n+1/2}-U_N^{n+1/2} \Big\|^2+ \Big\| \phi^{n+1/2}-\Phi_N^{n+1/2} \Big\|^2+\Big\|\xi_u^{n+1/2}\Big\|^2 \Big)\\
&\leq \mathcal{C}\Big( \Big\| \xi_u^{n+1/2} \Big\|^2+ \Big\| \xi_\phi^{n+1/2} \Big\|^2+N^{-2r}\Big)\nonumber\\
&\leq \mathcal{C}\Big( \Big\| \xi_u^{n+1} \Big\|^2+\Big\| \xi_u^{n} \Big\|^2+ \Big\| \xi_\phi^{n+1} \Big\|^2+\Big\|\xi_\phi^{n}\Big\|^2+N^{-2r}\Big),\nonumber
\end{align}
\begin{align}\label{error1-3}
-\kappa_2
\textrm{Im}&\Big((|u^{n+1}|^2+|u^{n}|^2)u^{n+1/2}\phi^{n+1/2}-(|U_N^{n+1}|^2+|U_N^{n}|^2)U_N^{n+1/2}\Phi_N^{n+1/2},\xi_u^{n+1/2}\Big)\nonumber\\
&\leq \kappa_2\Big\| (|u^{n+1}|^2+|u^{n}|^2)u^{n+1/2}\phi^{n+1/2}-(|U_N^{n+1}|^2+|U_N^{n}|^2)U_N^{n+1/2}\Phi_N^{n+1/2} \Big\| \Big\| \xi_u^{n+1/2} \Big\|\\
&\leq \mathcal{C}\Big( \Big\| \xi_u^{n+1} \Big\|^2+\Big\| \xi_u^{n} \Big\|^2+ \Big\| \xi_\phi^{n+1} \Big\|^2+\Big\|\xi_\phi^{n}\Big\|^2+N^{-2r}\Big)\nonumber
\end{align}
and
\begin{align}\label{error1-4}
\textrm{Im}(\mathscr{R}_1^{n},\xi_u^{n+1/2})\leq \mathcal{C}\Big( \Big\| \xi_u^{n+1} \Big\|^2+\Big\| \xi_u^{n} \Big\|^2+\tau^4 \Big).
\end{align}
Substituting \eqref{error1-1}--\eqref{error1-4} into \eqref{error1}, one has
\begin{align}\label{error1-5}
\Big\| \xi_u^{n+1} \Big\|^2 \leq \Big\| \xi_u^{n} \Big\|^2+\mathcal{C}\tau\Big( \Big\| \xi_u^{n+1} \Big\|^2+\Big\| \xi_u^{n} \Big\|^2+ \Big\| \xi_\phi^{n+1} \Big\|^2+\Big\|\xi_\phi^{n}\Big\|^2+\tau^4+N^{-2r}\Big).
\end{align}
Choosing $w_N=\xi_{\phi}^{n+1/2}$ in \eqref{errorF1-2}, we have
\begin{align}\label{error2-1}
\Big\| \xi_\phi^{n+1} \Big\|^2-\Big\| \xi_\phi^{n} \Big\|^2&=2\tau(\xi_{\psi}^{n+1/2},\xi_{\phi}^{n+1/2})+2\tau(\mathscr{R}_2^{n},\xi_{\phi}^{n+1/2})\\
&\leq \mathcal{C}\tau\Big( \Big\| \xi_\phi^{n+1} \Big\|^2+\Big\| \xi_\phi^{n} \Big\|^2
+\Big\| \xi_\psi^{n+1} \Big\|^2+\Big\| \xi_\psi^{n} \Big\|^2+\tau^4+N^{-2r} \Big).\nonumber
\end{align}
Employing the Gronwall inequality I (Lemma \ref{ineq2}), one obtains
\begin{align}\label{error2-2}
\Big\|\xi_\phi^{n+1}\Big\|^2 \leq \mathcal{C}\tau \sum\limits^{n}_{k=0}\Big\|\xi_\psi^{k+1}\Big\|^2+\mathcal{C}(\tau^4+N^{-2r}),
\end{align}
which implies
\begin{align}\label{ker}
\tau\sum\limits^{n}_{k=0} \Big\| \xi_\phi^{k+1} \Big\|^2&\leq
\mathcal{C}\tau^2\sum\limits^{n}_{k=0}\sum\limits^{k}_{j=0}\Big\| \xi_\psi^{j+1} \Big\|^2  + \mathcal{C}(\tau^4+N^{-2r})\nonumber\\
&\leq
\mathcal{C}\tau^2\sum\limits^{n}_{k=0}\sum\limits^{n}_{j=0}\Big\| \xi_\psi^{j+1} \Big\|^2  + \mathcal{C}(\tau^4+N^{-2r})&\\
&\leq
\mathcal{C}\tau\sum\limits^{n}_{k=0}\Big\| \xi_\psi^{k+1} \Big\|^2  + \mathcal{C}(\tau^4+N^{-2r}).\nonumber
\end{align}
Taking $w_N=\delta_t \xi_{\phi}^{n+1/2}$ in \eqref{errorF1-3}, we have
\begin{align}\label{error3}
&(\delta_t \xi_{\psi}^{n+1/2},\delta_t \xi_{\phi}^{n+1/2})+\gamma\frac{\Big| \xi_{\phi}^{n+1} \Big|^2_{\alpha/2}-\Big| \xi_{\phi}^{n} \Big|^2_{\alpha/2}}{2\tau}
+\eta^2 \frac{\Big\| \xi_{\phi}^{n+1} \Big\|^2-\Big\| \xi_{\phi}^{n} \Big\|^2}{2\tau}\nonumber\\
&\qquad\qquad-\frac{\kappa_1}{2} \Big((|u^{n+1}|^2+|u^{n}|^2)-(|U_N^{n+1}|^2+|U_N^{n}|^2),\delta_t \xi_{\phi}^{n+1/2}\Big)\\
&\qquad\qquad-\frac{\kappa_2}{2} \Big((|u^{n+1}|^4+|u^{n}|^4)-(|U_N^{n+1}|^4+|U_N^{n}|^4),\delta_t \xi_{\phi}^{n+1/2}\Big)=(\mathscr{R}_3^{n},\delta_t \xi_{\phi}^{n+1/2}).\nonumber
\end{align}
It follows from \eqref{errorF1-2} that
\begin{align}\label{error3-1}
\Big( \delta_t\xi_\psi^{n+1/2} ,\delta_t\xi_\phi^{n+1/2}\Big)&=( \xi_\psi^{n+\frac12},\delta_t\xi_\psi^{n+1/2} )+(\mathscr{R}_2^{n},\delta_t\xi_\psi^{n+1/2})\nonumber\\
&=\frac{\Big\| \xi_\psi^{n+1} \Big\|^2-\Big\| \xi_\psi^{n} \Big\|^2}{2\tau}+(\mathscr{R}_2^{n},\delta_t\xi_\psi^{n+1/2}).
\end{align}
Setting $w_N=\delta_t \Phi_N^{n+\frac12}- \Psi_N^{n+\frac12}$ in \eqref{CNGLS1-2}, we deduce
\begin{align*}
\delta_t \Phi_N^{n+\frac12}= \Psi_N^{n+\frac12}.
\end{align*}
It is trivial to check that
\begin{align*}
(\xi_\psi^{n+1/2},w_N)=(\delta_t\xi_\phi^{n+1/2}+\Pi_N^{\alpha/2,0}\psi^{n+1/2}-\Pi_N^{\alpha/2,0}\delta_t\phi^{n+1/2},w_N),\quad
\forall w_N\in X_N^0(\Omega),
\end{align*}
which implies
\begin{align*}
\delta_t\xi_\phi^{n+1/2}=\xi_\psi^{n+1/2}+\Pi_N^{\alpha/2,0}\delta_t\phi^{n+1/2}-\Pi_N^{\alpha/2,0}\psi^{n+1/2}.
\end{align*}
By employing Lemma \ref{nonlinear}, Lemma \ref{orgalpha} and the Cauchy-Schwarz inequality that
\begin{align}\label{error3-2}
-&\frac{\kappa_1}{2} \Big((|u^{n+1}|^2+|u^{n}|^2)-(|U_N^{n+1}|^2+|U_N^{n}|^2),\delta_t \xi_{\phi}^{n+1/2}\Big)\nonumber\\
&\leq \frac{\kappa_1}{2} \Big\| (|u^{n+1}|^2+|u^{n}|^2)-(|U_N^{n+1}|^2+|U_N^{n}|^2) \Big\|  \Big\| \delta_t \xi_{\phi}^{n+1/2} \Big\|\nonumber\\
&\leq \mathcal{C}\Big(  \Big\| (|u^{n+1}|^2+|u^{n}|^2)-(|U_N^{n+1}|^2+|U_N^{n}|^2) \Big\|^2+\Big\| \delta_t \xi_{\phi}^{n+1/2} \Big\|^2\Big)\nonumber\\
&\leq \mathcal{C}\Big( \Big\| u^{n+1}-U_N^{n+1} \Big\|^2+\Big\| u^{n}-U_N^{n} \Big\|^2+\Big\| \delta_t \xi_{\phi}^{n+1/2} \Big\|^2 \Big)\nonumber\\
&\leq \mathcal{C}\Big( \Big\| \xi_u^{n+1} \Big\|^2+\Big\| \xi_u^{n} \Big\|^2+\Big\| \xi_\psi^{n+1/2}+\Pi_N^{\alpha/2,0}\delta_t\phi^{n+1/2}-\Pi_N^{\alpha/2,0}\psi^{n+1/2} \Big\|^2+N^{-2r} \Big)\nonumber\\
&= \mathcal{C}\Big( \Big\| \xi_u^{n+1} \Big\|^2+\Big\| \xi_u^{n} \Big\|^2+\Big\| \xi_\psi^{n+1/2}+\Pi_N^{\alpha/2,0}\delta_t\phi^{n+1/2}-\delta_t\phi^{n+1/2}\\
&\quad+\delta_t\phi^{n+1/2}-\psi^{n+1/2}+\psi^{n+1/2}
-\Pi_N^{\alpha/2,0}\psi^{n+1/2} \Big\|^2+N^{-2r} \Big)\nonumber\\
&\leq\mathcal{C}\Big( \Big\| \xi_u^{n+1} \Big\|^2+\Big\| \xi_u^{n} \Big\|^2+\Big\| \xi_\psi^{n+1/2}\Big\|^2+\Big\| \Pi_N^{\alpha/2,0}\delta_t\phi^{n+1/2}-\delta_t\phi^{n+1/2}\Big\|^2\nonumber\\
&\quad+\Big\|\delta_t\phi^{n+1/2}-\psi^{n+1/2}\Big\|^2+\Big\|\psi^{n+1/2}
-\Pi_N^{\alpha/2,0}\psi^{n+1/2} \Big\|^2+N^{-2r} \Big)\nonumber\\
&\leq\mathcal{C}\Big( \Big\| \xi_u^{n+1} \Big\|^2+\Big\| \xi_u^{n} \Big\|^2+\Big\| \xi_\psi^{n+1}\Big\|^2+\Big\| \xi_\psi^{n}\Big\|^2
+\Big\|\delta_t\phi^{n+1/2}-\psi^{n+1/2}\Big\|^2+N^{-2r}\Big)\nonumber\\
&\leq\mathcal{C}\Big( \Big\| \xi_u^{n+1} \Big\|^2+\Big\| \xi_u^{n} \Big\|^2+\Big\| \xi_\psi^{n+1}\Big\|^2+\Big\| \xi_\psi^{n}\Big\|^2
+\tau^4+N^{-2r}\Big),\nonumber
\end{align}
where the last inequality has used the following estimate: by taking $w_N=\delta_t {\phi}^{n+1/2}-{\psi}^{n+1/2}$ in \eqref{semidiscrete2}, one has
\begin{align*}
\Big\|\delta_t {\phi}^{n+1/2}-{\psi}^{n+1/2}\Big\|\leq\Big\|\mathscr{R}_2^{n}\Big\|\leq \mathcal{C}(\tau^4+N^{-2r}).
\end{align*}
Similar estimates, we obtain
\begin{align}\label{error3-3}
-\frac{\kappa_2}{2} \Big((|u^{n+1}|^4&+|u^{n}|^4)-(|U_N^{n+1}|^4+|U_N^{n}|^4),\delta_t \xi_{\phi}^{n+1/2}\Big)\nonumber\\
&\leq\mathcal{C}\Big( \Big\| \xi_u^{n+1} \Big\|^2+\Big\| \xi_u^{n} \Big\|^2+\Big\| \xi_\psi^{n+1}\Big\|^2+\Big\| \xi_\psi^{n}\Big\|^2
+\tau^4+N^{-2r}\Big)
\end{align}
and
\begin{align}\label{error3-4}
(\mathscr{R}_3^{n},\delta_t \xi_{\phi}^{n+1/2})&\leq\mathcal{C}\Big( \Big\| \xi_\psi^{n+1}\Big\|^2+\Big\| \xi_\psi^{n}\Big\|^2
+\tau^4+N^{-2r}\Big).
\end{align}
Substituting \eqref{error3-1}--\eqref{error3-4} into \eqref{error3}, we arrive at
\begin{align}\label{error4}
&\Big\| \xi_\psi^{n+1} \Big\|^2+\gamma\Big|\xi_{\phi}^{n+1} \Big|^2_{\alpha/2}+\eta^2\Big\| \xi_{\phi}^{n+1} \Big\|^2\nonumber\\
&\leq \Big\| \xi_\psi^{n} \Big\|^2+\gamma\Big|\xi_{\phi}^{n} \Big|^2_{\alpha/2}+\eta^2\Big\| \xi_{\phi}^{n} \Big\|^2
+\mathcal{C}\tau\Big( \Big\| \xi_u^{n+1} \Big\|^2+\Big\| \xi_u^{n} \Big\|^2+\Big\| \xi_\psi^{n+1}\Big\|^2+\Big\| \xi_\psi^{n}\Big\|^2
+\tau^4+N^{-2r}\Big)\\
&\qquad -2\tau(\mathscr{R}_2^{n},\delta_t\xi_\psi^{n+1/2}).\nonumber
\end{align}
Summing up \eqref{error1-5} and \eqref{error4}, we have
\begin{align}\label{error5}
&\Big\| \xi_u^{n+1} \Big\|^2+\Big\| \xi_\psi^{n+1} \Big\|^2+\gamma\Big|\xi_{\phi}^{n+1} \Big|^2_{\alpha/2}+\eta^2\Big\| \xi_{\phi}^{n+1} \Big\|^2\nonumber\\
&\leq \Big\| \xi_u^{n} \Big\|^2+\Big\| \xi_\psi^{n} \Big\|^2+\gamma\Big|\xi_{\phi}^{n} \Big|^2_{\alpha/2}+\eta^2\Big\| \xi_{\phi}^{n} \Big\|^2
+\mathcal{C}\tau\Big( \Big\| \xi_u^{n+1} \Big\|^2+\Big\| \xi_u^{n} \Big\|^2+\Big\| \xi_\phi^{n+1}\Big\|^2+\Big\| \xi_\phi^{n}\Big\|^2\\
&\qquad+\Big\| \xi_\psi^{n+1}\Big\|^2+\Big\| \xi_\psi^{n}\Big\|^2+\tau^4+N^{-2r}\Big)
-2\tau(\mathscr{R}_2^{n},\delta_t\xi_\psi^{n+1/2})\nonumber\\
&\leq \mathcal{C}\tau\sum\limits^{n}_{k=0}\Big( \Big\| \xi_u^{k+1} \Big\|^2+\Big\| \xi_\phi^{k+1} \Big\|^2+\Big\| \xi_\psi^{k+1} \Big\|^2\Big)+\mathcal{C}\Big( \tau^4+N^{-2r} \Big)-2\tau\sum\limits^{n}_{k=0}(\mathscr{R}_2^{k},\delta_t\xi_\psi^{k+1/2}).\nonumber
\end{align}
It follows from \eqref{ker} and Lemma \ref{ineq1} that
\begin{align}\label{error6}
&\Big\| \xi_u^{n+1} \Big\|^2+\Big\| \xi_\psi^{n+1} \Big\|^2+\gamma\Big|\xi_{\phi}^{n+1} \Big|^2_{\alpha/2}+\eta^2\Big\| \xi_{\phi}^{n+1} \Big\|^2\nonumber\\
&\leq \mathcal{C}\tau\sum\limits^{n}_{k=0}\Big( \Big\| \xi_u^{k+1} \Big\|^2+\Big\| \xi_\psi^{k+1} \Big\|^2\Big)+\mathcal{C}\Big( \tau^4+N^{-2r} \Big)-2\tau\sum\limits^{n}_{k=0}(\mathscr{R}_2^{k},\delta_t\xi_\psi^{k+1/2})\\
&\leq \mathcal{C}\tau\sum\limits^{n}_{k=0}\Big( \Big\| \xi_u^{k+1} \Big\|^2+\Big\| \xi_\psi^{k+1} \Big\|^2\Big)+\frac12 \Big\|\xi_\psi^{n+1}\Big\|^2+\mathcal{C}\Big( \tau^4+N^{-2r} \Big).\nonumber
\end{align}
Then, we reformulate \eqref{error6} as
\begin{align}\label{error7}
&\Big\| \xi_u^{n+1} \Big\|^2+\frac12\Big\| \xi_\psi^{n+1} \Big\|^2+\gamma\Big|\xi_{\phi}^{n+1} \Big|^2_{\alpha/2}+\eta^2\Big\| \xi_{\phi}^{n+1} \Big\|^2\nonumber\\
&\leq \mathcal{C}\tau\sum\limits^{n}_{k=0}\Big( \Big\| \xi_u^{k+1} \Big\|^2+\frac12\Big\| \xi_\psi^{k+1} \Big\|^2\Big)+\mathcal{C}\Big( \tau^4+N^{-2r} \Big)\\
&\leq \mathcal{C}\tau\sum\limits^{n}_{k=0}\Big( \Big\| \xi_u^{k+1} \Big\|^2+\frac12\Big\| \xi_\psi^{k+1} \Big\|^2
+\gamma\Big|\xi_{\phi}^{n+1} \Big|^2_{\alpha/2}+\eta^2\Big\| \xi_{\phi}^{n+1} \Big\|^2\Big)+\mathcal{C}\Big( \tau^4+N^{-2r} \Big).\nonumber
\end{align}
Utilizing the Gronwall inequality II (Lemma \ref{ineq3}) for \eqref{error7}, we conclude
\begin{align*}
\Big\| \xi_u^{n+1} \Big\|\leq \mathcal{C}( \tau^2+N^{-r} ),\quad \Big\| \xi_{\phi}^{n+1} \Big\| \leq \mathcal{C}( \tau^2+N^{-r} ),
\quad \Big|\xi_{\phi}^{n+1} \Big|_{\alpha/2}\leq \mathcal{C}( \tau^2+N^{-r} ).
\end{align*}
By applying Lemma \ref{orgalpha}, we derive
\begin{align*}
\Big|{\phi}^{n+1}-{\Phi}_N^{n+1} \Big|_{\alpha/2}\leq \Big|\eta_{\phi}^{n+1} \Big|_{\alpha/2}+\Big|\xi_{\phi}^{n+1} \Big|_{\alpha/2}\leq \mathcal{C}( \tau^2+N^{\alpha/2-r} ),
\end{align*}
\begin{align*}
\Big\| u^{n+1}-U_N^{n+1} \Big\|\leq \Big\| \eta_u^{n+1} \Big\|+\Big\| \xi_u^{n+1} \Big\| \leq \mathcal{C}( \tau^2+N^{-r} )
\end{align*}
and
\begin{align*}
\Big\| \phi^{n+1}-\Phi_N^{n+1} \Big\|\leq \Big\| \eta_\phi^{n+1} \Big\|+\Big\| \xi_\phi^{n+1} \Big\| \leq \mathcal{C}( \tau^2+N^{-r} ).
\end{align*}
Therefore, combining the above estimates and Lemma \ref{boundinf}, we get
\begin{align*}
\Big\| \phi^{n+1}-\Phi_N^{n+1} \Big\|_\infty\leq \mathcal{C}\sqrt{ \Big\| \phi^{n+1}-\Phi_N^{n+1} \Big\|^2+\Big|{\phi}^{n+1}-{\Phi}_N^{n+1} \Big|_{\alpha/2}^2 }
 \leq \mathcal{C}( \tau^2+N^{\alpha/2-r} ).
\end{align*}
When $\alpha=\frac32$, the inference procedures are similar to the case $\alpha\neq\frac32$, here we omit it. We obtain the following consequences
\begin{align*}
\Big\| u^n-U^n_N \Big\| \leq \mathcal{C}(\tau^2+N^{\epsilon-r}),\quad
\Big\| \phi^n-\Phi^n_N \Big\|\leq \mathcal{C}(\tau^2+N^{\epsilon-r}),
\quad \Big\| \phi^n-\Phi^n_N \Big\|_\infty \leq \mathcal{C}(\tau^2+N^{\frac{\alpha}{2}-r}),\quad 0<\epsilon<\frac12.
\end{align*}
Thus, the proof is completed.
\end{proof}

\section{Linearly implicit and decoupled conservative scheme}\label{esavsgm}
The ESAV formulation of the FNKGS equation \eqref{equ1}--\eqref{equ2} introduces the exponential scalar auxiliary variables
\begin{align*}
p(t)=\exp\Big( 2\int_\Omega (\kappa_1|u|^2 + \kappa_2|u|^4)\phi dx\Big).
\end{align*}
Taking the derivative of $\ln(p)$, we have
\begin{align*}
\frac{d}{dt}\ln(p)&=2\int_\Omega \Big( 2\kappa_1\textrm{Re}( \bar{u}u_t )+4\kappa_2|u|^2\textrm{Re}( \bar{u}u_t ) \Big)\phi+\Big(  \kappa_1|u|^2 + \kappa_2|u|^4 \Big)\phi_tdx\\
&=4\textrm{Re}\Big( (\kappa_1 + 2\kappa_2|u|^2)u\phi,u_t \Big)+2\Big( \kappa_1|u|^2+\kappa_2|u|^4,\phi_t \Big)\\
&=4\textrm{Re}\Big( p{F}(u,\phi)u\phi,u_t \Big)+2\Big( p{G}(u,\phi),\phi_t \Big),
\end{align*}
where
\begin{align*}
{F}(u,\phi)=\frac{\kappa_1+2\kappa_2|u|^2}{\exp\Big( 2\int_\Omega (\kappa_1|u|^2 + \kappa_2|u|^4)\phi dx\Big)},\quad
{G}(u,\phi)=\frac{\kappa_1|u|^2+\kappa_2|u|^4}{\exp\Big( 2\int_\Omega (\kappa_1|u|^2 + \kappa_2|u|^4)\phi dx\Big)}.
\end{align*}
Then, the system \eqref{weak1-1}--\eqref{weak1-3} can be reformulated into an equivalent form
\begin{align}
&\textrm{i} (u_t,w)-\frac{\lambda}{2}\mathcal{B}(u,w)+ (p{F}(u,\phi)u\phi,w)=0,\\
&(\phi_t,w)=(\psi,w),\\
&(\psi_{t},w)+\gamma\mathcal{B}(\phi,w)+\eta^2 (\phi,w)- (p{G}(u,\phi),w)=0,\\
&\frac{d}{dt}\ln(p)=4\textrm{Re}\Big( p{F}(u,\phi)u\phi,u_t \Big)+2\Big( p{G}(u,\phi),\phi_t \Big),
\end{align}
with
\begin{align*}
u({ x},0)=u_0({ x}),\quad \phi({ x},0)=\phi_0({ x}),\quad \phi_t(x,0)=\phi_1(x)\quad p_0=\exp\Big( 2\int_\Omega (\kappa_1|u_0|^2 + \kappa_2|u_0|^4)\phi dx\Big).
\end{align*}
The energy conservation in the ESAV formulation is
\begin{align*}
\frac{d}{dt}\Big(\Big\|\psi\Big\|^2+\gamma\Big| \phi \Big|^2_{\alpha/2}+\eta^2 \Big\| \phi \Big\|^2
+\lambda\Big| u \Big|^2_{\alpha/2}-\ln(p)\Big)=0
.
\end{align*}
Combining the implicit midpoint method and spectral Galerkin method, the ESAV--SGM for \eqref{equ1}--\eqref{equ2} is constructed as: finding $U^n_N,\Phi^n_N,\Psi^n_N\in X_N^0(\Omega)$, for $w_N\in X_N^0(\Omega)$, such that
\begin{small}
\begin{align}
&\textrm{i} (\delta_tU_N^{n+1/2},w_N)-\frac{\lambda}{2}\mathcal{B}(U_N^{n+1/2},w_N)
+\widetilde{P}^{n+1/2}\Big(F(\widetilde{U}_N^{n+1/2},\widetilde{\Phi}_N^{n+1/2})\widetilde{U}_N^{n+1/2}\widetilde{\Phi}_N^{n+1/2},w_N\Big)=0,\label{ESAVGLS1-1}\\
&(\delta_t \Phi_N^{n+1/2},w_N)=(\Psi_N^{n+1/2},w_N),\label{ESAVGLS1-2}\\
&(\delta_t \Psi_N^{n+1/2},w_N)+\gamma\mathcal{B}(\Phi_N^{n+1/2},w_N)+\eta^2 (\Phi_N^{n+1/2},w_N)-\widetilde{P}^{n+1/2}\Big(G(\widetilde{U}_N^{n+1/2},\widetilde{\Phi}_N^{n+1/2}),w_N\Big)=0,\label{ESAVGLS1-3}\\
&\frac{\ln({P}^{n+1})-\ln({P}^{n})}{\tau}=4\widetilde{P}^{n+1/2} \textrm{Re}\Big( {F}(\widetilde{U}_N^{n+1/2},\widetilde{\Phi}_N^{n+1/2})\widetilde{U}_N^{n+1/2}\widetilde{\Phi}_N^{n+1/2},\delta_t{U}_N^{n+1/2} \Big)\nonumber\\
&\qquad\qquad\qquad\qquad\qquad+2\widetilde{P}^{n+1/2}\Big(G(\widetilde{U}_N^{n+1/2},\widetilde{\Phi}_N^{n+1/2}),\delta_t{\Phi}_N^{n+1/2} \Big),\label{ESAVGLS1-4}
\end{align}
\end{small}
with the startup scheme
\begin{small}
\begin{align}
&\textrm{i} (\delta_tU_N^{1/2},w_N)-\frac{\lambda}{2}\mathcal{B}(U_N^{1/2},w_N)
+{P}^{1/2}\Big(F({U}_N^{1/2},\Phi_N^{1/2}){U}_N^{1/2}{\Phi}_N^{1/2},w_N\Big)=0,\label{ESAVGLS0-1}\\
&(\delta_t \Phi_N^{1/2},w_N)=(\Psi_N^{1/2},w_N),\label{ESAVGLS0-2}\\
&(\delta_t \Psi_N^{1/2},w_N)+\gamma\mathcal{B}(\Phi_N^{1/2},w_N)+\eta^2 (\Phi_N^{1/2},w_N)-{P}^{1/2}\Big(G({U}_N^{1/2},\Phi_N^{1/2}),w_N\Big)=0,\label{ESAVGLS0-3}\\
&\frac{\ln({P}^{1})-\ln({P}^{0})}{\tau}=4{P}^{1/2} \textrm{Re}\Big( {F}({U}_N^{1/2},\Phi_N^{1/2}){U}_N^{1/2}{\Phi}_N^{1/2},\delta_t{U}_N^{1/2} \Big)
+2{P}^{1/2}\Big(G({U}_N^{1/2},\Phi_N^{1/2}),\delta_t{\Phi}_N^{1/2} \Big).\label{ESAVGLS0-4}
\end{align}
\end{small}

\begin{thm}
The numerical solution of the ESAV--SGM \eqref{ESAVGLS1-1}-\eqref{ESAVGLS1-4} enjoys an energy preservation in the sense that
\begin{align*}
\mathcal{E}^n=\cdots=\mathcal{E}^0,\qquad n=0,1,\cdots,N_t,
\end{align*}
where
\begin{align*}
\mathcal{E}^n=\Big\|\Psi^n_N\Big\|^2+\gamma\Big| \Phi^n_N \Big|^2_{\alpha/2}+\eta^2 \Big\| \Phi^n_N \Big\|^2+\lambda\Big| U^n_N \Big|^2_{\alpha/2}-\ln(P^n).
\end{align*}
\end{thm}
\begin{proof}
Setting $w_N=\delta_tU_N^{n+1/2}$ in \eqref{ESAVGLS1-1} and taking the real part, then choosing $w_N=\delta_t \Phi_N^{n+1/2}$ in \eqref{ESAVGLS1-3}, summing up the above resulting formulae similar to Theorem \ref{EP1-1}, the energy preservation is readily and immediately available.
\end{proof}

\section{Decoupled iterative implementation}\label{imple}
In this section, we focus on the implementation of CN--SGM \eqref{CNGLS1-1}--\eqref{CNGLS1-3}, the implementation of ESAV--SGM can be similarly modified. The associated numerical solutions $U_N^n$, $\Phi_N^n$ and $\Psi_N^n$ are of the form
\begin{align*}
U_N^n=\sum\limits^{N-2}_{k=0}\widehat{U}^n_{k}\varsigma_k(x),\qquad \Phi_N^n=\sum\limits^{N-2}_{k=0}\widehat{\Phi}^n_{k}\varsigma_k(x),\qquad \Psi_N^n=\sum\limits^{N-2}_{k=0}\widehat{\Psi}^n_{k}\varsigma_k(x)
\end{align*}
belong to the $N$-dimensional polynomial space $X_N^0(\Omega)$ which is given by
\begin{align*}
X_{N}^0(\Omega)=\textrm{span}\{ \varsigma_k(x):~k=0,1,\cdots,N-2 \},
\end{align*}
where $\varsigma_k(x)$ is determined by the following recurrence relation
\begin{align*}
&\varsigma_k(x)=L_k(\widehat{x})-L_{k+2}(\widehat{x}),\quad \widehat{x}\in[-1,1],\quad x=\frac{(b-a)\widehat{x}+(a+b)}{2}\in[a,b],
\end{align*}
in which $L_k(\hat{x})$ represents Legendre orthogonal polynomial \cite{shenj} which satisfies the following three-term recurrence relation
\begin{equation*}
\left\{
\begin{array}{ll}
L_0(\hat{x})=1,\quad L_1(\hat{x})=\hat{x},\\
(k+1)L_{k+1}(\hat{x})=(2k+1)\hat{x}L_k(\hat{x})-kL_{k-1}(\hat{x}),\quad k\geq 1.
\end{array}
\right.
\end{equation*}

The associated mass and stiff matrices of CN--SGM are all symmetric and computed by
\begin{align*}
\mathbb{M}_{l,k}=\Big(\varsigma_k({x}),\varsigma_l({x})\Big),\quad
\mathbb{S}_{l,k}
=\frac{1}{2\cos(\frac{\alpha\pi}{2})}\Big[\Big({}_{a}^{RL}\!D^{\alpha/2}_{x}\varsigma_k({x}),{}_{x}^{RL}\!D^{\alpha/2}_{b}\varsigma_l({x})\Big)
+\Big({}_{x}^{RL}\!D^{\alpha/2}_{b}\varsigma_k({x}),{}_{a}^{RL}\!D^{\alpha/2}_{x}\varsigma_l({x})\Big)\Big].
\end{align*}
The linearized and decoupled fixed-point iterative algorithm is presented for the CN--SGM in the following form
\begin{equation*}
\left\{
\begin{array}{ll}
\Big( \textrm{i}\mathbb{M}-\frac{\lambda\tau}{4} \mathbb{S} \Big)\widehat{U}^{n+1,s+1}=\Big( \textrm{i}\mathbb{M}+\frac{\lambda\tau}{4} \mathbb{S} \Big)\widehat{U}^{n}-\frac{\tau}{4}\Big( \kappa_1\mathcal{N}_1^{n+1,s}+\kappa_2\mathcal{N}_2^{n+1,s} \Big),\quad s=0,1,\cdots,\\\\
\Big( \mathbb{M}+\frac{\tau^2}{4}\mathbb{S}+\frac{\eta^2\tau^2}{4}\mathbb{M} \Big)\widehat{\Phi}^{n+1,s+1}
=\Big( \mathbb{M}-\frac{\tau^2}{4}\mathbb{S}-\frac{\eta^2\tau^2}{4}\mathbb{M} \Big)\widehat{\Phi}^{n}+\tau\mathbb{M}\widehat{\Psi}^{n}
+\frac{\tau^2}{4}\Big(\kappa_1\mathcal{N}_3^{n+1,s}+\kappa_2\mathcal{N}_4^{n+1,s}\Big),
\end{array}
\right.
\end{equation*}
where
\begin{equation*}
\widehat{U}^{n+1,0}=
\left\{
\begin{array}{ll} \widehat{U}^{0},\quad n=0,\\
2\widehat{U}^{n}-\widehat{U}^{n-1},\quad n\geq 1,
\end{array}
\right.
\qquad
\widehat{\Phi}^{n+1,0}=
\left\{
\begin{array}{ll} \widehat{\Phi}^{0},\quad n=0,\\
2\widehat{\Phi}^{n}-\widehat{\Phi}^{n-1},\quad n\geq 1
\end{array}
\right.
\end{equation*}
and
\begin{equation*}
\left\{
\begin{array}{ll}
(\mathcal{N}_1^{n+1,s})_{l}=\Big(( U_N^{n}+U_N^{n+1,s} )( \Phi_N^{n}+\Phi_N^{n+1,s} ),\varsigma_l(x)\Big),~l=0,1,\cdots,N-2,\\\\
(\mathcal{N}_2^{n+1,s})_{l}=\Big((|U_N^{n}|^2+|U_N^{n+1,s}|^2)( U_N^{n}+U_N^{n+1,s} )( \Phi_N^{n}+\Phi_N^{n+1,s} ),\varsigma_l(x)\Big),~l=0,1,\cdots,N-2,\\\\
(\mathcal{N}_3^{n+1,s})_{l}=\Big(|U_N^{n}|^2+|U_N^{n+1,s}|^2,\varsigma_l(x)\Big),~l=0,1,\cdots,N-2,\\\\
(\mathcal{N}_4^{n+1,s})_{l}=\Big(|U_N^{n}|^4+|U_N^{n+1,s}|^4,\varsigma_l(x)\Big),~l=0,1,\cdots,N-2.
\end{array}
\right.
\end{equation*}
Then, $U_N^{n+1,s+1}$ and $\Phi_N^{n+1,s+1}$ numerically converge to the numerical solutions $U_N^{n+1}$ and $\Phi_N^{n+1}$, respectively, if there satisfies $\Big\| U_N^{n+1,s+1}-U_N^{n+1,s} \Big\|+\Big\| \Phi_N^{n+1,s+1}-\Phi_N^{n+1,s} \Big\|\leq tol$ for the given stopping criterion $tol$ \cite{anj}.

\section{Numerical experiments}\label{numerexper}
In this section, some numerical results are reported to verify the proposed spectral Galerkin method. All the simulations are implemented by using Matlab R2018a software on a computer with Intel Core i7 and 16 GB RAM. Without special instructions, we always take the stopping criterion $tol=10^{-14}$.

In numerical tests, we compute the $L^2-$ and $L^\infty-$norm errors at $t=N_t\tau$ by
\begin{align*}
\|u^{N_t}-U_N^{N_t}\|=\sqrt{\int_\Omega \Big|  u^{N_t}(x)-U_N^{N_t}(x)  \Big|^2dx}\approx \sqrt{\frac{b-a}{2}\sum\limits^{M}_{j=0} \Big| u^{N_t}(x_j)-U^{N_t}_N(x_j) \Big|^2\varpi_j}
\end{align*}
and
\begin{align*}
\Big\| u^{N_t}-U_N^{N_t} \Big\|_\infty:=\underset{0\leq j\leq M }{\max}\Big| u^{N_t}(x_j)-U^{N_t}_N(x_j) \Big|\quad \textrm{with}\quad x_j=\frac{(b-a)\widehat{x}_j+(a+b)}{2},
\end{align*}
where $\{\widehat{x}_j\}$ and $\{\varpi_j\}$ are points and weights of the Legendre--Gauss--Lobatto quadrature, respectively, and $M=\mu N$ $(\mu>1, \mu\in \mathbb{N}_+)$. In convergence test, for the case that the solution of equation is unknown, we intend to regard the more accurate numerical solutions as the reference solutions. To illustrate the energy preservation of the proposed numerical schemes, define the following relative mass and energy deviations
\begin{align*}
\textrm{RM}^n = \frac{|\mathcal{M}^n-\mathcal{M}^0|}{\mathcal{M}^0},\qquad \textrm{RE}^n = \frac{|\mathcal{E}^n-\mathcal{E}^0|}{\mathcal{E}^0}.
\end{align*}

\begin{ex}\label{ex1}
\emph{Consider the FNKGS system \eqref{equ1}--\eqref{equ2} with $\lambda=1$, $\kappa_1=1$, $\kappa_2=0$, $\gamma=1$ and $\eta=1$ in domain $(-20,20)\times(0,T]$. When $\alpha=2$,  the system has the exact solitary wave solutions as follows:
\begin{align}
&u(x,t,\nu)=\frac{3\sqrt{2}}{4\sqrt{1-\nu^2}} \sech^2\frac{x-\nu t-\chi_0}{2\sqrt{1-\nu^2}}
\exp\Big(  \textrm{i}( \nu x+\frac{1-\nu^2+\nu^4}{2(1-\nu^2)})t \Big),\label{exact1-1}\\
&\phi(x,t,\nu)=\frac{3}{4(1-\nu^2)}\sech^2\frac{x-\nu t-\chi_0}{2\sqrt{1-\nu^2}},\label{exact1-2}\\
&\phi_t(x,t,\nu)=\frac{3\nu}{4(1-\nu^2)^\frac32}\sech^2\frac{x-\nu t-\chi_0}{2\sqrt{1-\nu^2}}\tanh\frac{x-\nu t-\chi_0}{2\sqrt{1-\nu^2}}.\label{exact1-3}
\end{align}
Here, the initial datum $u_0(x)$, $\phi_0(x)$ and $\phi_1(x)$ are determined by the exact solutions \eqref{exact1-1}--\eqref{exact1-3}, where $\nu=0.8$ and $\chi_0=-10$.
}
\end{ex}
Firstly, we carry out testing the convergence accuracy of CN--SGM in time and space. For the fixed temporal step, we find from Fig. \ref{numer0} that the errors are exponentially decaying in $L^2$ and $L^\infty$ norm. For the fixed polynomial degree $N=150$, the temporal numerical results are listed in Table \ref{TAB-1}. The numerical results show that the proposed CN--SGM possesses the second-order accuracy in $L^2$ and $L^\infty$ norm, which are in great accordance with the theoretical results. Secondly, we show the effect of fractional order $\alpha$ on numerical profiles. It is obvious to observe from Fig. \ref{numer1} that
the fractional order $\alpha$ will dramatically affect the shapes of the solitons, we refer readers to \cite{liNA} for more details. Finally, for the purpose of numerical comparisons, we utilize the CN--SGM, the ESAV--SGM and the finite difference scheme (LF-FDM \cite{shiy}) to solve Example \ref{ex1} until $t=100$. Fig. \ref{numer2} and Fig. \ref{numer3} illustrates that the CN--SGM uniformly preserves the mass and energy to machine accuracy. Fig. \ref{numer3} verifies that the ESAV--SGM preserves the energy well, but the mass preservation can not be done from Fig. \ref{numer2}. Actually, the mass conservation is also an important structure for the convergence and stability analyses of a conservative algorithm. Fig. \ref{numer2} and Fig. \ref{numer3} shows that the LF-FDM is not up to the task of long-time mass and energy preservations. These numerical pictures illustrate that the CN--SGM enjoys a great superior than the ESAV--SGM and the LF-FDM for energy preservation. Subsequently, we plot the numerical errors and associated CPU time of the CN--SGM and the ESAV--SGM in Fig. \ref{cpu1-1} by setting the various temporal step $\tau$ and the same polynomial degree $N$. From the numerical results, we observe that the ESAV--SGM performs more effective than the CN--SGM, in the other word, the ESAV--SGM needs less time to obtain the same errors. It is worth noticing that the CN--SGM can obtain the more accurate numerical results than the ESAV--SGM for the same temporal step $\tau$ and polynomial degree $N$. Therefore, the CN--SGM and the ESAV--SGM each have their own advantages in long-time computations.
\begin{figure}[!htb]
\centering
\includegraphics[width=0.3\linewidth]{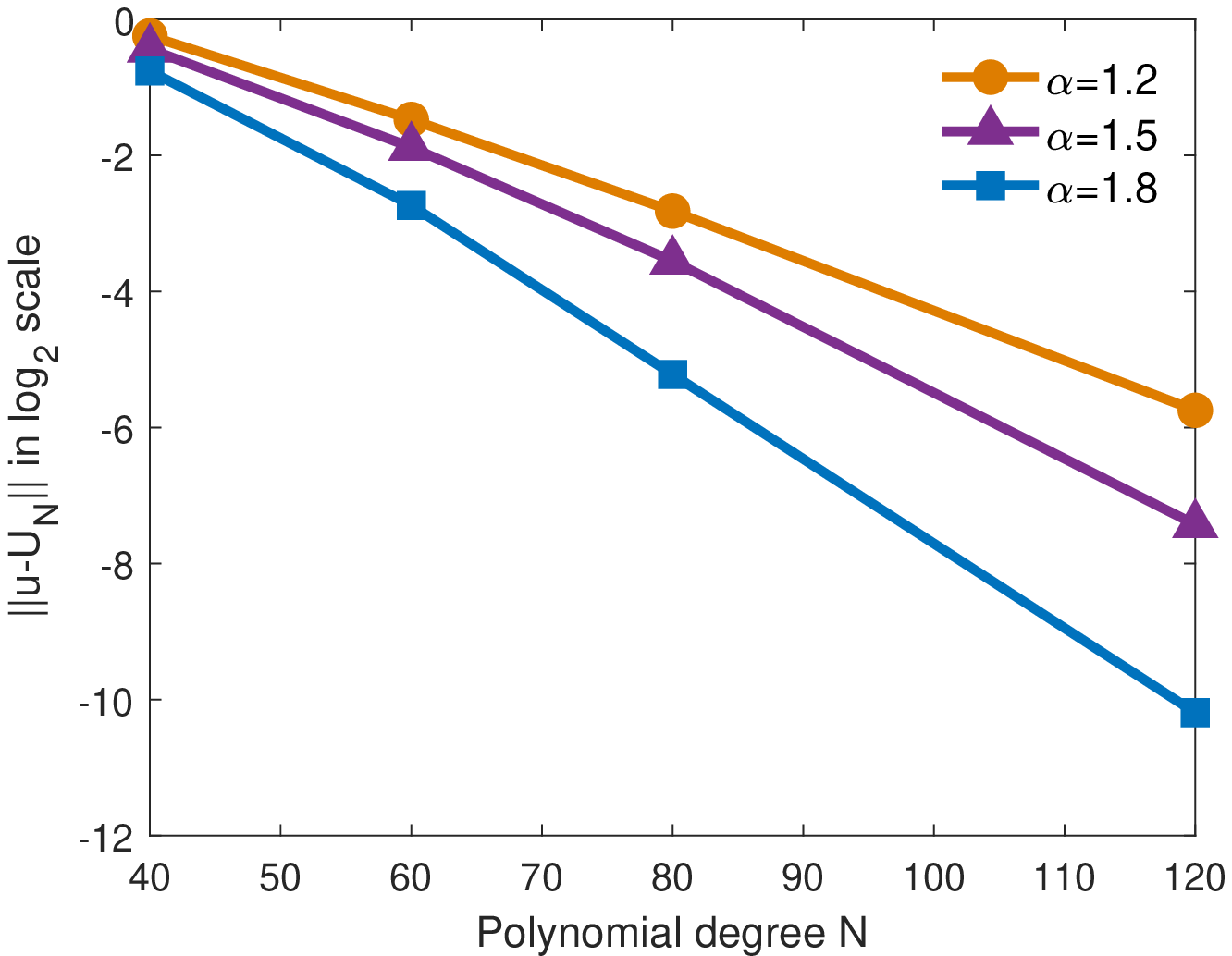}
\includegraphics[width=0.3\linewidth]{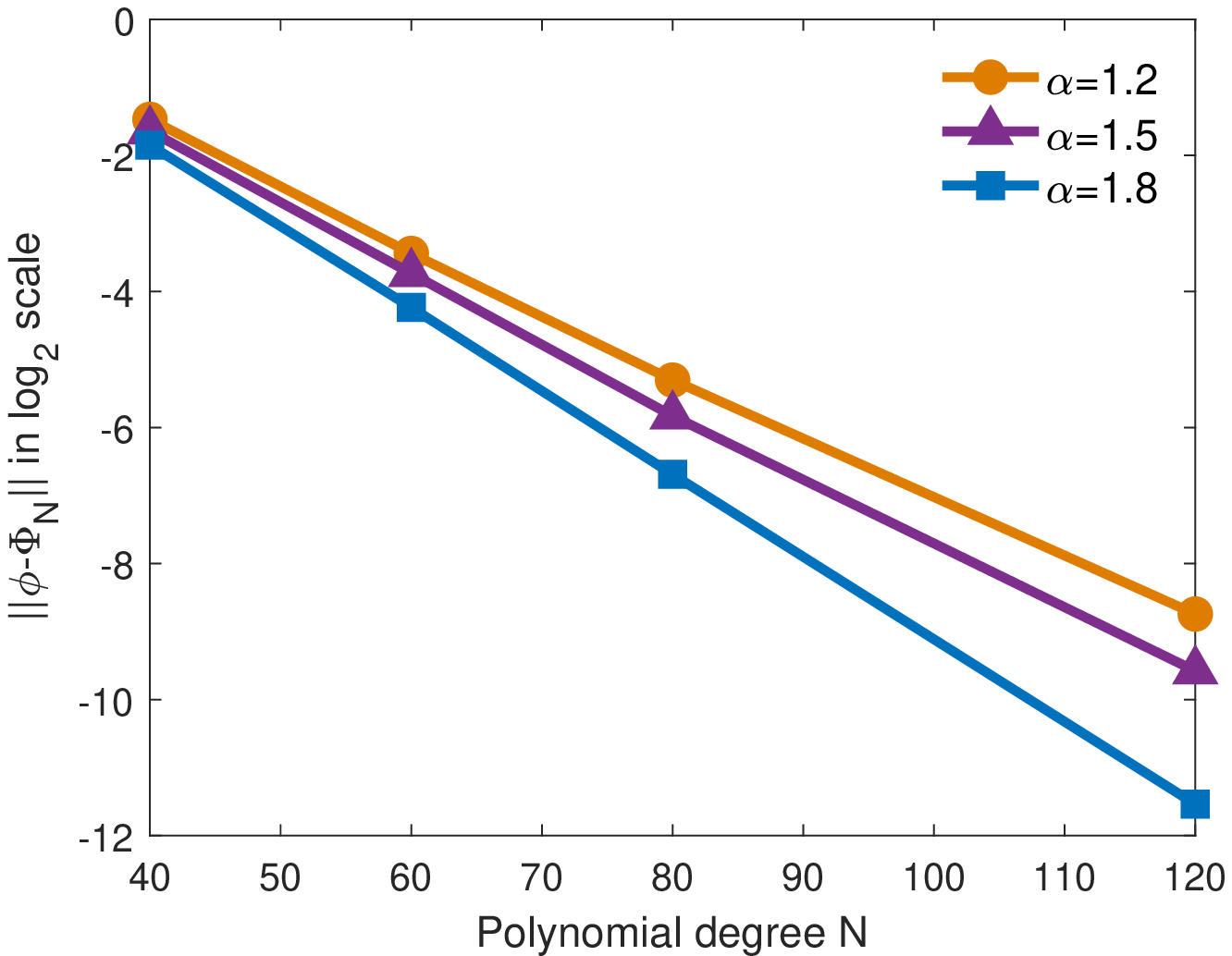}
\includegraphics[width=0.3\linewidth]{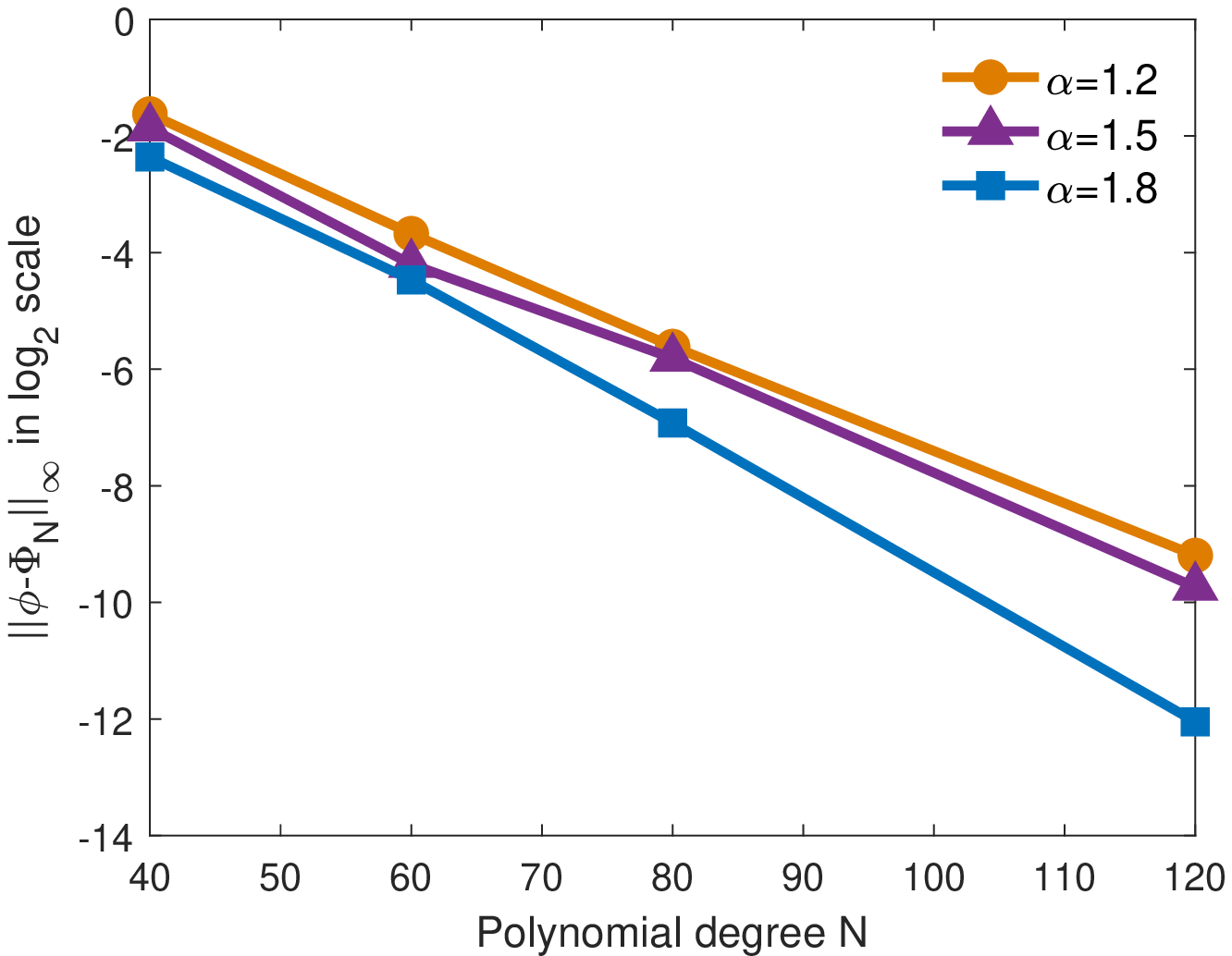}
\caption{\small{Spectral   accuracy tests of CN--SGM for Example \ref{ex1} with $\tau=1/1000$ at t=1.}}\label{numer0}
\end{figure}

\begin{table}[!htb]
\centering
\caption{\small{Temporal accuracy tests of CN--SGM for Example \ref{ex1} with $N=150$ at $t=1$.}}\label{TAB-1}
\resizebox{\textwidth}{!}{
\begin{tabular}{||cc|ccccccc||}
\hline 
$\alpha$ & $\tau$ & $\|u^{Nt}-U_N^{Nt}\|$ & \texttt{Conv.rate} & $\|\phi^{Nt}-\Phi_N^{Nt}\|$  & \texttt{Conv.rate} & $\|\phi^{Nt}-\Phi_N^{Nt}\|_\infty$ & \texttt{Conv.rate} & CPU time   \\
\hline
1.2& 0.1   & 4.5813e-03 &       -& 1.1683e-03 &       -& 12134e-03 &      - & 1.03s \\
   & 0.05  & 1.1492e-03 & 1.9951 & 2.9386e-04 & 1.9912 & 30651e-04 & 1.9851 & 1.54s \\
   & 0.025 & 2.8769e-04 & 1.9981 & 7.3579e-05 & 1.9978 & 76822e-05 & 1.9963 & 2.43s \\
   & 0.0125& 7.2149e-05 & 1.9955 & 1.8402e-05 & 1.9994 & 19218e-05 & 1.9991 & 4.00s \\
   \hline
1.5& 0.1   & 4.8761e-03 &       -& 1.0746e-03 &       -& 1.0233e-03 &      - & 1.05s\\
   & 0.05  & 1.2232e-03 & 1.9951 & 2.7114e-04 & 1.9866 & 2.6063e-04 & 1.9731 & 1.69s\\
   & 0.025 & 3.0612e-04 & 1.9985 & 6.7948e-04 & 1.9965 & 6.5450e-05 & 1.9935 & 2.69s\\
   & 0.0125& 7.6661e-05 & 1.9975 & 1.6996e-04 & 1.9992 & 1.6384e-05 & 1.9981 & 4.30s\\
   \hline
1.8& 0.1   & 5.4219e-03 &       -& 9.0221e-04 &       -& 8.5874e-04 &      - & 1.04s\\
   & 0.05  & 1.3661e-03 & 1.9888 & 2.2578e-04 & 1.9985 & 2.1491e-04 & 1.9989 & 1.69s\\
   & 0.025 & 3.4227e-04 & 1.9968 & 5.6461e-05 & 1.9996 & 5.3741e-05 & 1.9997 & 2.43s\\
   & 0.0125& 8.5705e-05 & 1.9977 & 1.4124e-05 & 1.9991 & 1.3429e-05 & 2.0012 & 3.99s\\
\hline
\end{tabular}
}
\end{table}
\begin{figure}[!htb]
\centering
\includegraphics[width=0.32\linewidth]{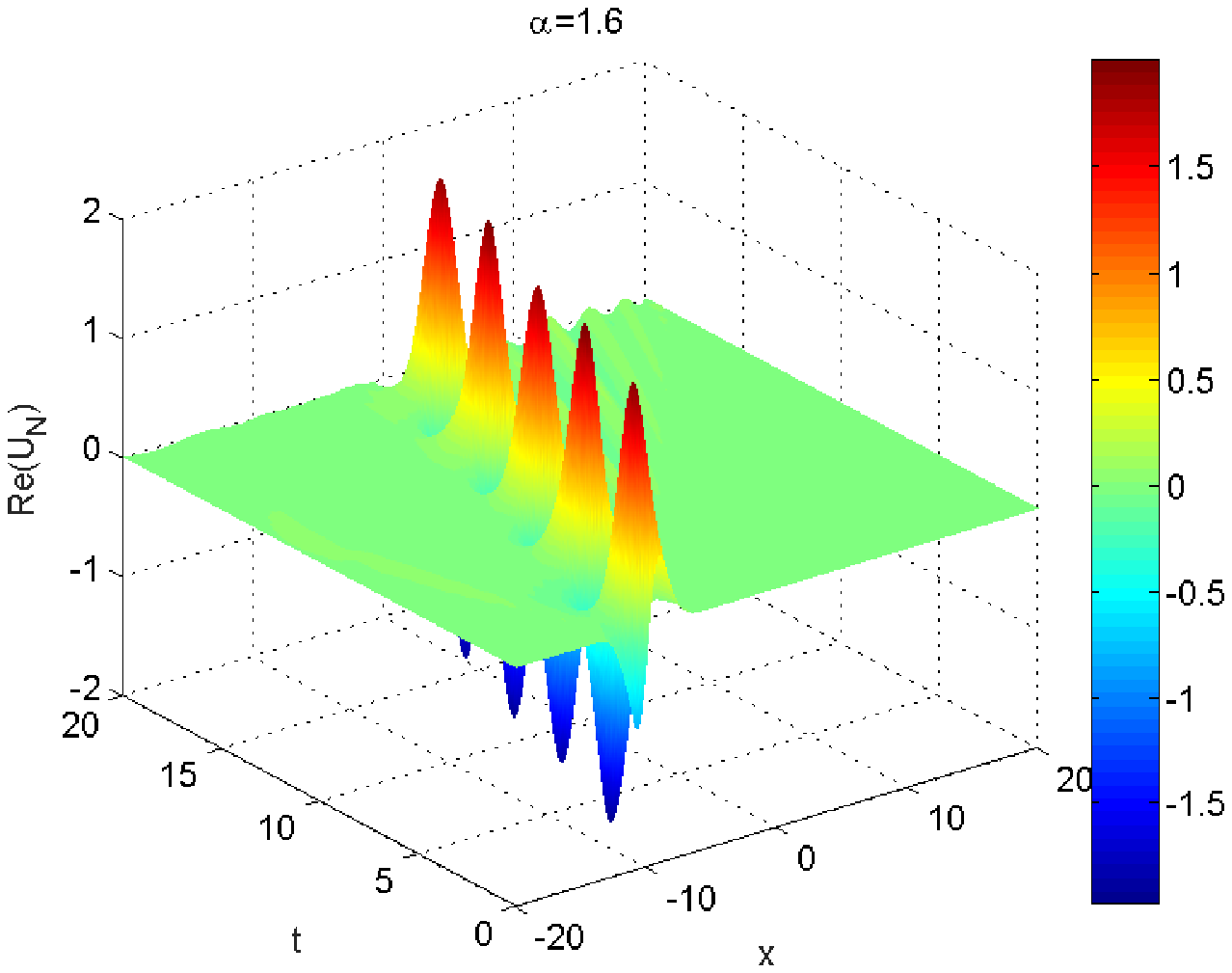}
\includegraphics[width=0.32\linewidth]{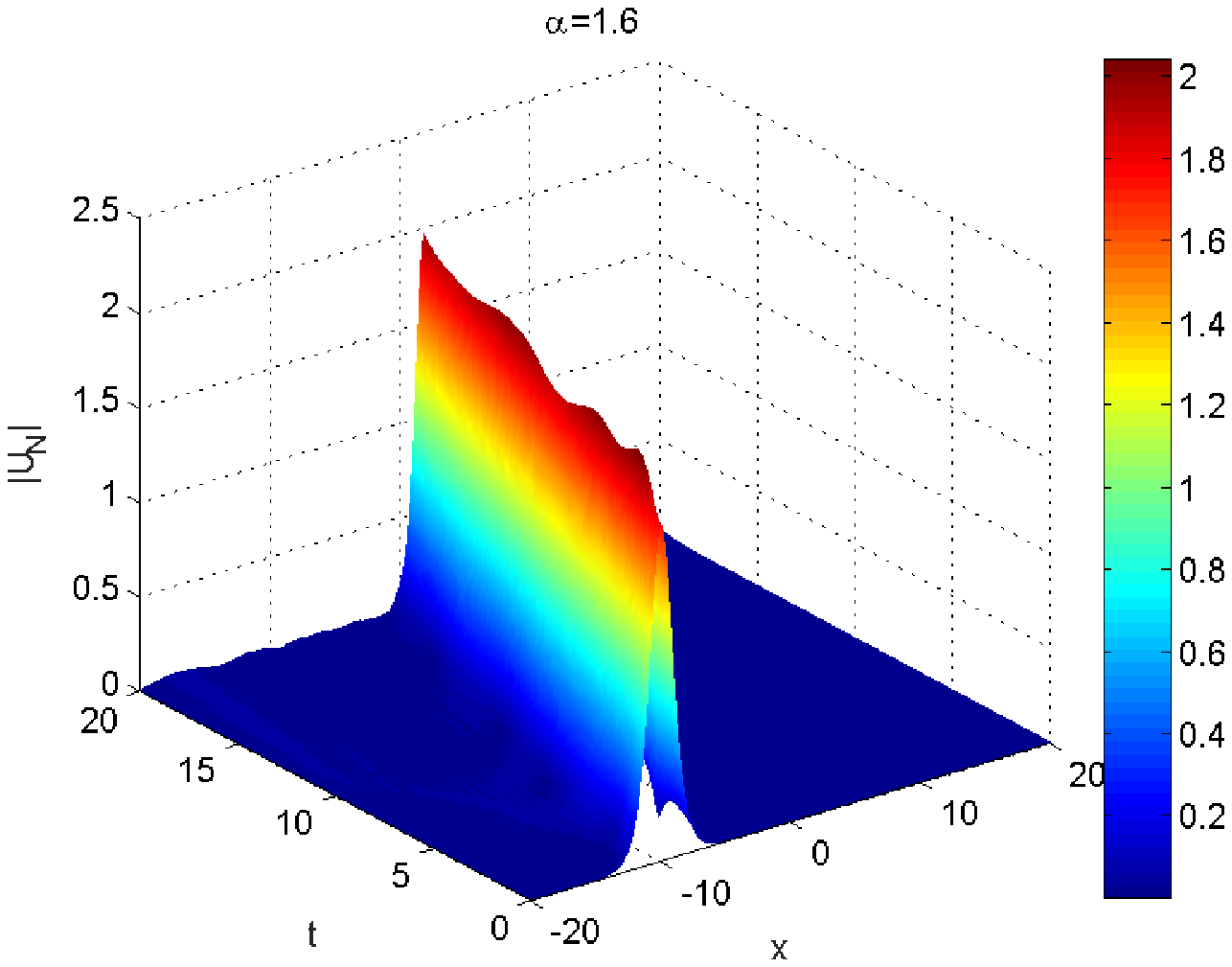}
\includegraphics[width=0.32\linewidth]{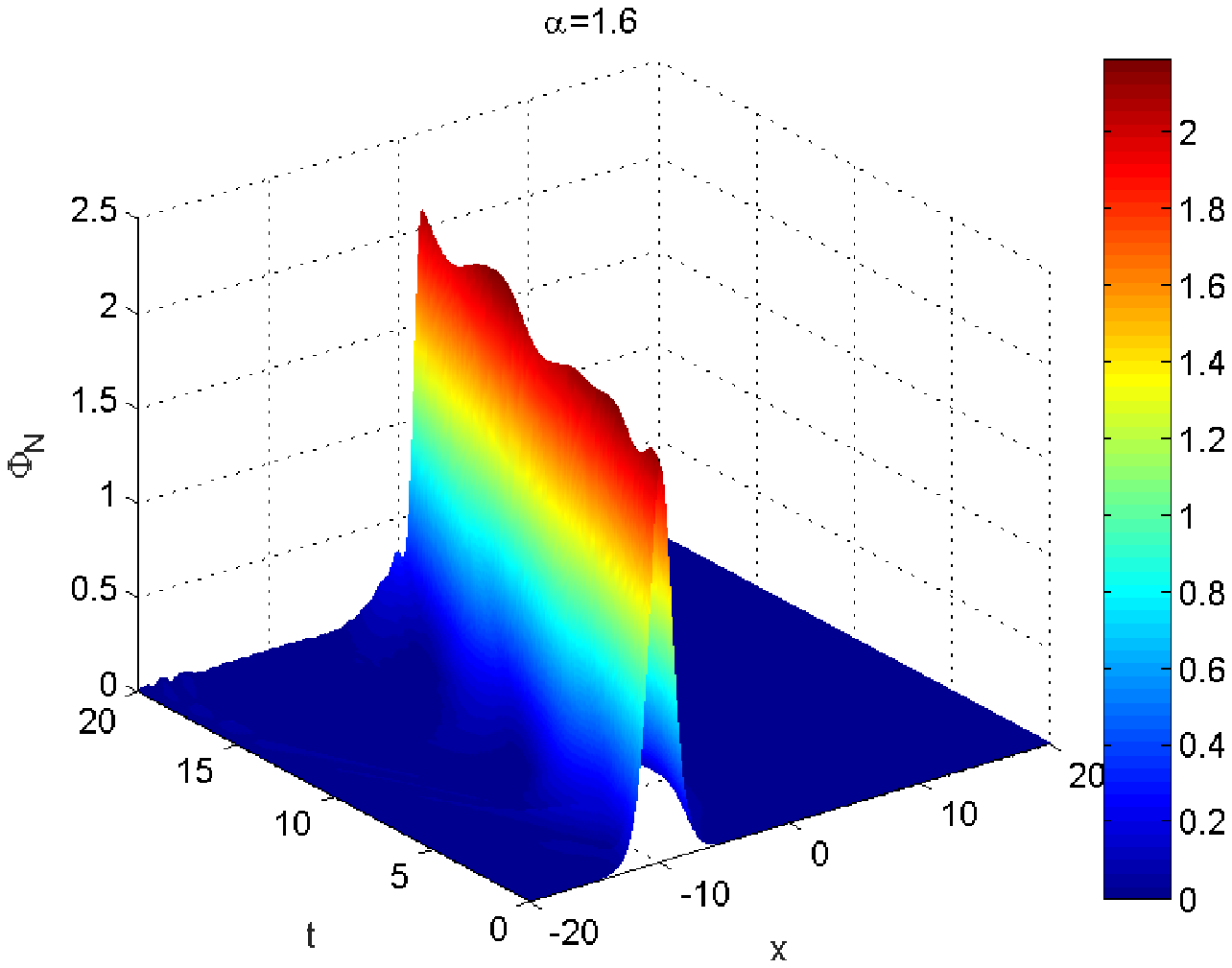}
\includegraphics[width=0.32\linewidth]{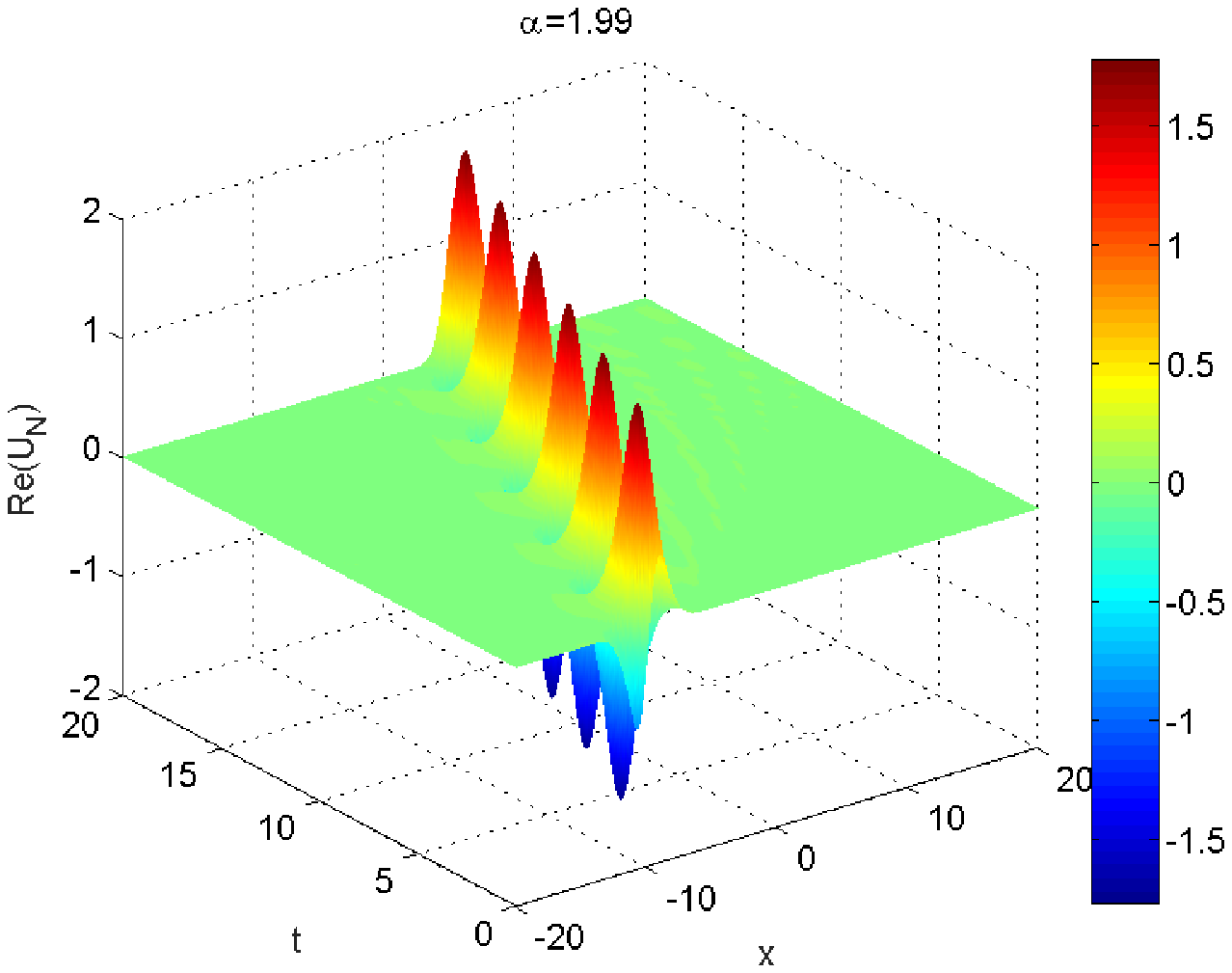}
\includegraphics[width=0.32\linewidth]{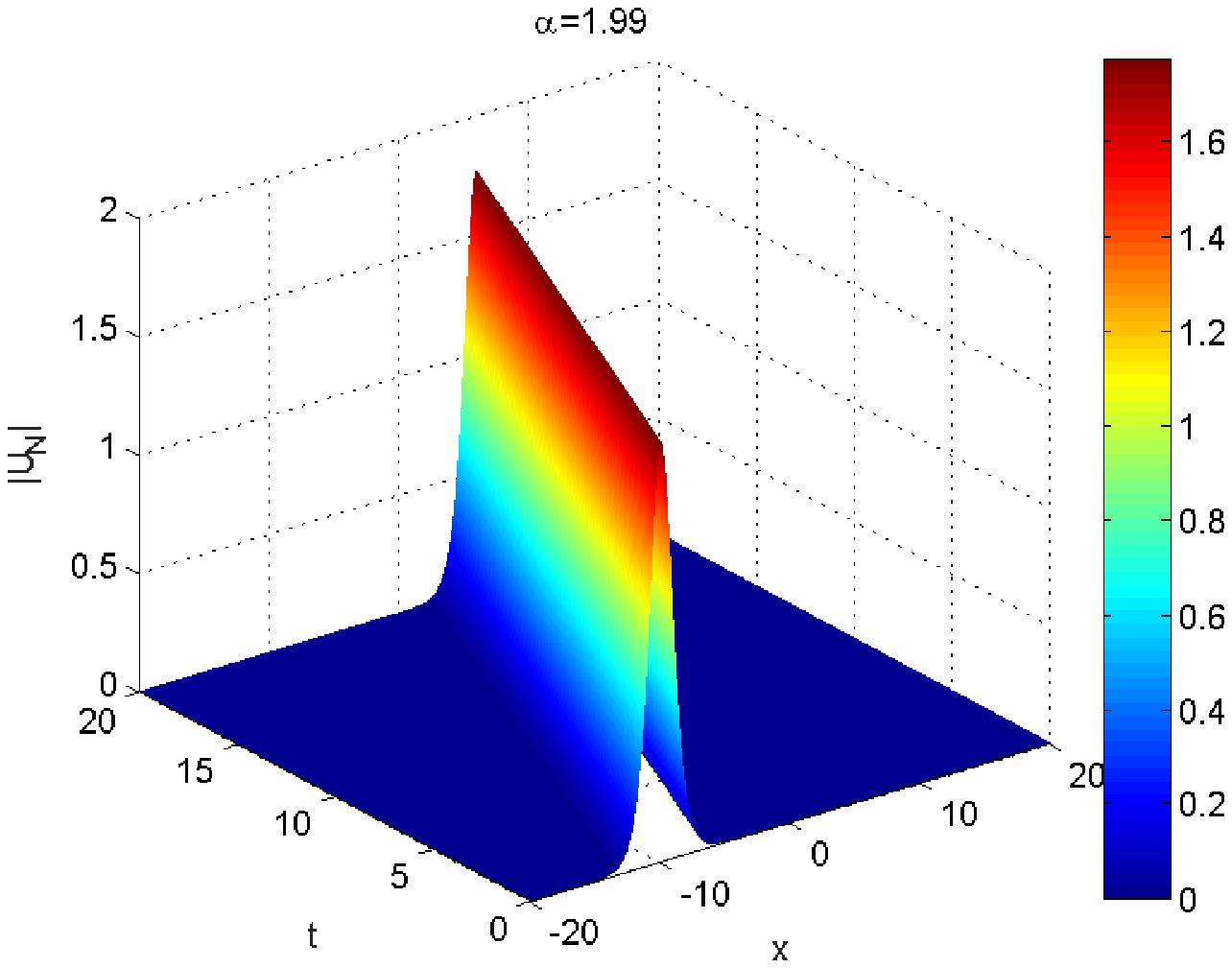}
\includegraphics[width=0.32\linewidth]{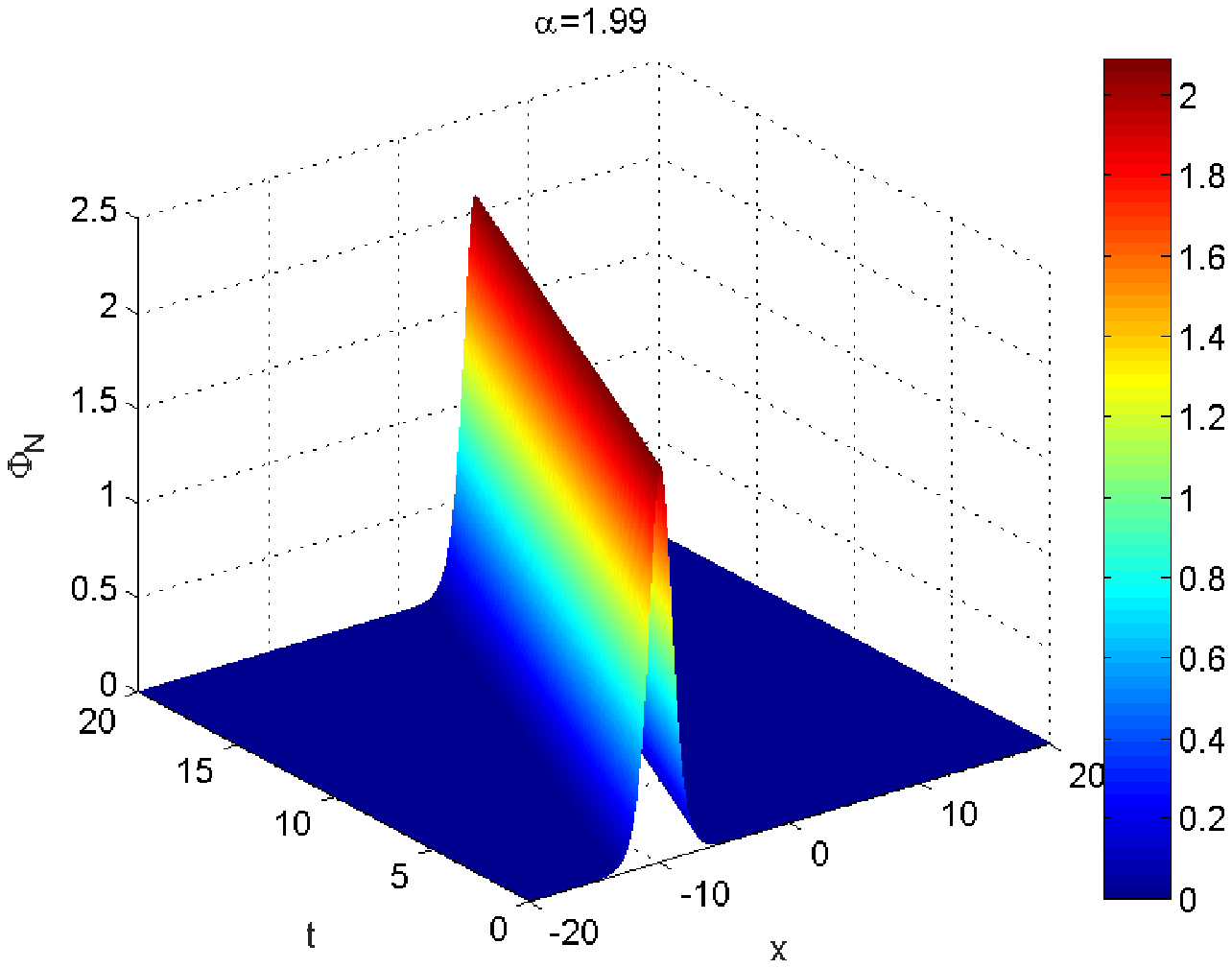}
\caption{\small{Time evolutions of the relative energy and mass of CN--SGM for Example \ref{ex1} with $\tau=1/10$ and $N=100$.}}\label{numer1}
\end{figure}

\begin{figure}[!htb]
\centering
\includegraphics[width=0.32\linewidth]{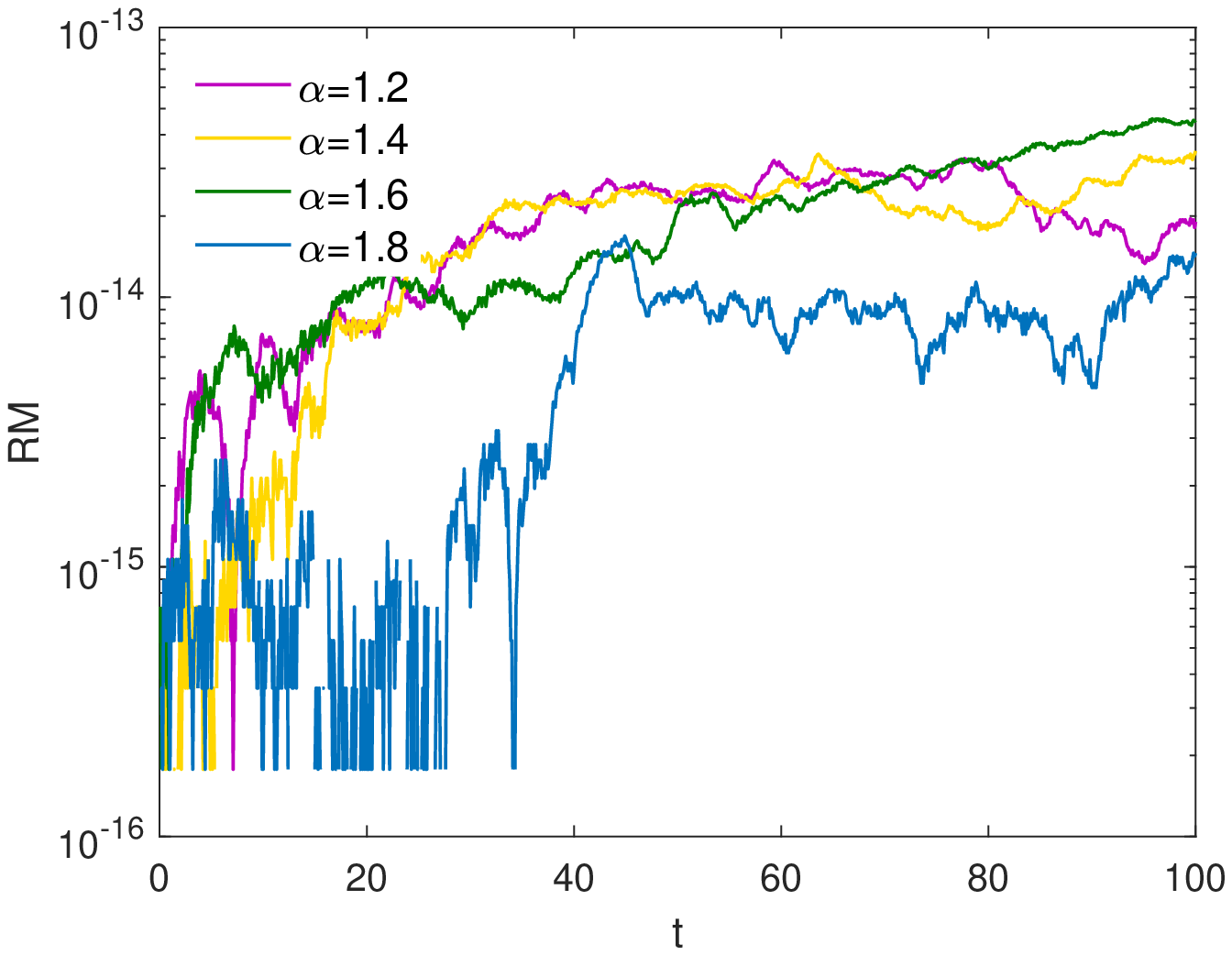}
\includegraphics[width=0.32\linewidth]{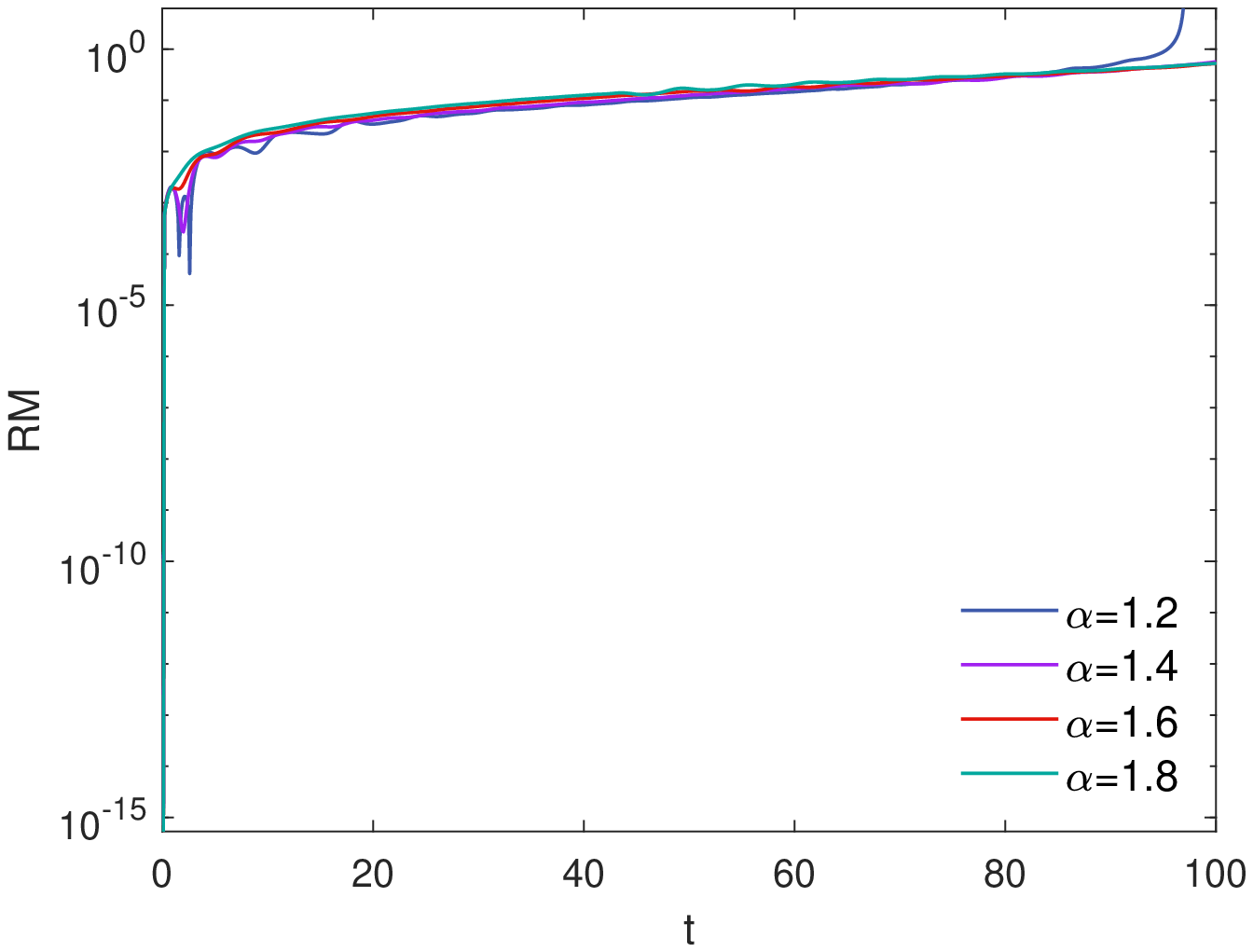}
\includegraphics[width=0.32\linewidth]{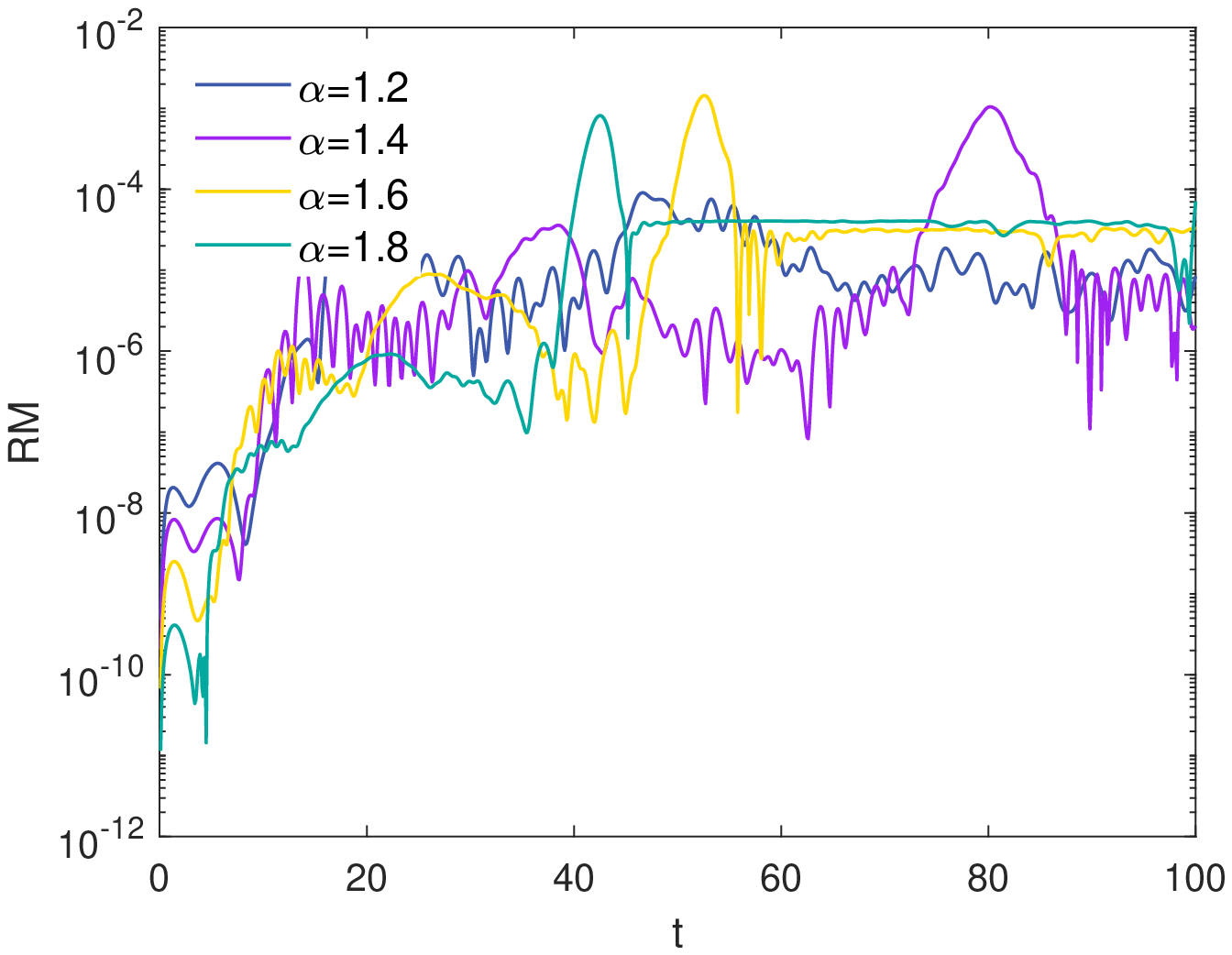}
\caption{\small{Time evolutions of the relative mass deviations of CN--SGM (Left), ESAV--SGM (Middle) and LF--FDM (Right) for Example \ref{ex1} with $\tau=1/10$ and $N=100$.}}\label{numer2}
\end{figure}

\begin{figure}[!htb]
\centering
\includegraphics[width=0.32\linewidth]{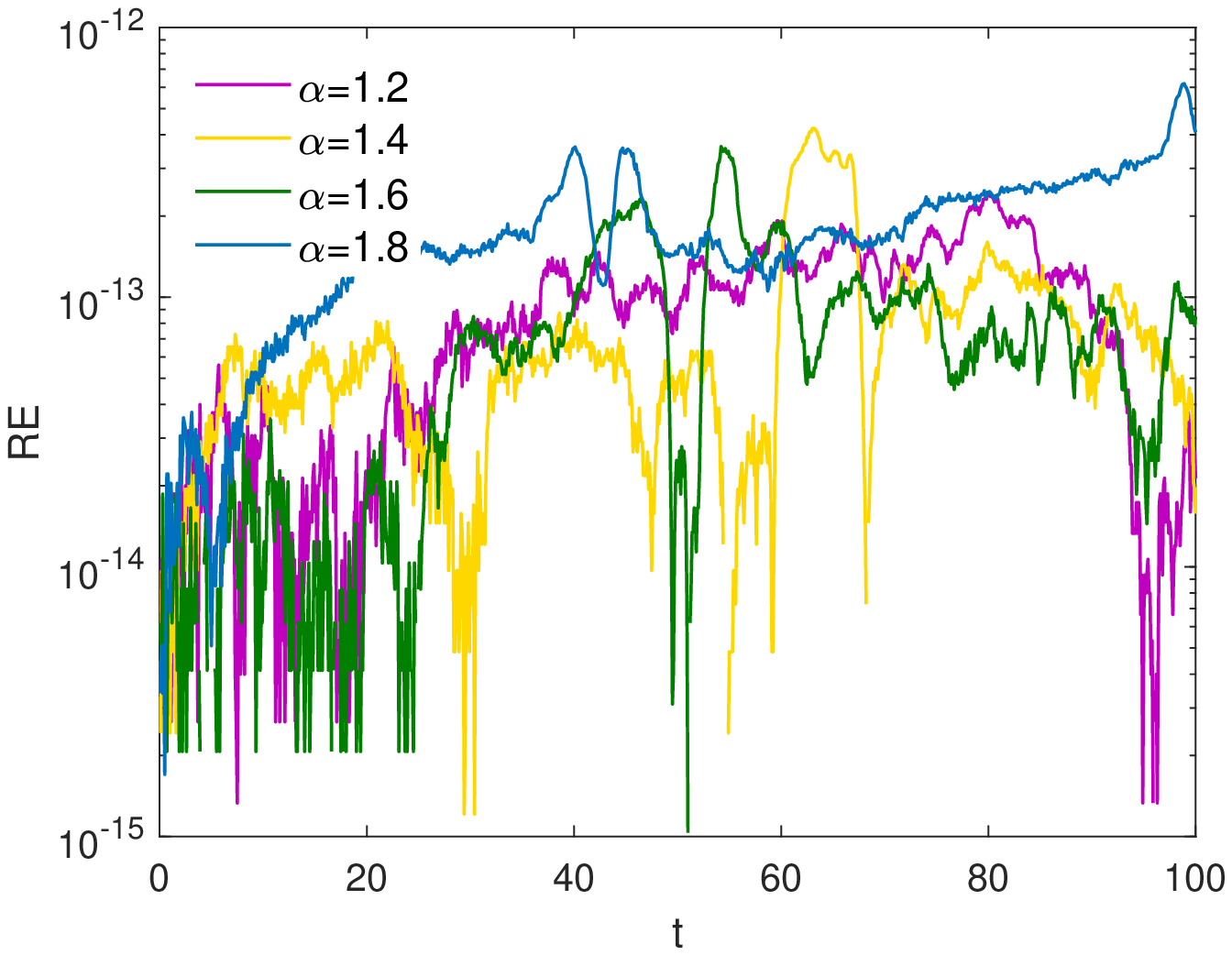}
\includegraphics[width=0.32\linewidth]{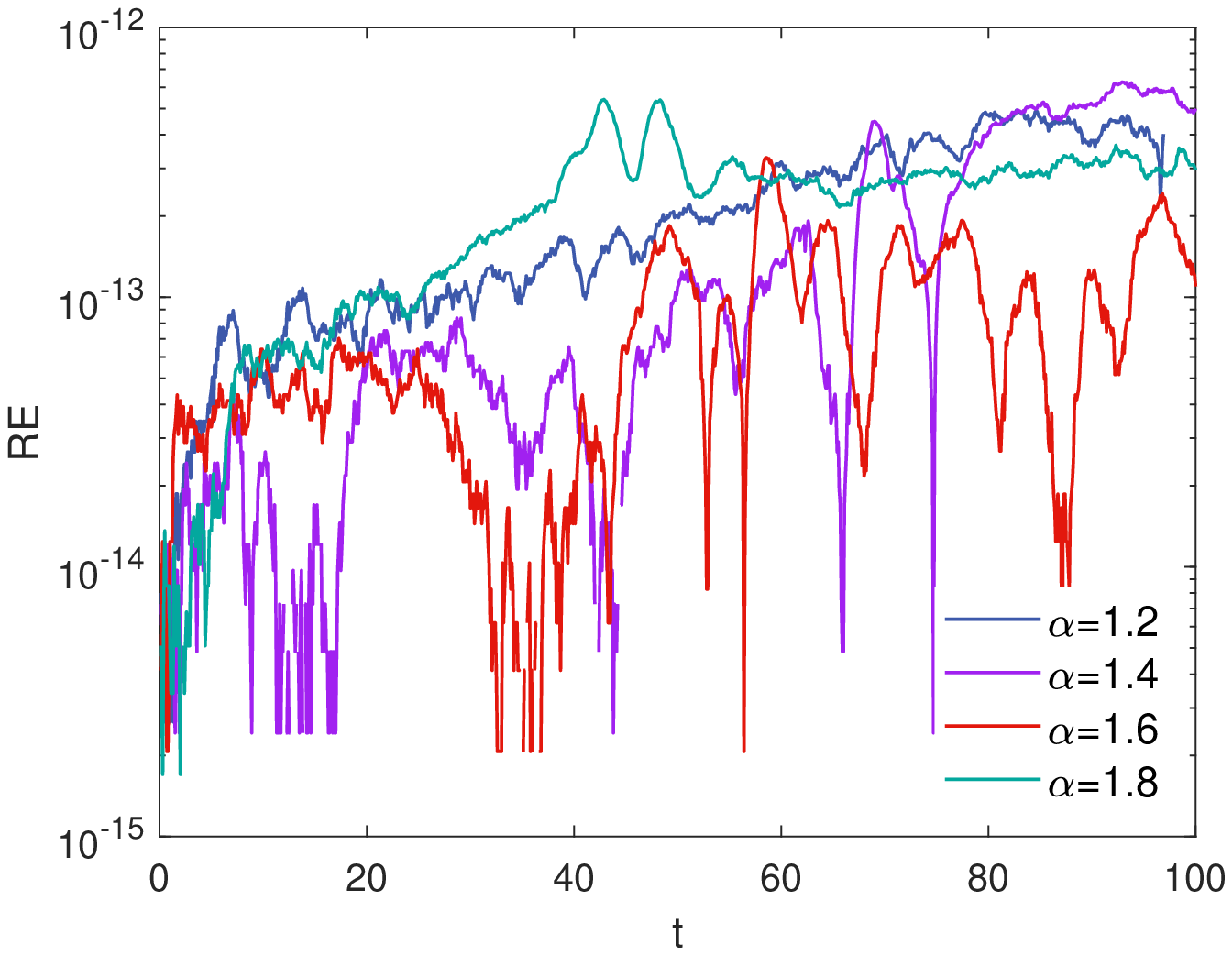}
\includegraphics[width=0.32\linewidth]{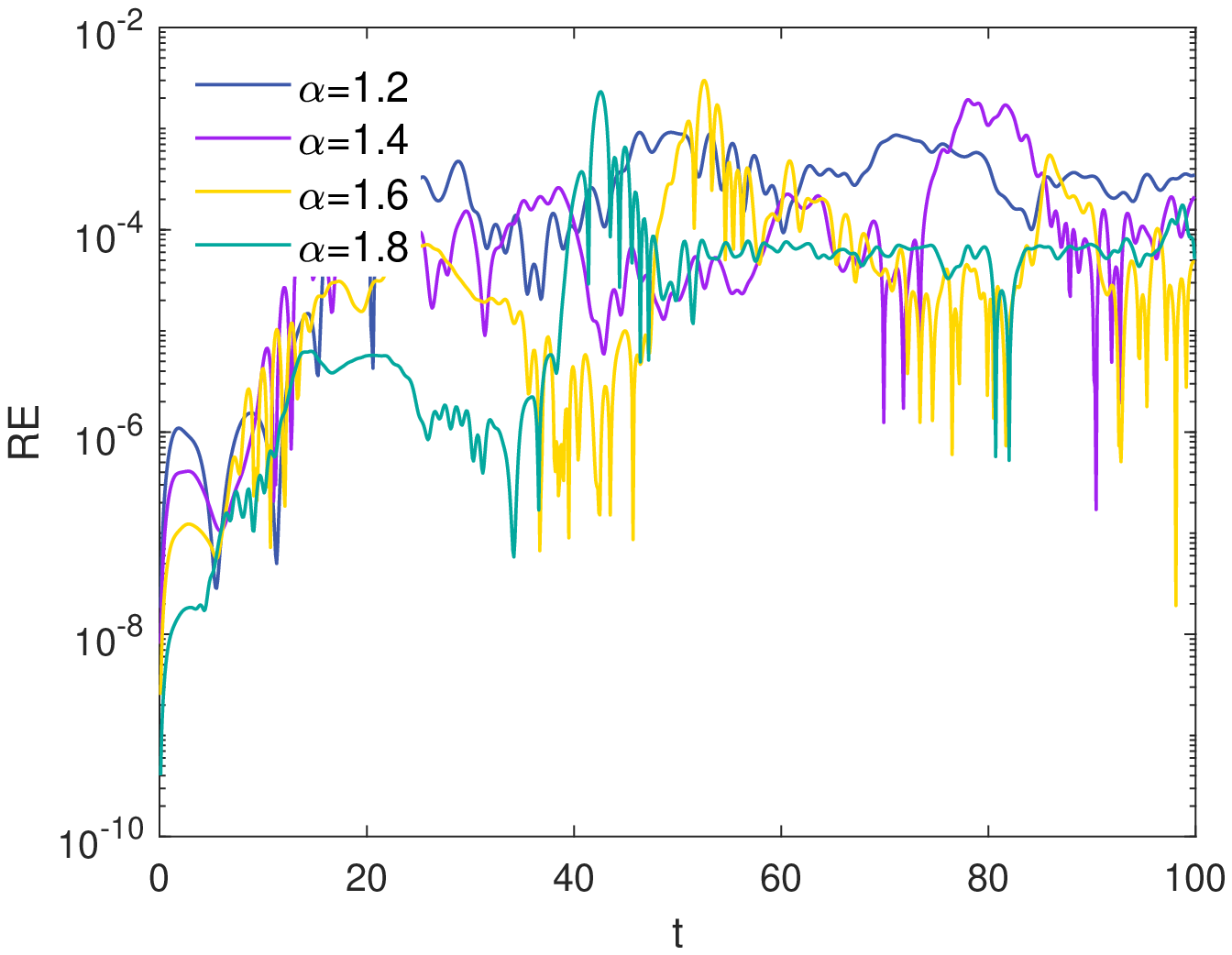}
\caption{\small{Time evolutions of the relative energy deviations of CN--SGM (Left), ESAV--SGM (Middle) and LF--FDM (Right) for Example \ref{ex1} with $\tau=1/10$ and $N=100$.}}\label{numer3}
\end{figure}

\begin{figure}[!htb]
\centering
\includegraphics[width=0.32\linewidth]{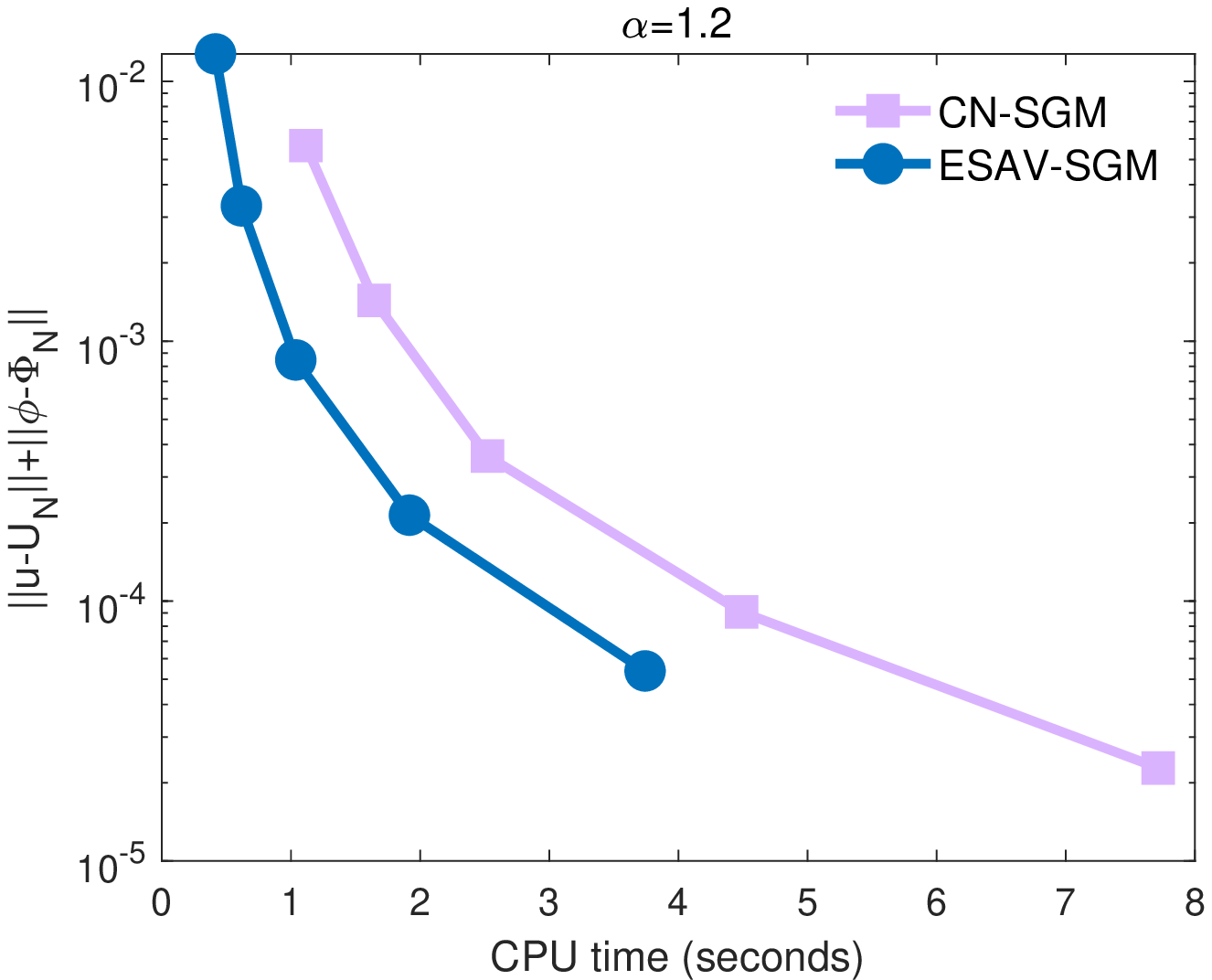}
\includegraphics[width=0.32\linewidth]{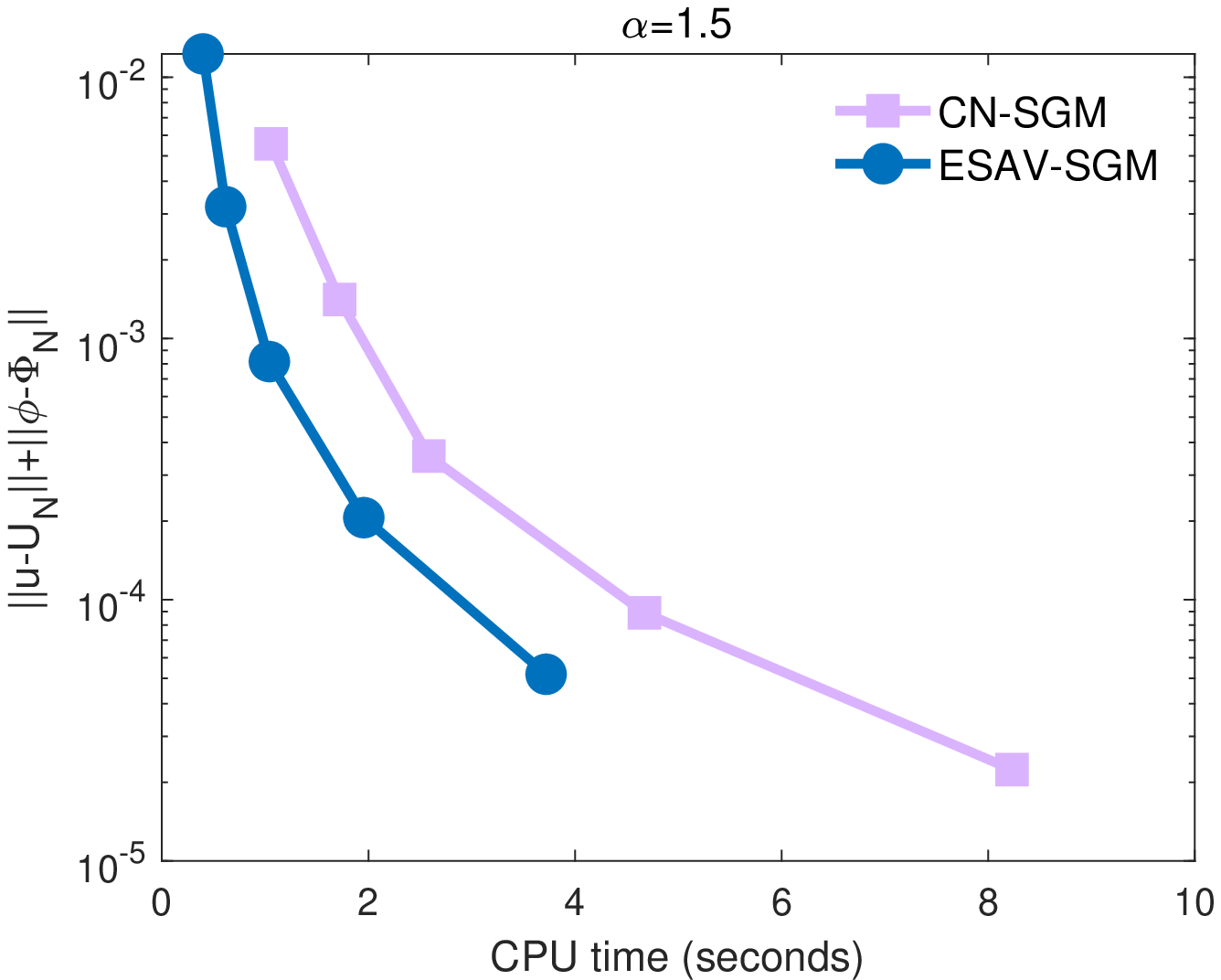}
\includegraphics[width=0.32\linewidth]{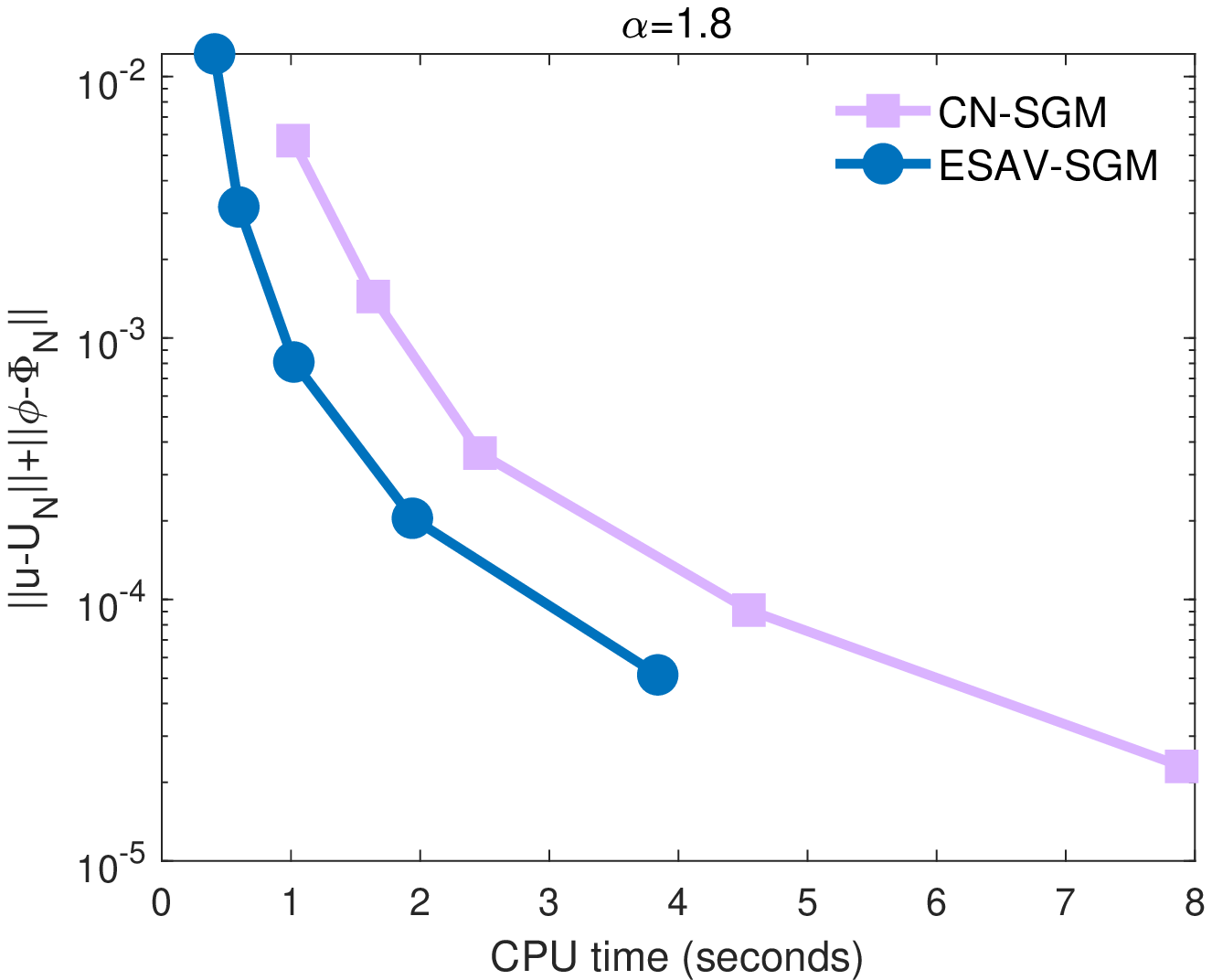}
\caption{\small{Errors $\|u^{Nt}-U_N^{Nt}\|+\|\phi^{Nt}-\Phi_N^{Nt}\|$ and the corresponding CPU time costs for Example \ref{ex1} with $N=150$ and $\tau=1/20, 1/40, 1/80, 1/160, 1/320$.}}\label{cpu1-1}
\end{figure}

\begin{ex}\label{ex2}
\emph{Consider the FNKGS system \eqref{equ1}--\eqref{equ2} with $\lambda=1$, $\kappa_1=1$, $\kappa_2\geq 0$, $\gamma=1$ and $\eta=1$ in domain $(-20,20)\times(0,T]$. We choose the following initial datum with two solitons:
\begin{align*}
&u_0(x)=u(x-p_1,0,\nu_1)+u(x-p_2,0,\nu_2),\\
&\phi_0(x)=\phi(x-p_1,0,\nu_1)+\phi(x-p_2,0,\nu_2),\\
&\phi_1(x)=\phi_t(x-p_1,0,\nu_1)+\phi_t(x-p_2,0,\nu_2),
\end{align*}
where $p_1=-p_2=-10$, $\nu_1=-\nu_2=0.8$ and $\chi_0=0$.}
\end{ex}

For the FNKGS system with Yukawa interaction, we intend to perform the effect of fractional orders on the numerical accuracy and discrete conservation laws. We apply the CN--SGM to solve Example \ref{ex2}, the convergence rate and the associated CPU time for various $\kappa_2$ are shown in Table \ref{TAB-2}--Table \ref{TAB-4}, the numerical results verify the second-order accuracy in time, which are in good agreement with theoretical results. For the fixed temporal step $\tau=1/50$ and polynomial degree $N=200$, the numerical solutions $|U_N^{n}(x)|$ and $\Phi_N^{n}(x)$ are plotted in Fig. \ref{numer4}, we observe that the collisions will occur earlier as $\alpha$ increases, and the effect of fractional order $\alpha$ on two solitons collide is distinct. By setting $\tau=1/50$, $N=150$ and $\kappa_2=0.01$, the time evolution of relative mass and relative energy errors are presented in Fig. \ref{numer5}, the numerical results exactly confirm that the proposed spectral Galerkin scheme preserves the mass and energy well in long-time computations.  In the end, one can also notice from Fig. \ref{numer6} that the discrete energy occurs negative if we choose a large $\kappa_2$, which demonstrates that the analytical solutions of the FNKGS system \eqref{equ1}--\eqref{equ2} will be blow up in finite time \cite{Boulenger2016JFA}. Also, the blow-up time will get smaller when the ESAV--SGM is adopted, which implies that the CN--SGM is more suitable than the ESAV--SGM to capture blow up.

\begin{table}[!htb]
\centering
\caption{\small{Temporal accuracy of CN--SGM for Example \ref{ex2} with $\kappa_2=0$ and $N=150$ at $t=1$.}}\label{TAB-2}
\resizebox{\textwidth}{!}{
\begin{tabular}{||cc|ccccccc||}
\hline 
$\alpha$ & $\tau$ & $\|u^{Nt}-U_N^{Nt}\|$ & \texttt{Conv.rate} & $\|\phi^{Nt}-\Phi_N^{Nt}\|$  & \texttt{Conv.rate} & $\|\phi^{Nt}-\Phi_N^{Nt}\|_\infty$ & \texttt{Conv.rate} & CPU time   \\
\hline
1.2& 0.1   & 6.4901e-03 &       -& 1.6517e-03 &       -& 1.2104e-03 &      - & 1.15s \\
   & 0.05  & 1.6281e-03 & 1.9950 & 4.1544e-04 & 1.9912 & 3.0573e-04 & 1.9851 & 1.73s \\
   & 0.025 & 4.0760e-04 & 1.9980 & 1.0402e-04 & 1.9978 & 7.6629e-05 & 1.9963 & 2.44s \\
   & 0.0125& 1.0228e-04 & 1.9946 & 2.6016e-05 & 1.9994 & 1.9169e-05 & 1.9991 & 4.15s \\
   \hline
1.5& 0.1   & 6.8998e-03 &       -& 1.5194e-03 &       -& 1.0228e-03 &      - & 1.03s \\
   & 0.05  & 1.7308e-03 & 1.9951 & 3.8338e-04 & 1.9866 & 2.6053e-04 & 1.9729 & 1.60s \\
   & 0.025 & 4.3318e-04 & 1.9984 & 9.6079e-05 & 1.9965 & 6.5416e-05 & 1.9937 & 2.42s \\
   & 0.0125& 1.0851e-04 & 1.9972 & 2.4032e-05 & 1.9993 & 1.6379e-05 & 1.9978 & 4.10s \\
   \hline
1.8& 0.1   & 7.6685e-03 &       -& 1.2759e-03 &       -& 8.5880e-04 &      - & 1.01s \\
   & 0.05  & 1.9321e-03 & 1.9888 & 3.1930e-04 & 1.9985 & 2.1494e-04 & 1.9984 & 1.71s \\
   & 0.025 & 4.8409e-04 & 1.9968 & 7.9848e-05 & 1.9996 & 5.3753e-05 & 1.9995 & 2.48s \\
   & 0.0125& 1.2122e-04 & 1.9976 & 1.9986e-05 & 1.9983 & 1.3424e-05 & 2.0015 & 4.18s \\
\hline
\end{tabular}
}
\end{table}

\begin{table}[!htb]
\centering
\caption{\small{Temporal accuracy of CN--SGM for Example \ref{ex2} with $\kappa_2=0.01$ and $N=150$ at $t=1$.}}\label{TAB-3}
\resizebox{\textwidth}{!}{
\begin{tabular}{||cc|ccccccc||}
\hline 
$\alpha$ & $\tau$ & $\|u^{Nt}-U_N^{Nt}\|$ & \texttt{Conv.rate} & $\|\phi^{Nt}-\Phi_N^{Nt}\|$  & \texttt{Conv.rate} & $\|\phi^{Nt}-\Phi_N^{Nt}\|_\infty$ & \texttt{Conv.rate} & CPU time   \\
\hline
1.2& 0.1   & 7.6888e-03 &       -& 1.6980e-03 &       -& 1.2975e-03 &      - & 1.09s \\
   & 0.05  & 1.9294e-03 & 1.9946 & 4.2726e-04 & 1.9906 & 3.2780e-04 & 1.9848 & 1.62s \\
   & 0.025 & 4.8297e-04 & 1.9981 & 1.0699e-04 & 1.9976 & 8.2164e-05 & 1.9962 & 2.52s \\
   & 0.0125& 1.2105e-04 & 1.9963 & 2.6759e-05 & 1.9994 & 2.0555e-05 & 1.9990 & 4.09s \\
   \hline
1.5& 0.1   & 8.1416e-04 &       -& 1.6199e-03 &       -& 1.2303e-03 &      - & 1.05s \\
   & 0.05  & 2.0423e-04 & 1.9951 & 4.0976e-04 & 1.9830 & 3.1367e-04 & 1.9716 & 1.64s \\
   & 0.025 & 5.1110e-04 & 1.9985 & 1.0275e-04 & 1.9956 & 7.8783e-05 & 1.9933 & 2.62s \\
   & 0.0125& 1.2794e-04 & 1.9981 & 2.5706e-05 & 1.9990 & 1.9725e-05 & 1.9978 & 4.18s \\
   \hline
1.8& 0.1   & 9.2848e-03 &       -& 1.0966e-03 &       -& 6.9719e-04 &      - & 1.17s \\
   & 0.05  & 2.3492e-03 & 1.9827 & 2.7501e-04 & 1.9955 & 1.7479e-04 & 1.9959 & 1.55s \\
   & 0.025 & 5.8940e-04 & 1.9948 & 6.8813e-05 & 1.9987 & 4.3738e-05 & 1.9987 & 2.56s \\
   & 0.0125& 1.4756e-04 & 1.9979 & 1.7233e-05 & 1.9975 & 1.0917e-05 & 2.0024 & 4.04s \\
\hline
\end{tabular}
}
\end{table}

\begin{table}[!htb]
\centering
\caption{\small{Temporal accuracy of CN--SGM for Example \ref{ex2} with $\kappa_2=0.1$ and $N=150$ at $t=1$.}}\label{TAB-4}
\resizebox{\textwidth}{!}{
\begin{tabular}{||cc|ccccccc||}
\hline 
$\alpha$ & $\tau$ & $\|u^{Nt}-U_N^{Nt}\|$ & \texttt{Conv.rate} & $\|\phi^{Nt}-\Phi_N^{Nt}\|$  & \texttt{Conv.rate} & $\|\phi^{Nt}-\Phi_N^{Nt}\|_\infty$ & \texttt{Conv.rate} & CPU time   \\
\hline
1.2& 0.1   & 5.0647e-02 &       -& 3.9455e-03 &       -& 3.7289e-03 &      - & 1.78s \\
   & 0.05  & 1.2893e-02 & 1.9739 & 1.0040e-03 & 1.9745 & 9.5011e-04 & 1.9726 & 2.19s \\
   & 0.025 & 3.2382e-03 & 1.9933 & 2.5212e-04 & 1.9935 & 2.3866e-04 & 1.9931 & 3.34s \\
   & 0.0125& 8.1054e-04 & 1.9982 & 6.3100e-05 & 1.9984 & 5.9737e-05 & 1.9983 & 5.28s \\
   \hline
1.5& 0.1   & 5.4014e-02 &       -& 6.6000e-03 &       -& 5.9641e-03 &      - & 1.76s \\
   & 0.05  & 1.3701e-02 & 1.9790 & 1.6773e-03 & 1.9763 & 1.5190e-03 & 1.9731 & 2.22s \\
   & 0.025 & 3.4379e-03 & 1.9947 & 4.2105e-04 & 1.9941 & 3.8152e-04 & 1.9933 & 3.41s \\
   & 0.0125& 8.6037e-04 & 1.9985 & 1.0537e-04 & 1.9985 & 9.5486e-05 & 1.9984 & 5.46s \\
   \hline
1.8& 0.1   & 5.1174e-02 &       -& 6.3376e-03 &       -& 5.2555e-03 &      - & 1.67s \\
   & 0.05  & 1.2977e-02 & 1.9795 & 1.6217e-03 & 1.9664 & 1.3529e-03 & 1.9578 & 2.12s \\
   & 0.025 & 3.2563e-03 & 1.9946 & 4.0782e-04 & 1.9915 & 3.4072e-04 & 1.9894 & 3.22s \\
   & 0.0125& 8.1506e-04 & 1.9982 & 1.0211e-04 & 1.9978 & 8.5323e-05 & 1.9976 & 5.18s \\
\hline
\end{tabular}
}
\end{table}

\begin{figure}[!htb]
\centering
\includegraphics[width=0.32\linewidth]{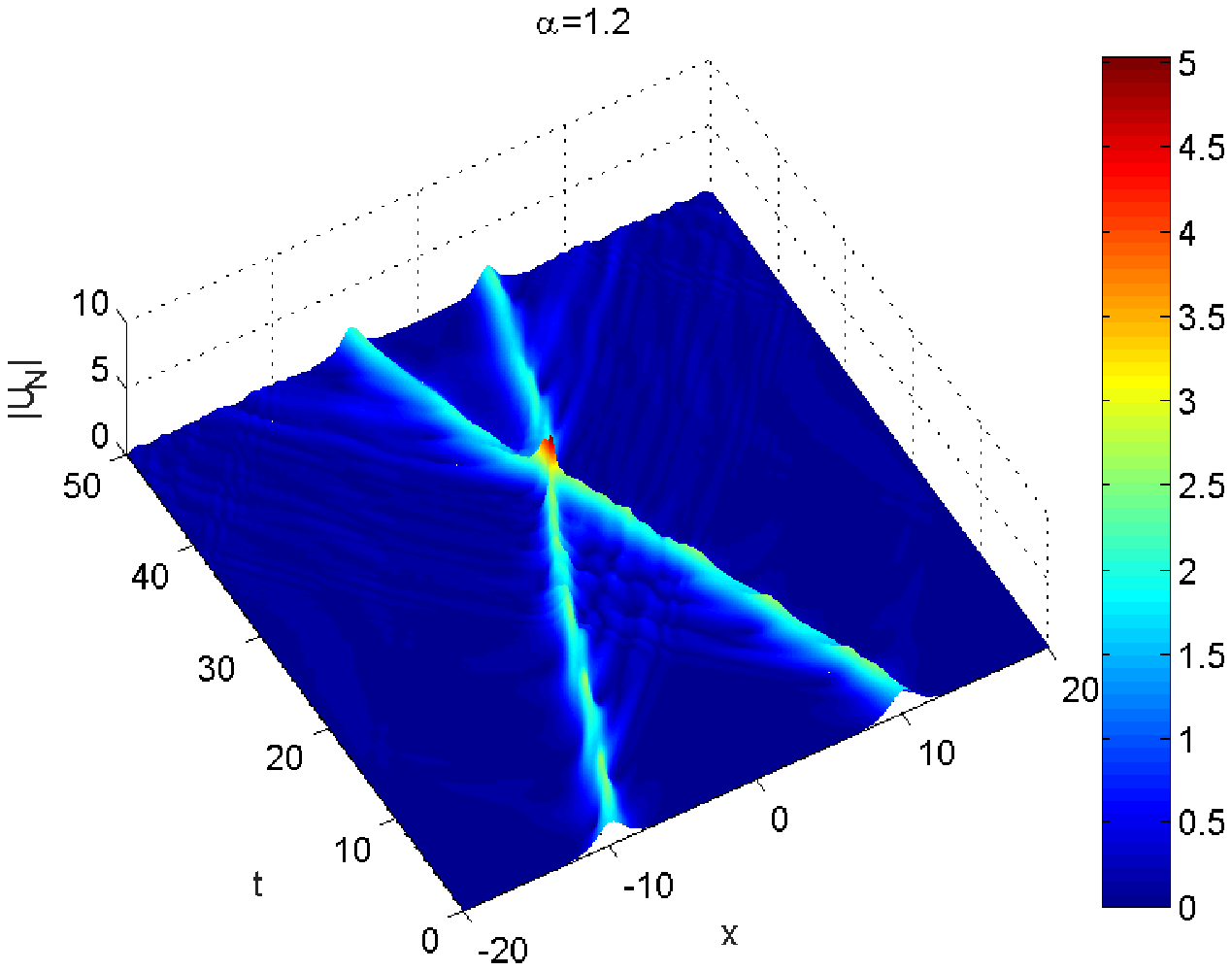}
\includegraphics[width=0.32\linewidth]{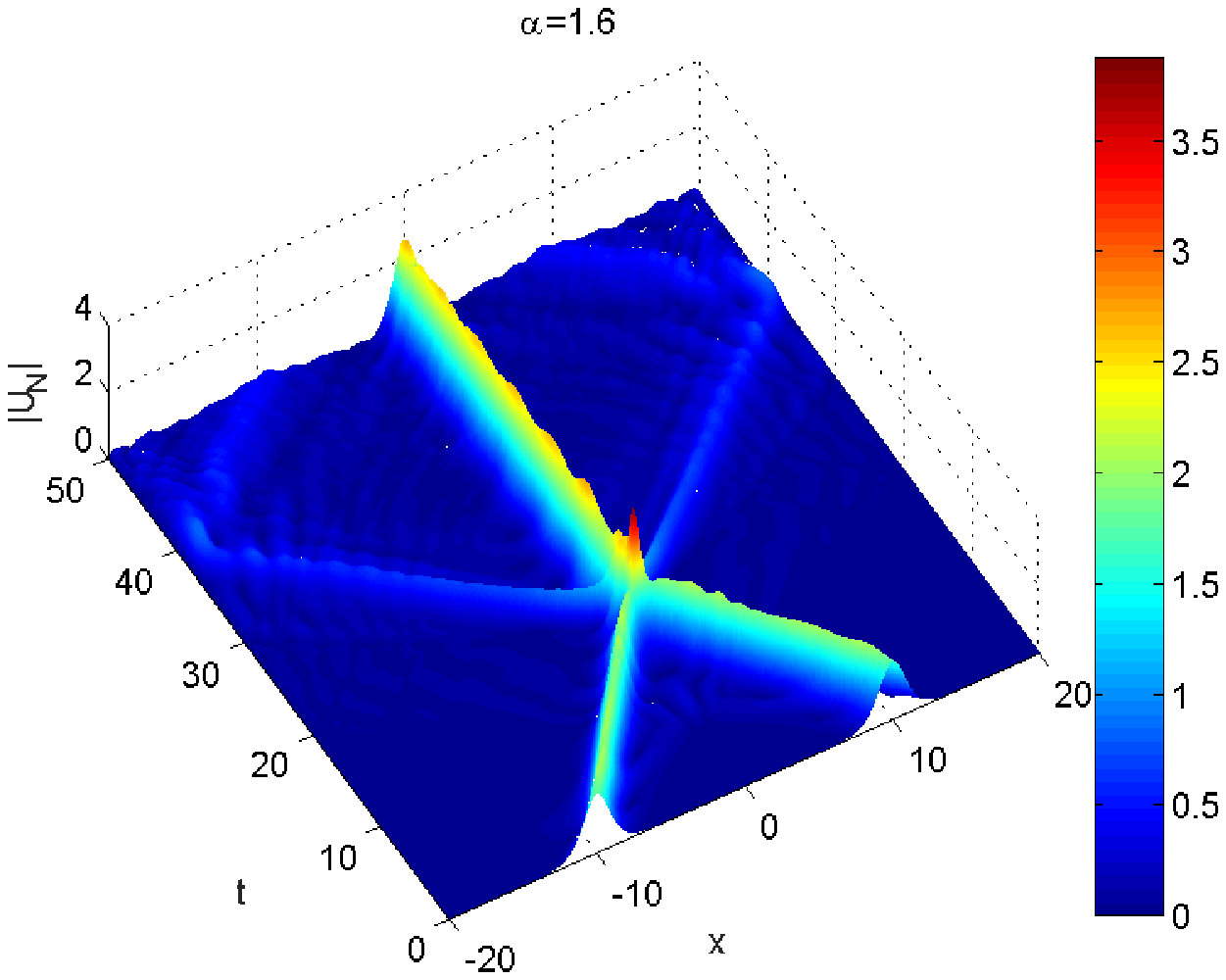}
\includegraphics[width=0.32\linewidth]{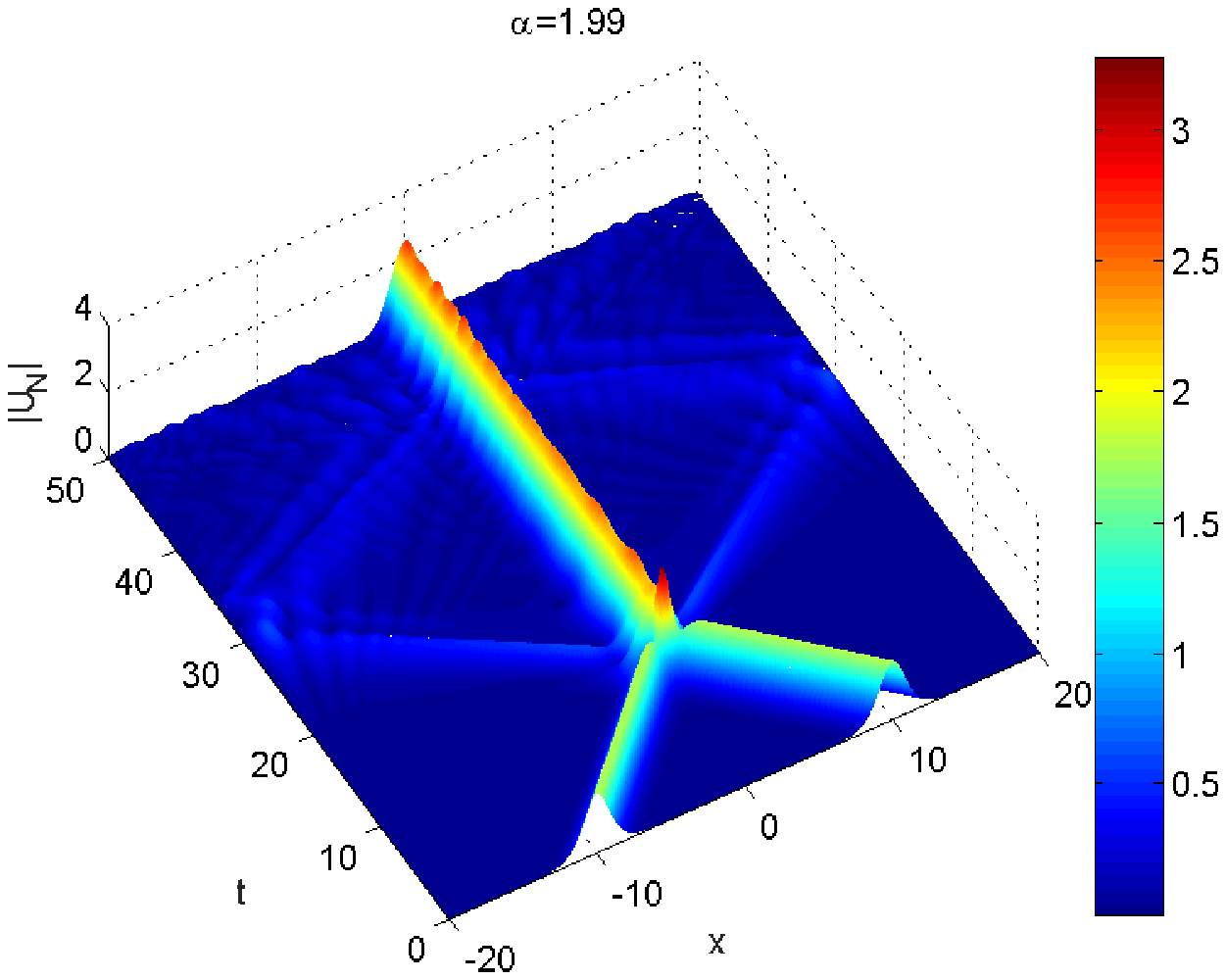}
\includegraphics[width=0.32\linewidth]{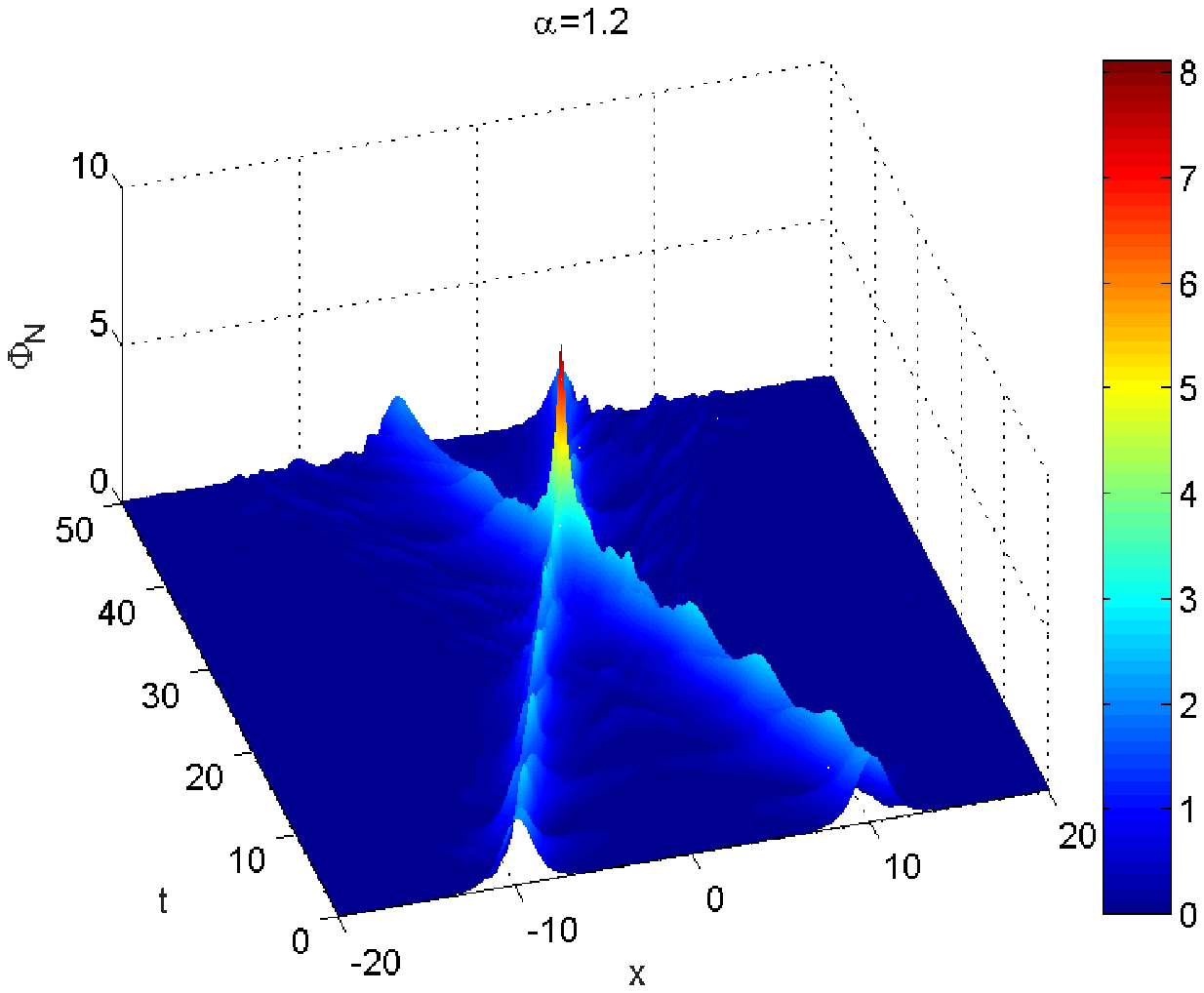}
\includegraphics[width=0.32\linewidth]{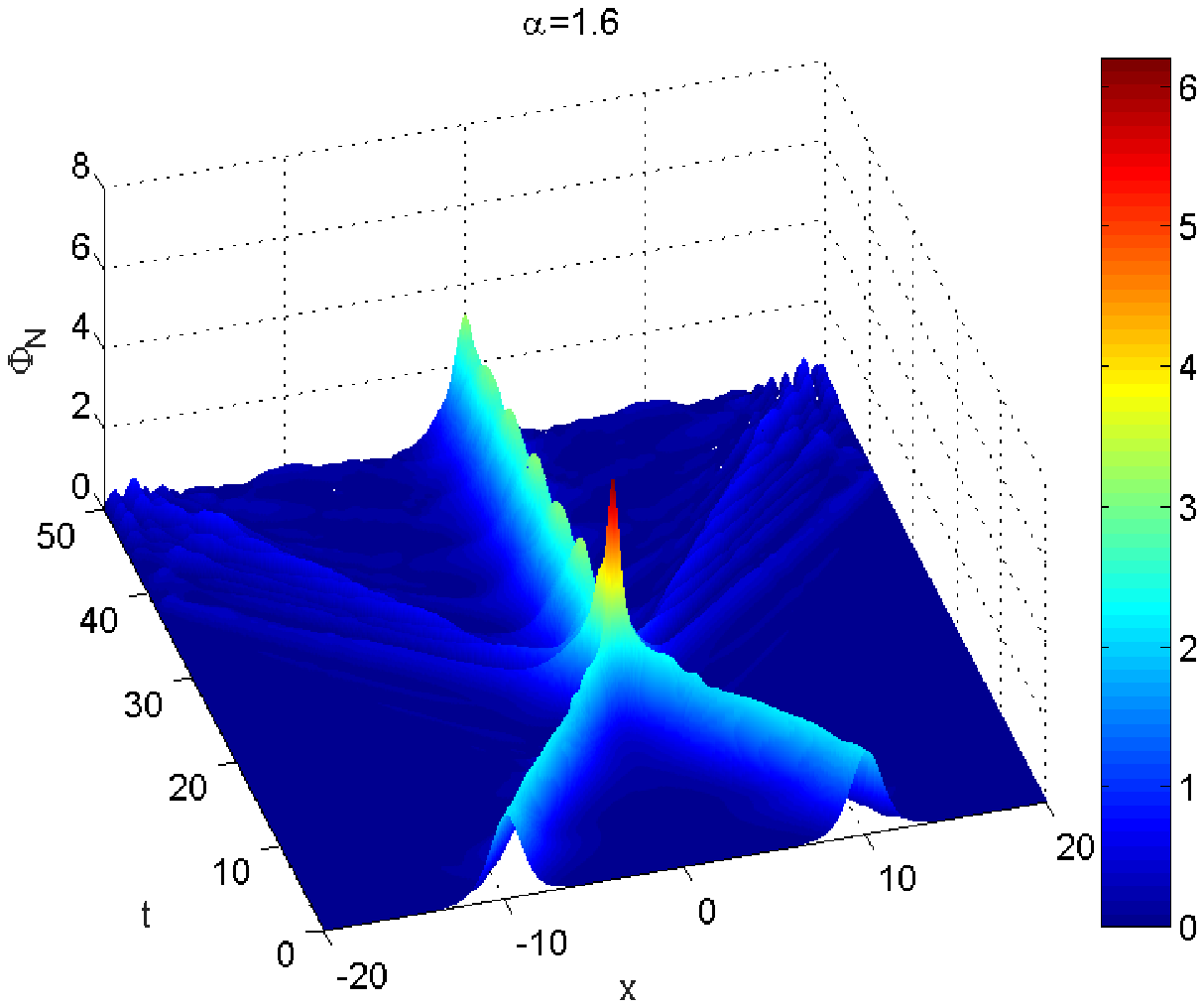}
\includegraphics[width=0.32\linewidth]{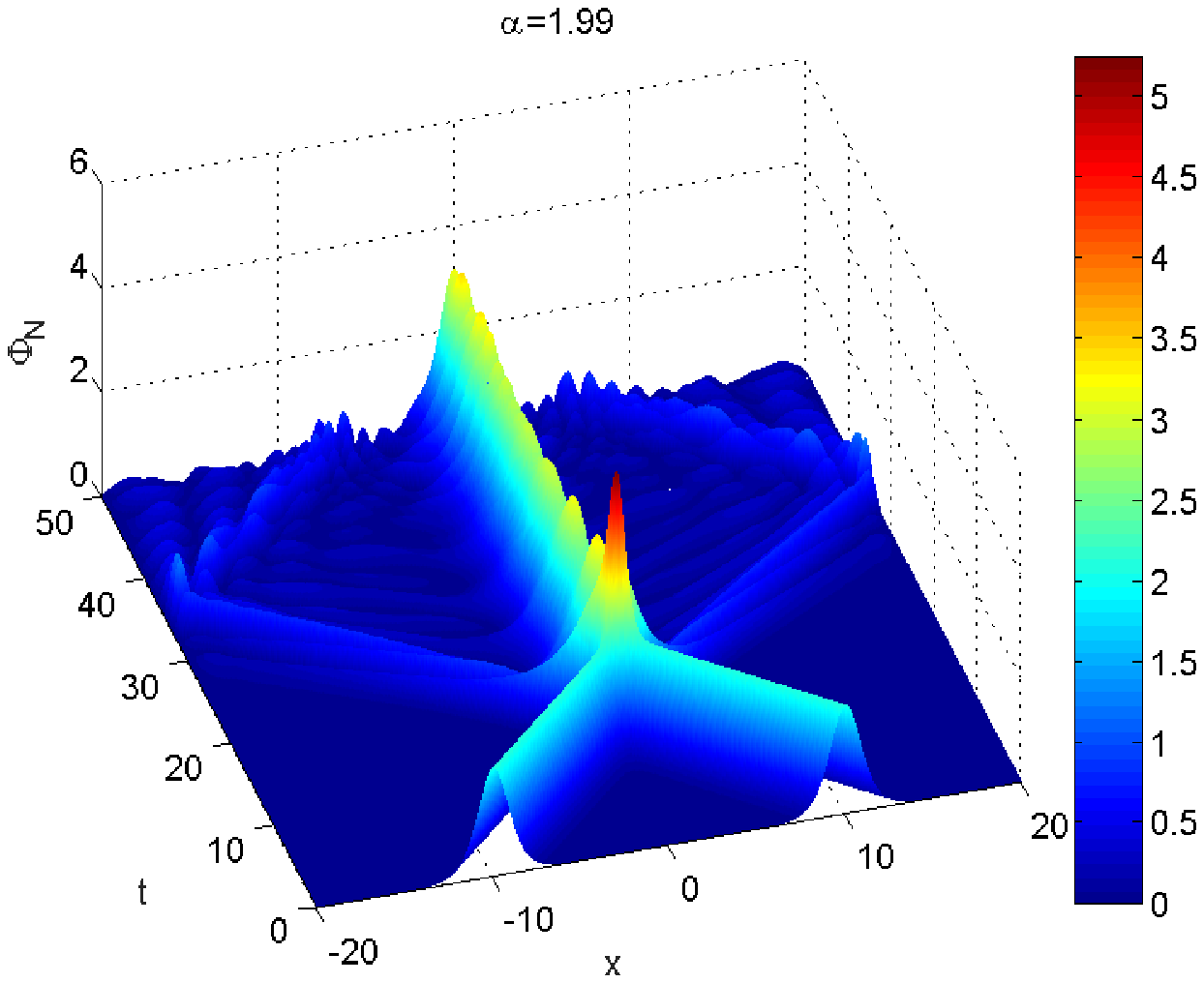}
\caption{\small{Time evolution of numerical solutions of CN--SGM for Example \ref{ex2} with $\tau=1/50$ and $N=200$ for $\kappa_2=0$.}}\label{numer4}
\end{figure}

\begin{figure}[!htb]
\centering
\includegraphics[width=0.49\linewidth]{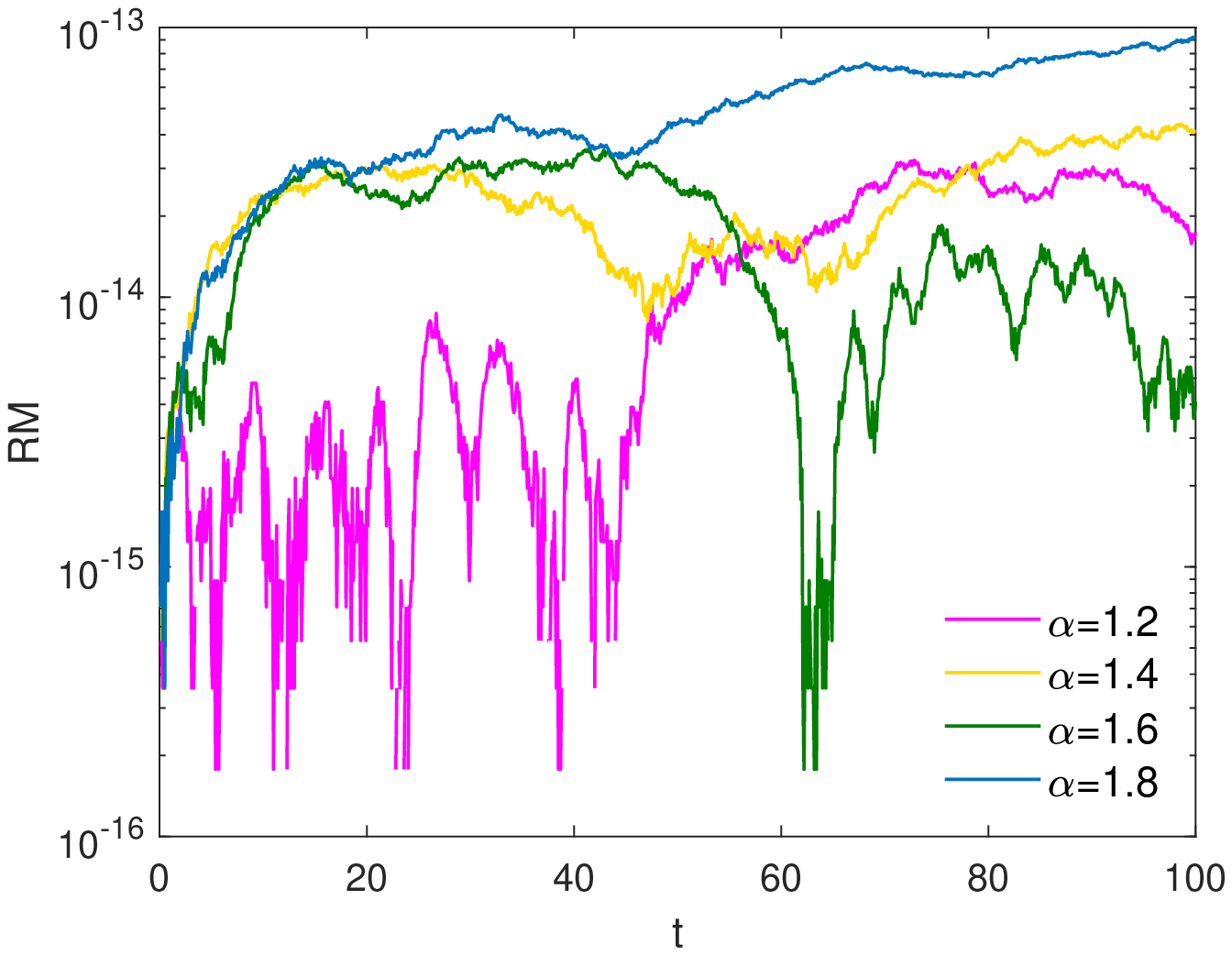}
\includegraphics[width=0.49\linewidth]{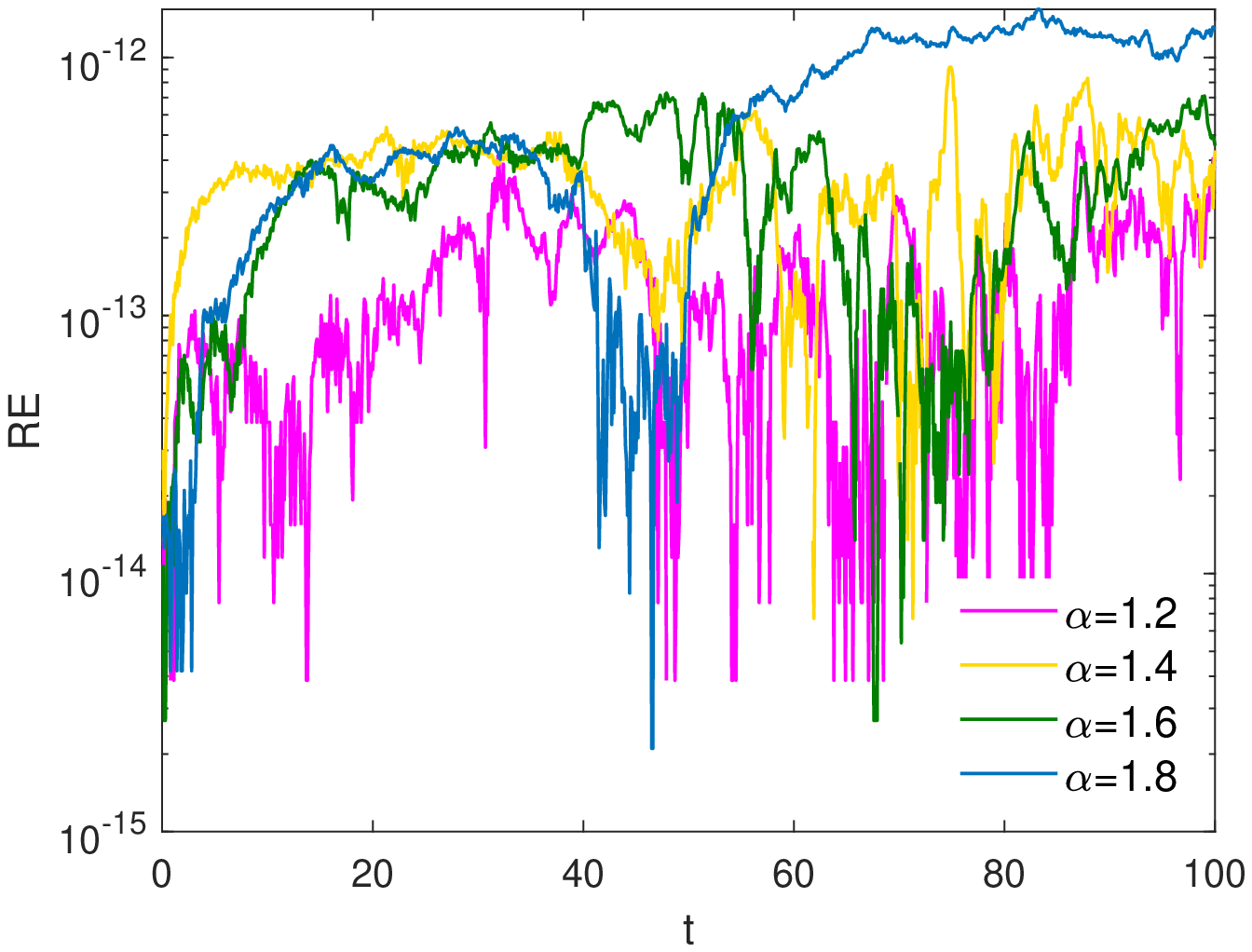}
\caption{\small{Time evolution of the relative mass and energy deviations of CN--SGM for Example \ref{ex2} with $\tau=1/50$ and $N=150$ for $\kappa_2=0.01$.}}\label{numer5}
\end{figure}

\begin{figure}[!htb]
\centering
\includegraphics[width=0.32\linewidth]{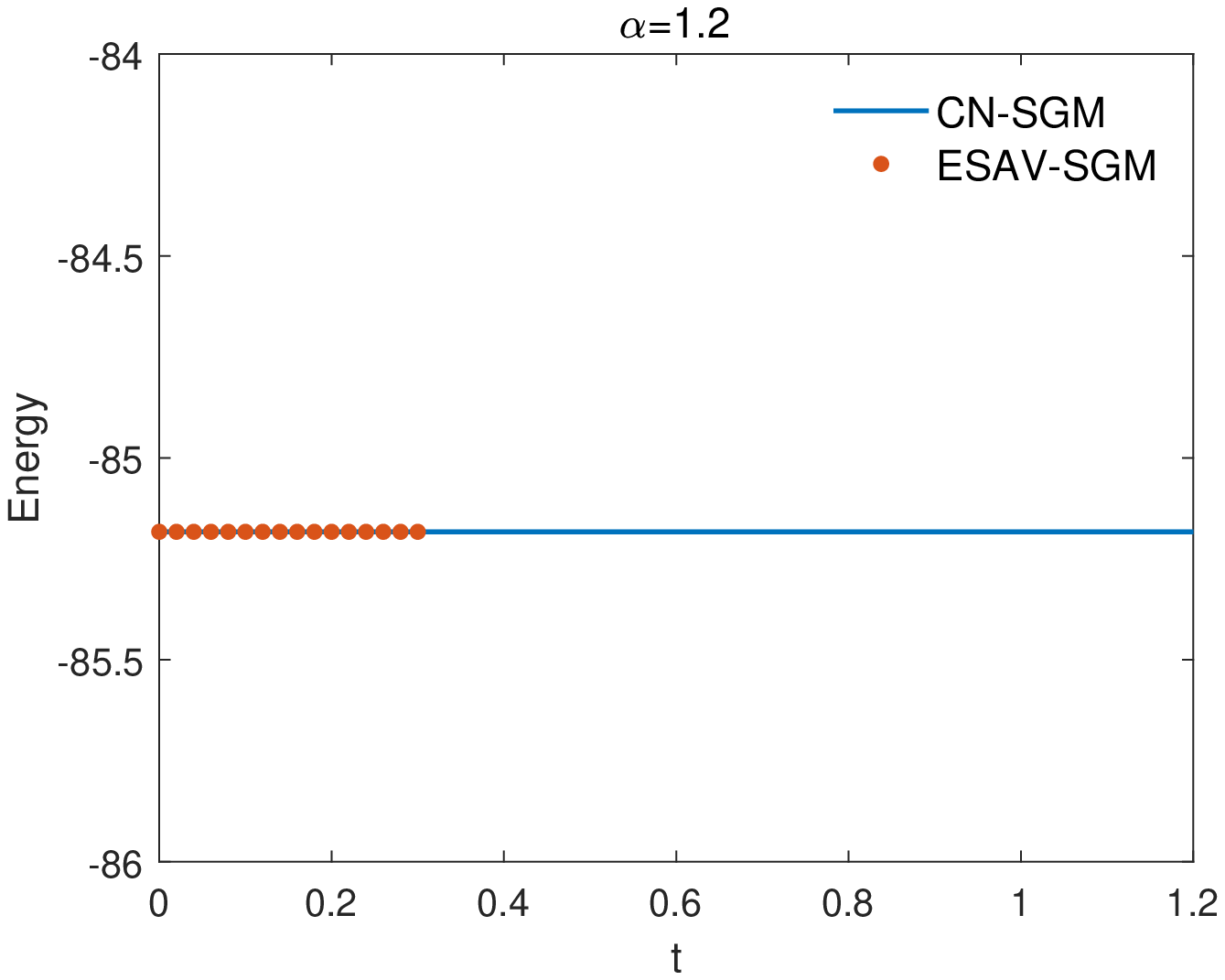}
\includegraphics[width=0.32\linewidth]{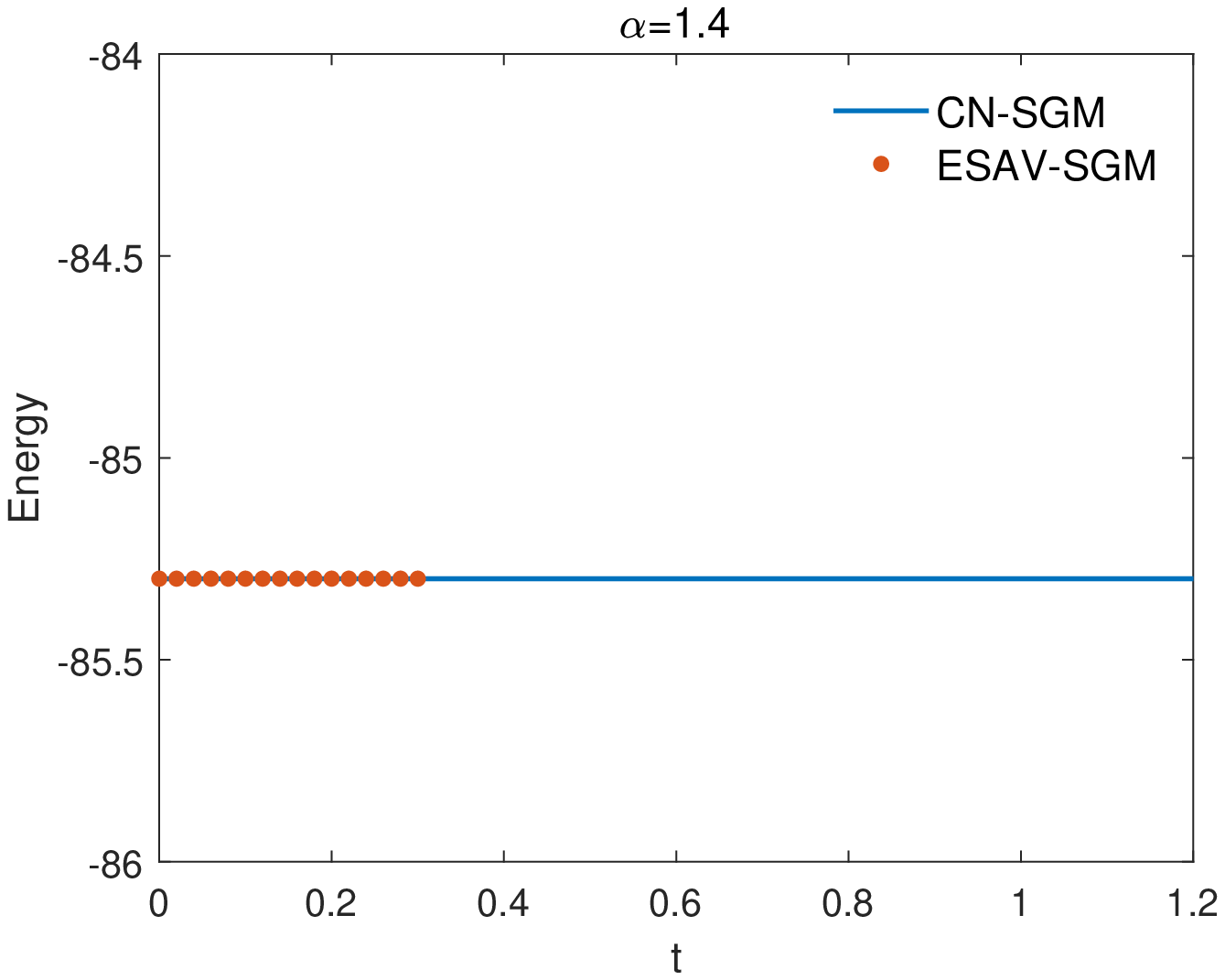}
\includegraphics[width=0.32\linewidth]{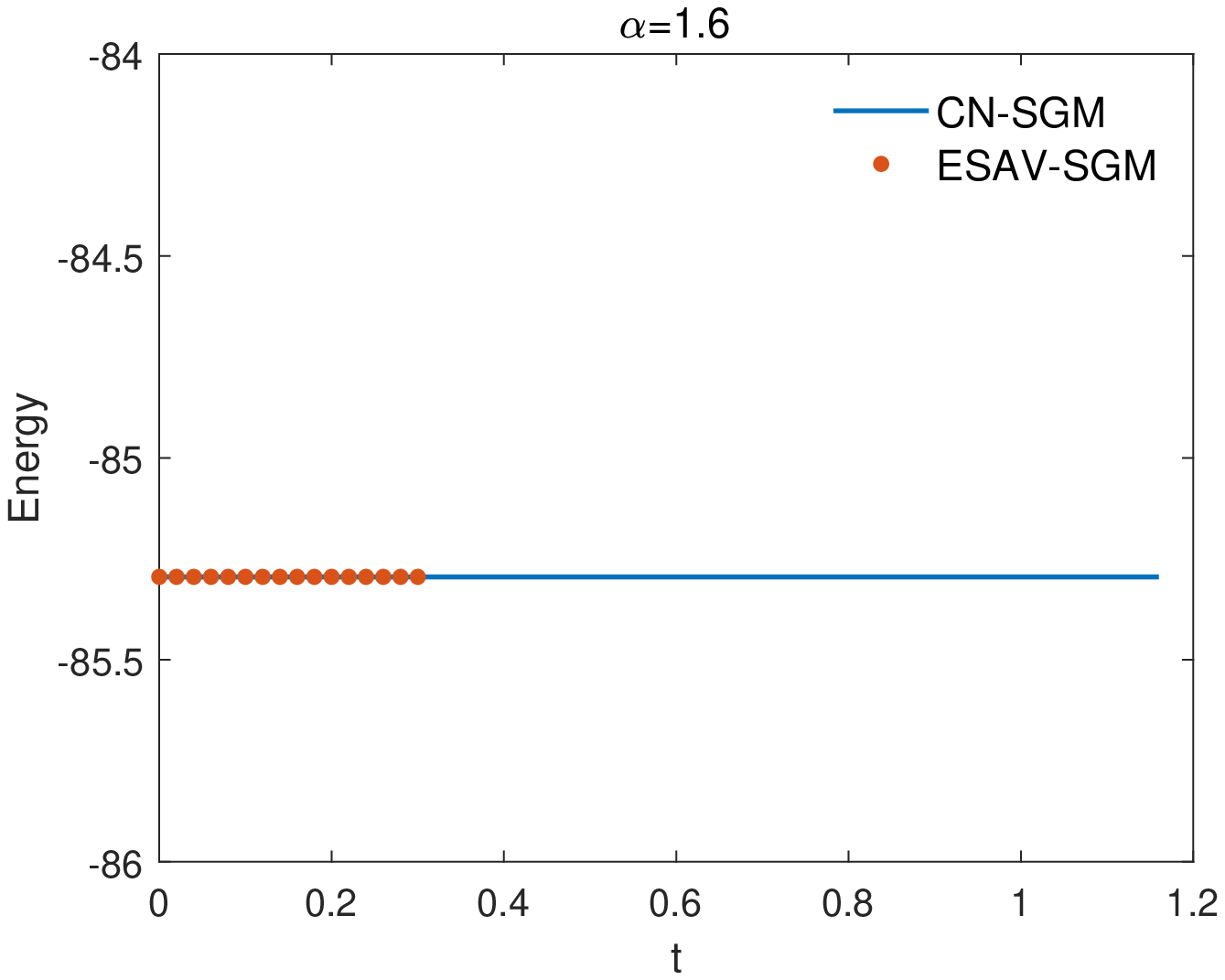}
\caption{\small{Time evolution of the mass and energy of CN--SGM and ESAV--SGM for Example \ref{ex2} with $\tau=1/50$ and $N=150$ for $\kappa_2=1$ till the solutions blow up.}}\label{numer6}
\end{figure}

\section{Conclusions}\label{conclu}
In this paper, we focus on the spectral Galerkin methods for the coupled fractional nonlinear Klein-Gordon-Schr\"odinger equation. We have proposed two structure-preserving schemes, also analyzed the advantages and disadvantages of each other by numerical comparisons.  In theoretical aspects, the maximum-norm boundness of numerical solutions of the CN--SGM are proved, and the unique solvability is obtained with the help of Browder fixed point theorem and maximum-norm boundness of the numerical solutions. Moreover, the unconditional convergence are analyzed without any restriction on grid ratio. Numerical results are reported to verify that the proposed numerical schemes have excellent capabilities in numerical accuracy and long-time conservations. In future work, the convergence proof of the ESAV-SGM will be carried out, also, our schemes and theoretical results will be extended to the high-dimensional FNKGS equation.

\section{Acknowledgements}
This work is partially supported by the National Natural Science Foundation of China (Grant Nos.~12171245, 11971242) and the Research Start-up Foundation of Jiangxi Normal University.

\section{Declarations}
\textbf{Conflict of interest} The authors declare that they have no conflict of interest.

\appendix
\section{Appendix: Proof of Theorem \ref{unique}}\label{app}

The proof of Theorem \ref{unique} is divided into two parts, including the existence and uniqueness.

(I)~\textbf{Existence}:
\begin{proof}
It is worth noting that $U_N^{n+1}=2U_N^{n+1/2}-U_N^{n}$ and $\delta_t\Phi_N^{n+1/2}=\Psi_N^{n+1/2}$. We reformulate the spectral Galerkin scheme \eqref{CNGLS1-1}--\eqref{CNGLS1-3} into the following form
\begin{align}
& (U_N^{n+1/2}-U_N^{n},w_N)+\frac{\lambda\tau}{4}\textrm{i}\mathcal{B}(U_N^{n+1/2},w_N)-\frac{\kappa_1\tau}{2}\textrm{i} (U_N^{n+1/2}\Phi_N^{n+1/2},w_N)\nonumber\\
&\qquad\qquad\qquad\qquad-\frac{\kappa_2\tau}{2} \textrm{i}\Big((|2U_N^{n+1/2}-U_N^{n}|^2+|U_N^{n}|^2)U_N^{n+1/2}\Phi_N^{n+1/2},w_N\Big)=0,\label{app0-1}\\
&(2\Phi_N^{n+1/2}-2\Phi_N^{n}-\tau\Psi_N^{n},w_N)+\frac{\gamma\tau^2}{2}\mathcal{B}(\Phi_N^{n+1/2},w_N)+\frac{\eta^2\tau^2}{2} (\Phi_N^{n+1/2},w_N)\nonumber\\
&\qquad-\frac{\kappa_1\tau^2}{4} \Big((|2U_N^{n+1/2}-U_N^{n}|^2+|U_N^{n}|^2),w_N\Big)-\frac{\kappa_2\tau^2}{4} \Big((|2U_N^{n+1/2}-U_N^{n}|^4+|U_N^{n}|^4),w_N\Big)=0.\label{app0-2}
\end{align}

Now, we carry out proving the existence of $U_N^{n+1/2}$ and $\Phi_N^{n+1/2}$. For convenience, we define the map $\textbf{s}=(s_1,s_2)$, $\mathscr{F}=(\mathscr{F}_1,\mathscr{F}_2):$ $(X_N^0(\Omega),X_N^0(\Omega))\rightarrow (X_N^0(\Omega),X_N^0(\Omega))$, such that
\begin{align}\label{app1}
(\mathscr{F}_1(\textbf{s}),w_N)=&(s_1-U_N^n,w_N)+\frac{\lambda\tau}{4}\textrm{i}\mathcal{B}( s_1,w_N )-\frac{\kappa_1\tau}{2}\textrm{i}(s_1s_2,w_N)\\
&-\frac{\kappa_2\tau}{2}\textrm{i}\Big( ( |2s_1-U_N^n|^2+|U_N^n|^2 )s_1s_2,w_N\Big),\quad \forall w_N\in X_N^0(\Omega)\nonumber
\end{align}
and
\begin{align}\label{app2}
(\mathscr{F}_2(\textbf{s}),w_N)=&(2s_2-2\Phi_N^{n}-\tau\Psi_N^{n},w_N)+\frac{\gamma\tau^2}{2}\mathcal{B}( s_2,w_N )+\frac{\eta^2\tau^2}{2}(
s_2,w_N)\\
&-\frac{\kappa_1\tau^2}{4} \Big(|2s_1-U_N^n|^2+|U_N^{n}|^2,w_N \Big)-\frac{\kappa_2\tau^2}{4} \Big(|2s_1-U_N^n|^4+|U_N^{n}|^4,w_N \Big),\quad \forall w_N\in X_N^0(\Omega).\nonumber
\end{align}

Choosing $w_N=s_1$ in \eqref{app1} and taking the real part, from the Young's inequality, we derive
\begin{align*}
\textrm{Re}(\mathscr{F}_1(\textbf{s}),s_1)=&\textrm{Re}(s_1-U_N^n,s_1)\}=\|s_1\|^2-\textrm{Re}( U_N^n,s_1 )\geq \frac12\|s_1\|^2-\frac12  \|U_N^n\|^2\geq \frac12\|s_1\|^2-\mathcal{C}.
\end{align*}
Setting $w_N=s_2$ in \eqref{app2}, by utilizing the Cauchy-Schwarz inequality, we obtain
\begin{align}\label{app4}
&(\mathscr{F}_2(\textbf{s}),s_2)\geq(2s_2-2\Phi_N^{n}-\tau\Psi_N^{n},s_2)-\frac{\kappa_1\tau^2}{4} \Big(|2s_1-U_N^n|^2+|U_N^{n}|^2,s_2 \Big)-\frac{\kappa_2\tau^2}{4} \Big(|2s_1-U_N^n|^4+|U_N^{n}|^4,s_2 \Big)\nonumber\\
&=2\Big\|s_2\Big\|^2-(2\Phi_N^{n}+\tau\Psi_N^{n},s_2)-\frac{\kappa_1\tau^2}{4} \Big(|2s_1-U_N^n|^2+|U_N^{n}|^2,s_2 \Big)-\frac{\kappa_2\tau^2}{4} \Big(|2s_1-U_N^n|^4+|U_N^{n}|^4,s_2 \Big).
\end{align}
By employing the boundness of the numerical solutions (see Theorem \ref{BOUND}) and the Young's inequality, for sufficient small $\tau$, we conclude
\begin{align}\label{app4-1}
(2\Phi_N^{n}+\tau\Psi_N^{n},s_2)\leq \Big\| 2\Phi_N^{n}+\tau\Psi_N^{n} \Big\| \Big\| s_2 \Big\|\leq  \frac12\Big\| 2\Phi_N^{n}+\tau\Psi_N^{n} \Big\|^2 +\frac12 \Big\| s_2 \Big\|^2,
\end{align}
\begin{align}\label{app4-2}
\frac{\kappa_1\tau^2}{4}\Big(|2s_1-U_N^n|^2+|U_N^{n}|^2,s_2 \Big)&\leq \frac{\kappa_1^2\tau^4}{32} \Big\| |2s_1-U_N^n|^2+|U_N^{n}|^2 \Big\|^2 + \frac{1}{2} \Big\| s_2 \Big\|^2\leq \mathcal{C}+\frac{1}{2} \Big\| s_2 \Big\|^2
\end{align}
and
\begin{align}\label{app4-3}
\frac{\kappa_2\tau^2}{2}\Big(|2s_1-U_N^n|^4+|U_N^{n}|^4,s_2 \Big)\leq \mathcal{C}+\frac{1}{2} \Big\| s_2 \Big\|^2.
\end{align}
Substituting \eqref{app4-1}--\eqref{app4-3} into \eqref{app4}, we arrive at
\begin{align*}
(\mathscr{F}_2(\textbf{s}),s_2)\geq \frac{1}{2} \Big\| s_2 \Big\|^2-\frac{1}{2}\Big\| 2\Phi_N^{n}+\tau\Psi_N^{n} \Big\|^2-\mathcal{C}.
\end{align*}
Therefore, we have
\begin{align*}
\textrm{Re} (\mathscr{F}(\textbf{s}),s_1)&=\textrm{Re} (\mathscr{F}_1(\textbf{s}),s_2)+\textrm{Re} (\mathscr{F}_2(\textbf{s}),s)\nonumber\\
&\geq \frac12 \Big\| \textbf{s} \Big\|^2-\frac12\Big\| 2\Phi_N^{n}+\tau\Psi_N^{n} \Big\|^2-\mathcal{C}
=\frac12\Big(  \Big\| \textbf{s} \Big\|^2 - \Big\| \Big(  2\Phi_N^{n}+\tau\Psi_N^{n},  \sqrt{2\mathcal{C}}\Big) \Big\|^2  \Big).
\end{align*}
Taking $\delta=\Big\| \Big(  2\Phi_N^{n}+\tau\Psi_N^{n},  \sqrt{2\mathcal{C}}\Big) \Big\|$, which satisfies the condition of Lemma \ref{Browder}, we derive
\begin{align*}
\text{Re}( \mathscr{F}(\textbf{s}),\textbf{s}) \geq 0,\quad \forall \textbf{s}:~~\|\textbf{s}\|=\delta.
\end{align*}
This proves the existence of the numerical solution of \eqref{CNGLS1-1}--\eqref{CNGLS1-3}.
\end{proof}

Next, we prove the uniqueness of the numerical solution.

(II)~\textbf{Uniqueness}:
\begin{proof}
We prove the theorem by introduction. It is obvious to find from \eqref{initialvalue} that the numerical solution $(U_N^0,\Phi^0_N,\Psi_N^0)\in (X_N^0(\Omega),X_N^0(\Omega),X_N^0(\Omega))$ exists and is unique. Assume that $(U_N^n,\Phi^n_N,\Psi_N^n)$ is the unique solution of \eqref{CNGLS1-1}--\eqref{CNGLS1-3} for $n=0,1,\cdots,N_t-1$. Next, we prove the uniqueness of the solution $(U_N^{n+1/2},\Phi^{n+1/2}_N)$. Assume there are two solutions $X^{n+1/2}=(X_1^{n+1/2},X_2^{n+1/2})$ and $Y^{n+1/2}=(Y_1^{n+1/2},Y_2^{n+1/2})$ for scheme \eqref{CNGLS1-1}--\eqref{CNGLS1-3}. Then $X_1^{n+1/2}-Y_1^{n+1/2}$ and $X_2^{n+1/2}-Y_2^{n+1/2}$ satisfy \eqref{app1} and \eqref{app2} as follows
\begin{align*}
(\mathscr{F}_1({X}^{n+1/2})-\mathscr{F}_1({Y}^{n+1/2}),X_1^{n+1/2}-Y_1^{n+1/2})=0,\\
(\mathscr{F}_2({X}^{n+1/2})-\mathscr{F}_2({Y}^{n+1/2}),X_2^{n+1/2}-Y_2^{n+1/2})=0.
\end{align*}
By the definition of $\mathscr{F}_1$, we have
\begin{align}\label{unique1}
\Big\| X_1^{n+1/2}-Y_1^{n+1/2} \Big\|^2&+\frac{\lambda\tau}{4} \textrm{i}\mathcal{B}(X_1^{n+1/2}-Y_1^{n+1/2},X_1^{n+1/2}-Y_1^{n+1/2})\nonumber\\
&-\frac{\kappa_1\tau}{2}\textrm{i}(X_1^{n+1/2}X_2^{n+1/2}
-Y_1^{n+1/2}Y_2^{n+1/2},X_1^{n+1/2}-Y_1^{n+1/2})\\
&-\frac{\kappa_2\tau}{2}\textrm{i}\Big( ( |2X_1^{n+1/2}-U_N^n|^2+|U_N^n|^2 )X_1^{n+1/2}X_2^{n+1/2}\nonumber\\
&-( |2Y_1^{n+1/2}-U_N^n|^2+|U_N^n|^2 )Y_1^{n+1/2}Y_2^{n+1/2},X_1^{n+1/2}-Y_1^{n+1/2}\Big)=0.\nonumber
\end{align}
Taking the real part of \eqref{unique1}, by admitting the boundness of the numerical solutions (see Theorem \ref{BOUND}), Lemma \ref{nonlinear} and the Cauchy-Schwarz inequality, we obtain
\begin{align}\label{unique2}
\Big\| X_1^{n+1/2}-Y_1^{n+1/2} \Big\|^2&=-\frac{\kappa_1\tau}{2}\textrm{Im}(X_1^{n+1/2}X_2^{n+1/2}-Y_1^{n+1/2}Y_2^{n+1/2},X_1^{n+1/2}-Y_1^{n+1/2})\nonumber\\
&\quad+\frac{\kappa_2\tau}{2}\textrm{Im}\Big( ( |2X_1^{n+1/2}-U_N^n|^2+|U_N^n|^2 )X_1^{n+1/2}X_2^{n+1/2}\\
&\quad-( |2Y_1^{n+1/2}-U_N^n|^2+|U_N^n|^2 )Y_1^{n+1/2}Y_2^{n+1/2},X_1^{n+1/2}-Y_1^{n+1/2}\Big)\nonumber\\
&\leq \mathcal{C}\tau\Big(\Big\| X_1^{n+1/2}-Y_1^{n+1/2} \Big\|^2+\Big\| X_1^{n+1/2}-Y_1^{n+1/2} \Big\|^2 \Big).\nonumber
\end{align}
Taking into account of the definition of $\mathscr{F}_2$, we get
\begin{align*}
&\Big\| X_2^{n+1/2}-Y_2^{n+1/2} \Big\|^2+\frac{\gamma\tau^2}{4}\mathcal{B}(X_2^{n+1/2}-Y_2^{n+1/2},X_2^{n+1/2}-Y_2^{n+1/2})\nonumber\\
&+\frac{\eta^2\tau^2}{4} (X_2^{n+1/2}-Y_2^{n+1/2},X_2^{n+1/2}-Y_2^{n+1/2})\\
&\qquad-\frac{\kappa_1\tau^2}{8} \Big((|2X_2^{n+1/2}-U_N^{n}|^2)-(|2Y_2^{n+1/2}-U_N^{n}|^2),X_2^{n+1/2}-Y_2^{n+1/2}\Big)\nonumber\\
&-\frac{\kappa_2\tau^2}{8} \Big((|2X_2^{n+1/2}-U_N^{n}|^4)-(|2Y_2^{n+1/2}-U_N^{n}|^4),X_2^{n+1/2}-Y_2^{n+1/2}\Big)=0.\nonumber
\end{align*}
By employing Theorem \ref{BOUND}, Lemma \ref{nonlinear} and the Cauchy-Schwarz inequality, we deduce
\begin{align}\label{unique4}
\Big\| X_2^{n+1/2}-Y_2^{n+1/2} \Big\|^2&\leq\Big\| X_2^{n+1/2}-Y_2^{n+1/2} \Big\|^2+\frac{\gamma\tau^2}{4}\Big|X_2^{n+1/2}-Y_2^{n+1/2}\Big|^2_{\alpha/2}+\frac{\eta^2\tau^2}{4} \Big\|X_2^{n+1/2}-Y_2^{n+1/2}\Big\|^2\nonumber\\
&=\frac{\kappa_1\tau^2}{8} \Big((|2X_2^{n+1/2}-U_N^{n}|^2)-(|2Y_2^{n+1/2}-U_N^{n}|^2),X_2^{n+1/2}-Y_2^{n+1/2}\Big)\nonumber\\
&\qquad+\frac{\kappa_2\tau^2}{8} \Big((|2X_2^{n+1/2}-U_N^{n}|^4)-(|2Y_2^{n+1/2}-U_N^{n}|^4),X_2^{n+1/2}-Y_2^{n+1/2}\Big)\\
&\leq \mathcal{C}\tau^2 \Big\| X_2^{n+1/2}-Y_2^{n+1/2} \Big\|^2.\nonumber
\end{align}
Summing up \eqref{unique2} and \eqref{unique4}, for a sufficient small $\tau$, we arrive at
\begin{align*}
\Big\| X^{n+1/2}-Y^{n+1/2} \Big\|^2&=\Big\| X_1^{n+1/2}-Y_1^{n+1/2} \Big\|^2+\Big\| X_2^{n+1/2}-Y_2^{n+1/2} \Big\|^2\nonumber\\
&\leq \mathcal{C}\tau\Big\| X_1^{n+1/2}-Y_1^{n+1/2} \Big\|^2+\mathcal{C}\tau \Big\| X_2^{n+1/2}-Y_2^{n+1/2} \Big\|^2+\mathcal{C}\tau^2 \Big\| X_2^{n+1/2}-Y_2^{n+1/2} \Big\|^2\nonumber\\
&\leq \mathcal{C}\tau\Big\| X_1^{n+1/2}-Y_1^{n+1/2} \Big\|^2+\mathcal{C}\tau \Big\| X_2^{n+1/2}-Y_2^{n+1/2} \Big\|^2\\
&= \mathcal{C}\tau\Big\| X^{n+1/2}-Y^{n+1/2} \Big\|^2.\nonumber
\end{align*}
By assuming $\mathcal{C}\tau<1$, it leads to $\Big\| X^{n+1/2}-Y^{n+1/2} \Big\|=0$, which implies $X_1^{n+1/2}=Y_1^{n+1/2}$ and $X_2^{n+1/2}=Y_2^{n+1/2}$. This ends the proof of the  uniqueness.
\end{proof}

\end{document}